\documentclass[10pt, reqno]{amsart}
\usepackage{amsthm, amsmath, amsfonts, amssymb, enumerate, textcmds, tensor, enumitem, mathdots}
\usepackage{geometry}
\usepackage{pst-node}
\usepackage{tikz-cd} 
\usepackage{appendix}
\usepackage{graphicx,tikz}
\usepackage{nicefrac}
\usepackage{faktor}
\usepackage[colorlinks=true]{hyperref}

\usepackage[style=numeric-comp, url=false,doi=false]{biblatex}
\AtEveryBibitem{
 \clearlist{address}
 \clearfield{date}
 \clearfield{isbn}
 \clearfield{issn}
 \clearlist{location}
 \clearfield{month}
 \clearfield{series}
 \clearfield{pages}

 \ifentrytype{book}{}{
  \clearlist{publisher}
  \clearname{editor}
 }
}

\hypersetup{
	colorlinks,
	citecolor=black,
	filecolor=black,
	linkcolor=black,
	urlcolor=black
}
\usepackage{cleveref}

\usepackage{physics}
\usepackage{enumerate}
\usepackage{pinlabel}
\usepackage[colorinlistoftodos]{todonotes}
\newcommand{\R}{\mathbb R}

\newcommand{\p}{\mathfrak{p}}
\newcommand{\g}{\mathfrak{g}}
\newcommand{\kk}{\mathfrak{k}}

\newtheorem{thm}{Theorem}[section]
\newtheorem{lem}[thm]{Lemma}

\newtheorem{prop}[thm]{Proposition}

\newtheorem{cor}[thm]{Corollary}

\newtheorem*{thma}{Theorem A}

\newtheorem*{thma'}{Theorem A'}
\newtheorem*{thmb}{Theorem B}
\newtheorem*{thmb'}{Theorem B'}
\newtheorem*{thmc}{Theorem C}
\newtheorem*{thmd}{Theorem D}

\DeclareMathOperator{\C}{\mathbb{C}}

\DeclareMathOperator{\Diff}{\mathrm{Diff}}
\DeclareMathOperator{\CP}{\mathbb{C}\mathbb{P}}

\DeclareMathOperator{\PSL}{\mathrm{PSL}}

\DeclareMathOperator{\inners}{\langle \cdot, \cdot \rangle}

\newcommand{\ccpair}{(c_1, \overline{c_2})}
\DeclareMathOperator{\fm}{\vb*{\overline {f_{--}}}}
\DeclareMathOperator{\f+}{\vb*{f_+}}
\newcommand{\hpair}{\vb*h{(c_1, \overline{c_2})}}

\DeclareMathOperator{\eps}{\varepsilon}

\newcommand{\redd}[1]{\textcolor{red}{#1}}

\DeclareMathOperator{\CAS}{\mathrm {SOL}}

\usepackage{physics}
\usepackage{soul}

\usepackage{scalerel}

\usepackage{stmaryrd}
\usepackage{upgreek}

\newcommand{\CSCS}{\mathcal C(S)\times \mathcal C(\overline S)}
\newcommand{\TSTS}{\mathcal T(S)\times \mathcal T(\overline S)}

\newcommand{\GC}{G^{\mathbb C}}

\newcommand{\Tau}{\mathscr T}

\usepackage{mathrsfs}
\usepackage{soul}

\theoremstyle{definition}

\newtheorem{defn}[thm]{Definition}
\newtheorem{remark}[thm]{Remark}
\newtheorem{ass}[thm]{Assumption}

\makeatletter
\newcommand\xleftrightarrow[2][]{%
  \ext@arrow 9999{\longleftrightarrowfill@}{#1}{#2}}
\newcommand\longleftrightarrowfill@{%
  \arrowfill@\leftarrow\relbar\rightarrow}
\makeatother

\begin{document}
\title[complex harmonic maps and rank 2 higher teichm{\"u}ller theory]{
 complex harmonic maps and rank 2 higher teichm{\"u}ller theory}

\author[Christian El Emam]{Christian El Emam}
\address{Christian El Emam: University of Torino, Dipartimento di Matematica ``Giuseppe Peano", Via Carlo Alberto, 10, 10123 Torino, Italy.} \email{christian.elemam@unito.it}

\author[Nathaniel Sagman]{Nathaniel Sagman}
\address{Nathaniel Sagman: University of North Carolina at Chapel Hill, 
3250 Phillips Hall $\#3$, Chapel Hill, NC 27599, United States.} \email{nsagman@unc.edu}

\begin{abstract}
We initiate and develop the theory of complex harmonic maps to holomorphic Riemannian symmetric spaces, which we make use of to study complex analytic aspects of higher Teichm{\"u}ller theory, with a focus on rank $2$ Hitchin components.

Complex harmonic maps lead to various generalizations of objects from the theory of G-Higgs bundles; for instance, the Hitchin fibration, cyclic G-Higgs bundles, and the affine Toda equations. Beyond such generalizations, we also find a relation between complex harmonic maps and opers.

Within the realm of higher Teichm{\"u}ller theory, for any rank $2$ Hitchin component, we prove a Bers-type theorem, which extends and improves our previous work on $\mathrm{SL}(3,\R)$, and we prove that Goldman's symplectic form is compatible with Labourie's complex structure, so that the two determine a mapping class group invariant pseudo-K{\"a}hler structure. We obtain partial generalizations in higher rank.  
\end{abstract}
\maketitle
\tableofcontents

\section{Introduction}
Let $S$ be a closed oriented surface of genus $\mathrm g\geq 2,$ and let $\mathcal{T}(S)$ be the Teichm{\"u}ller space of marked Riemann surface structures on $S.$ Given a non-compact semisimple Lie group $G$, a higher Teichm{\"u}ller space for $S$ and $G$ is a connected component of the
character variety $\chi(\pi_1(S),G)$ of conjugacy classes of representations from $\pi_1(S)$ to $G$ that consists entirely of classes of discrete and faithful representations. The definition
is motivated by the fact that $\mathcal{T}(S)$ identifies with such a component for $G = \mathrm{PSL}(2, \R)$. The most well-known higher Teichm{\"u}ller spaces are the Hitchin components, defined for split real simple Lie groups of adjoint type and originally studied by Hitchin in \cite{Hi}, which all contain $\mathcal{T}(S)$. Since \cite{Hi}, higher Teichm{\"u}ller spaces have been shown to possess interesting geometric, dynamical, and analytic properties, some of which generalize Teichm{\"u}ller space, and others that are new and surprising (see \cite{Wie} for a survey).

The field of higher Teichm{\"u}ller theory has developed alongside the theories of harmonic maps and minimal immersions to symmetric spaces of non-compact type, Higgs bundles, and the non-abelian Hodge correspondence. To give some idea, Donaldson and Corlette proved that, given a Riemann surface structure on $S$ and an irreducible representation $\rho:\pi_1(S)\to G$, there exists a unique $\rho$-equivariant harmonic map from the universal cover $\widetilde{S}$ to the Riemannian symmetric space of $G$. One can then use that harmonic map to glean information about the representation $\rho$. Moreover, the harmonic map determines a Higgs bundle, which is a holomorphic object that allows one to see the representation and the harmonic map from another viewpoint. One example of these objects in action is in Hitchin's work on the Hitchin components in \cite{Hi}.

In this paper, we take up two objectives. Our first one is to initiate and develop the theory of complex harmonic maps to holomorphic Riemannian manifolds, which extends the ordinary theory of harmonic maps. We show that when the target is a holomorphic Riemannian symmetric space (see Section \ref{section: holo Riem mflds}), our complex harmonic maps give rise to new objects that resemble Higgs bundles, and which we call complex harmonic $G$-bundles. We provide a way of constructing complex harmonic maps by solving systems of complex elliptic equations (Theorem C), and we find an unexpected connection to the theory of opers (Theorem D).

Our second objective is to use complex harmonic maps to study complex analytic aspects of higher Teichm{\"u}ller theory. Our main focus is on the Hitchin components of rank $2.$ As is well known, $\mathcal{T}(S)$ carries the structure of a contractible complex manifold, biholomorphic to a bounded domain in $\C^{3g-3},$ and it also has a K{\"a}hler metric called the Weil-Petersson metric. For every split real simple Lie group $G$ of adjoint type and of rank $2$, Labourie used conformal harmonic maps (equivalently, minimal immersions) to the Riemannian symmetric space of $G$ to give a real analytic and mapping class group equivariant parametrization of the Hitchin component of $G$ as a holomorphic vector bundle $\mathcal{M}(S,G)$ over Teichm{\"u}ller space \cite{Lab2}. Precisely, the fiber of $\mathcal{M}(S,G)$ over a class of complex structures $[c]\in\mathcal{T}(S)$ is represented by the space of holomorphic differentials of degree $d_G$ on $(S,c)$, where $d_G$ is the Coxeter number of $G$ (the dimension of a fiber is $(2d_G-1)(g-1)$). For each group $G$, we denote the parametrization $$\mathcal{L}_G:\mathcal{M}(S,G)\to \mathrm{Hit}(S,G),$$ and we endow $\mathrm{Hit}(S,G)$ with the complex structure of $\mathcal{M}(S,G)$. The almost complex structures on $\mathrm{Hit}(S,G)$ and on $\mathcal M(S,G)$ are denoted by $\mathcal{J}_G$ and $\mathcal J'_G$ respectively.

The only split real semisimple Lie group of adjoint type and rank $2$ that's not simple is $G=\mathrm{PSL}(2,\R)^2$. For products of Fuchsian representations into $G$, there is a similar parametrization by the bundle of holomorphic quadratic differentials over $\mathcal{T}(S)$ (see Section \ref{sec: Labourie paper}). In this case, we define $d_G=2$ and keep the notations $\mathcal{L}_G,$ $\mathcal{M}(S,G)$, and $\mathcal{J}_G$. So as not to have unruly Theorem statements, we incorrectly refer to the parametrization as Labourie's. 

Our main results on $\mathrm{Hit}(S,G)$ are a generalization of Bers' Simultaneous Uniformization Theorem for this complex structure (Theorem B), and a proof that $\mathcal{J}_G$ is compatible with Goldman's symplectic form and that the two objects together determine a pseudo-K{\"a}hler structure on $\mathrm{Hit}(S,G)$ (Theorem A). Outside of the rank $2$ thread, we find partial generalizations for Hitchin components for Lie groups of higher rank (Theorems A' and B').

In the coming subsections, we state our results precisely. We first give the applications to rank $2$ higher Teichm{\"u}ller theory (Theorems A and B) in Sections \ref{sec: thma} and \ref{sec: Bers}, and then we discuss complex harmonic maps and our main results on this new theory (Theorems C and D) in Section \ref{sec: complex harmonic maps intro}. Since Theorems A and B vs. Theorems C and D might be of interest to different audiences, we have written Section \ref{sec: complex harmonic maps intro} to be mostly independent of Sections \ref{sec: thma} and \ref{sec: Bers}.

\subsection{Goldman's form and Labourie's complex structure}\label{sec: thma}

On the smooth part of any character variety as above, Goldman defined a symplectic form, often referred to as the \textbf{Goldman symplectic form} \cite{Golform} (see Section \ref{sec: G symplectic form}). This symplectic form is always invariant under the natural action of the mapping class group on the character variety, and, when restricted to the Hitchin components, its further restriction to Teichm{\"u}ller space is, up to normalization, the classical Weil-Petersson symplectic form. At this point, the theory around the Goldman symplectic form has been developed from many viewpoints.

Let $\omega_G$ be the Goldman symplectic form on $\mathrm{Hit}(S,G).$ For $G=\mathrm{PSL}(3,\R)$, Labourie's parametrization was known before \cite{Lab3} and due to Loftin \cite{Lof} and Labourie \cite{Lab2}. Since \cite{Lof} and \cite{Lab2}, just for $G=\mathrm{PSL}(3,\R)$, the question of the relationship between $\omega_G$ and $\mathcal{J}_{G}$, in particular whether they are compatible, has been raised consistently (see, for example, \cite{RunTambnonkahler}). Note that in \cite{Lab3}, Labourie inquires about the general rank $2$ case. Theorem A says that $\omega_G$ and $\mathcal{J}_G$ are indeed compatible in a certain sense.

In this paper, we say that an almost complex structure $\mathcal{J}_0$ and a 2-form $\omega_0$ on a manifold $M$ are \textbf{compatible} if $\omega_0(\mathcal J_0\cdot,\mathcal J_0\cdot)=\omega_0$. 
 If $\omega_0$ is non-degenerate, then $\omega_0(\cdot, \mathcal J_0\cdot)$ is a pseudo-Riemannian metric. If $\mathcal J_
 0$ is integrable and $\omega_0$ is a compatible symplectic form, we say that $(M,\omega_0,\mathcal J_0)$ is \textbf{pseudo-Kähler}. When $\omega_0(\cdot, \mathcal J_0\cdot)$ is positive definite, the manifold is K{\"a}hler. When $M$ is connected, the signature of $(M,\omega_0,\mathcal J_0)$ is defined to be the signature of $\omega_0(\cdot, \mathcal J_0\cdot)$.
\begin{thma}
\label{thm: Hit is Kähler}
    Let $G$ be a split real Lie group of adjoint type and of rank $2$. The space $\mathrm{Hit}(S,G)$, endowed with the Goldman symplectic form and Labourie's complex structure, is a pseudo-Kähler manifold of signature $(6\mathrm g-6,2(2d_G-1)(\mathrm g-1))$. 
\end{thma}
Observe that the mapping class group action preserves the pseudo-Kähler structure, since it preserves both the complex structure and the Goldman symplectic form. Teichm{\"u}ller space is totally geodesic (see Theorem \ref{rmk: Teich totally geodesic}).

For $G=\mathrm{PSL}(2,\R)^2$, Theorem A was already proved by Mazzoli-Seppi-Tamburelli in \cite{2021parahyperkahler}. For $G=\mathrm{PSL}(3,\R),$ in \cite{RunTambnonkahler}, Rungi and Tamburelli show that if one assumes the Goldman symplectic form is compatible with Labourie's complex structure, then $\omega_{G}(\cdot, \mathcal{J}_G\cdot)$ would have signature $\left(6\mathrm g-6, 10 \mathrm g -10 \right)$. So, for $G=\mathrm{PSL}(3,\R)$, our main contribution is the compatibility.

The mapping class group invariant pseudo-Riemannian metric $\omega_G(\cdot,\mathcal{J}_G\cdot)$ seems worthy of further investigation. We point out that a lot of effort has already been invested in constructing and studying metrics on rank $2$ Hitchin components using Labourie's parametrizations.
Just for $G=\mathrm{PSL}(3,\R)$, there are metrics due to Darvishzadeh-Goldman \cite{DG} (built out of $\omega_G$ and an almost complex structure different from $\mathcal J_G$), Li \cite{Ligold} (Riemannian), Kim-Zhang \cite{kim2017kahler} (K{\"a}hler for the complex structure on the dual bundle, see also \cite[Corollary 1.3.2]{Lab3}), and Rungi-Tamburelli \cite{RunTambnonkahler} (semi-pseudo-K{\"a}hler with respect to $\mathcal{J}_G$). As well, metrics have been defined in higher rank from different perspectives (for instance, see \cite{Bridgeman2013ThePM}). See also \cite{trautwein2018infinite} for symplectic structures on related bundles of holomorphic differentials.

\subsection{Rank $2$ Bers Theorems}\label{sec: Bers}
Embedding $\mathrm{PSL}(2,\R)$ inside its complexification $\mathrm{PSL}(2,\C),$ any Fuchsian representation produces a representation to $\mathrm{PSL}(2,\C)$ that preserves a round circle in $\mathbb{CP}^1.$ More generally, a representation to $\mathrm{PSL}(2,\C)$ is called quasi-Fuchsian if it preserves a Jordan curve in $\mathbb{CP}^1$. Quotients of $\mathbb{H}^3$ by quasi-Fuchsian representations are the well-known quasi-Fuchsian $3$-manifolds. The classical Bers' Simultaneous Uniformization Theorem (Theorem \ref{thm: Bers thm} below) gives a way to associate two oppositely oriented complex structures on $S$ to a quasi-Fuchsian representation. Moreover, if we denote by $\mathcal{QF}(S)\subset {\chi}(\pi_1(S),\mathrm{PSL}(2,\C))$ the open subset of quasi-Fuchsian representations, the construction descends to a biholomorphism $$\mathcal{T}(S)\times\mathcal{T}(\overline{S})\to \mathcal{QF}(S).$$ Above $\overline{S}$ is the oppositely oriented surface.
Note that an embedding $\mathrm{PSL}(2,\R)\to G$ complexifies to an embedding $\mathrm{PSL}(2,\C)\to G^{\C}$, and hence we can extend Bers' Simultaneous Uniformization to a map from $\mathcal{T}(S)\times \mathcal{T}(\overline{S})\to \chi^{\mathrm{an}}(\pi_1(S),G^{\C})$ that also extends the inclusion $\mathcal{T}(S)\to \mathrm{Hit}(S,G)$ (when we identify $\mathcal{T}(S)$ with the diagonal in $\mathcal{T}(S)\times \mathcal{T}(\overline{S})$).

Equipping the rank $2$ Hitchin components with Labourie's complex structures, we look for a Bers-type theorem. Denote by $\mathrm{Hit}(\overline S,G)$ the Hitchin component endowed with the opposite complex structure--the notation is chosen because this identifies with $\mathcal{M}(\overline{S},G).$ We identify the diagonal of $\mathrm{Hit}(S,G)\times \mathrm{Hit}(\overline{S},G)$ with 
$\mathrm{Hit}(S,G)$ via $(\rho, \rho)\mapsto\rho$. By analytically continuing power series, one can find neighbourhoods $V$ of the diagonal in $\mathrm{Hit}(S,G)\times \mathrm{Hit}(\overline{S},G)$ and $W$ of $\mathrm{Hit}(S,G)$ inside the analytic character variety ${\chi}^{\mathrm{an}}(\pi_1(S),G^{\C})$ on which the above identification extends uniquely to a biholomorphism $\mathcal B_G:V\to W.$ However, this description of $\mathcal{B}_G$ is rather formal and not immediately meaningful. We seek a more geometric understanding of this map, which will allow us to study it better, and to see how far we can extend $\mathcal{B}_G$ into $\mathcal{M}(S,G)\times \mathcal{M}(\overline{S},G)$. 

By introducing complex harmonic maps to holomorphic Riemannian symmetric spaces, we address all of these problems: we show that complex harmonic maps do indeed give a geometric realization of $\mathcal{B}_G,$ and that $\mathcal{B}_G$ extends farther than originally expected. Given a semisimple Lie group $G$ with maximal compact subgroup $K$, the quotient $G/K$ carries a $G$-invariant metric that makes it a Riemannian symmetric space. As we alluded to above, for $G$ of rank $2,$ Labourie's parametrization associated points in $\mathcal{M}(S,G)$ to holonomies of conformal harmonic maps to $G/K.$
If $G^{\C}$ and $K^{\C}$ are the complexifications of $G$ and $K$ respectively, then that same invariant metric on $G/K$ induces a holomorphic Riemannian metric on $G^{\C}/K^{\C}$ (see Section \ref{section: holo Riem mflds}), into which we can define a notion of harmonic maps (see Section \ref{sec: complex harmonic maps intro} below for more details).

\begin{thmb}
Let $G$ be a split real Lie group of adjoint type and of rank $2$. There exists a $\mathrm{MCG}(S)$-invariant connected subset $\Omega_{G}\subset \mathrm{Hit}(S,G)\times \mathrm{Hit}(\overline S,G)$ such that
\begin{itemize}
    \item $\Omega_{G}$ contains the diagonal, $ \mathrm{Hit}(S,G)\times \mathcal T(\overline S)$, and $\mathcal T(S)\times \mathrm{Hit}(\overline S,G)$,
    \item $\mathrm{int}(\Omega_{{G}})$ intersects $\mathrm{Hit}(S,G)\times \mathcal T(\overline S)$ and $\mathcal T(S)\times \mathrm{Hit}(\overline S,G)$ in connected and dense subsets of the latter two spaces,
\end{itemize}
and such that the identification $(\rho, \rho)\mapsto \rho$ extends uniquely to a $\mathrm{MCG}(S)$-equivariant continuous map $$\mathcal{B}_G: \Omega_{G}\to \chi^{\mathrm{an}}(\pi_1(S), G^{\C})$$ that is a local biholomorphism on $\mathrm{int}(\Omega_G)$. The map $\mathcal{B}_G$ associates points to holonomies of conformal complex harmonic maps to $G^{\C}/K^{\C}$.  

On $ \mathrm{Hit}(S,G)\times \mathcal T(\overline S)$ and $\mathcal T(S)\times \mathrm{Hit}(\overline S,G)$, ${\mathcal{B}}_G$ takes points to holonomies of $G^{\C}$-opers, and, on the subset $\mathcal{T}(S)\times \mathcal{T}(\overline{S})$, $\mathcal{B}_G$ agrees with Bers' simultaneous uniformization map.
\end{thmb}
For $G=\mathrm{PSL}(3,\R)$ and $\mathrm{PSL}(2,\R)^2,$ we can in fact extend further so that $\mathcal{B}_G$ is defined on neighbourhoods of all of $ \mathrm{Hit}(S,G)\times \mathcal T(\overline S)$ and $\mathcal T(S)\times \mathrm{Hit}(\overline S,G)$. It could be possible for the other Lie groups, but pursuing this would require a deeper investigation of certain systems of complex elliptic PDE's (see the equation (\ref{eq: bers Laplacian loc}) below), which would take us too far away from the aims of the paper. We do not expect that $\mathcal{B}_G$ extends globally in a holomorphic fashion (see Remark \ref{rem: failure} and \cite[Theorem C]{ElSa}).

We discuss $G^{\C}$-opers in Section \ref{sec: complex harmonic maps intro} below, giving more explanation related to $\textrm{Hit}(S,G)\times \mathcal{T}(\overline{S})$ and its partner. Perhaps most in line with the higher Teichm{\"u}ller theory perspective, a $G^{\C}$-oper comes with a holonomy representation from $\pi_1(S)\to G^{\C},$ together with an equivariant and holomorphic map to the complex flag variety $G^{\C}/B$ satisfying a non-degeneracy condition (see \cite{Sanders} and Section \ref{subsubsec: opers}). When $G^{\C}=\mathrm{PSL}(2,\C),$ this is the developing map of a $\mathbb{CP}^1$ structure, and $\mathrm{PSL}(2,\C)$-opers are in fact equivalent to $\mathbb{CP}^1$ structures. Since Bers' simultaneous uniformization procedure goes through constructing $\mathbb{CP}^1$-structures with quasi-Fuchsian holonomies, it feels quite natural that opers appear in the image of ${\mathcal{B}}_G.$

  Theorem B improves \cite[Theorem A]{ElSa}, where we used complex affine spheres to prove a version of Theorem B in the case $G=\textrm{PSL}(3,\R)$ (see also Remark \ref{rem: other immersions}). Note that we did not prove in \cite{ElSa} that $ \mathrm{Hit}(S,G)\times \mathcal T(\overline S)$ corresponds to opers, which are more naturally seen and studied from the standpoint of harmonic maps.
  
  The geometric interpretations that we give for ${\mathcal{B}}_G$ should allow one to probe the geometry of the image representations. On the Anosov property, see the end of Section \ref{sec: proving Theorem B}.

\subsection{Complex harmonic maps and $G^{\C}$-opers}\label{sec: complex harmonic maps intro}
Finally, we get to the engine behind Theorems A and B: complex harmonic maps. In this section, we work with a general non-compact surface $S$. Let $(G/K,\nu)$ be a Riemannian symmetric space. Equivariant harmonic maps to $G/K$ give rise to $G$-Higgs bundles. Given a $G$-Higgs bundle, if one can solve Hitchin's self-duality equations (see Section \ref{sec: ordinary harmonic maps}), then it comes from a harmonic map. 

We propose complex harmonic maps to the holomorphic Riemannian symmetric space $G^{\C}/K^{\C}$, with the holomorphic Riemannian metric induced by $\nu$, as a complexification of the theory of harmonic maps. 

To define harmonic maps from surfaces, one needs a reference complex structure. For our notion of ``complex harmonic," given in Section 4, one requires two oppositely oriented complex structures $c_1,\overline{c_2}$ on $S$. One hint as to why complex harmonic maps can be seen as a complexification is as follows. Ordinary harmonic maps, when satisfying non-degeneracy conditions, vary real analytically (see \cite{EL}). For complex harmonic maps, the defining equation depends holomorphically on its input data. Hence, with non-degeneracy conditions, complex harmonic maps should vary holomorphically.

As we've mentioned, complex harmonic maps to $G^{\C}/K^{\C}$ lead to a a generalization of $G$-Higgs bundles that we call complex harmonic $G$-bundles. Many aspects of the theory of $G$-Higgs bundles have some extension for complex harmonic $G$-bundles. We highlight here the Hitchin fibration and cyclic $G$-Higgs bundles. 
Assume now $G$ is split real, of adjoint type, and of rank $l$. Let $m_1<\dots<m_l$ be the exponents of $G$ ($m_l=d$ is the Coxeter number).
Given a Riemann surface $(S,c)$, the Hitchin base is $$H(c,G)=\oplus_{i=1}^{l}H^0(S,\mathcal{K}_c^{m_i}).$$ The classical Hitchin fibration associates a $G$-Higgs bundle to a point in $H(c,G).$ 
In \cite{Hi}, Hitchin produces a section of the Hitchin fibration from $H(c,G)$ to a component of the space of $G$-Higgs bundles for $(S,c)$, which under the non-abelian Hodge correspondence becomes $\mathrm{Hit}(S,G).$

A $G$-Higgs bundle is called cyclic if it carries a certain symmetry; such $G$-Higgs bundles have been the subject of intense study (see \cite{Baraglia2010CyclicHB}, \cite{C}, \cite{GPRi}, \cite{ST}, etc.). 
The cyclic $G$-Higgs bundles in the Hitchin section correspond to points $(0,\dots, q_l)\in H(c,G).$


In Section \ref{sec: ch G-bundles}, we define the bi-Hitchin fibration on complex harmonic $G$-bundles, which associates complex harmonic $G$-bundles to points in what we call the bi-Hitchin base, $$bH(c_1,\overline{c_2},G)=H(c_1,G)\oplus H(\overline{c_2},G).$$  {For each point in the bi-Hitchin base, we produce an equation (similar to Hitchin's self-duality equations) whose solutions yield complex harmonic $G$-bundles lying over that point (see Theorem \ref{thm: biHitchin}).} We then impose that  our bi-Hitchin basepoint is of the form $$((0,\dots,0,q_1),(0,\dots, 0,\overline{q_2}))\in bH(c_1,\overline{c_2},G).$$  
Theorem C says that, in this case, the equation from Theorem \ref{thm: biHitchin} simplifies drastically. We define the \textbf{complex affine Toda equations} for $(G,c_1,\overline{c_2},q_1,\overline{q_2})$ (Definition \ref{def: CAT}, Equation \eqref{eq: CAT}), an equation of the form, for a function $U=(U_{\alpha_1},\dots, U_{\alpha_l}),$
$$
\Delta_h \log U_\alpha=F_\alpha(z,U_{\alpha_1},\dots, U_{\alpha_l},h, q_1,\overline{q_2}),$$
where $h$ is the Bers metric of $(c_1,\overline{c_2})$, $\Delta_h$ is the Bers Laplacian (developed in \cite{Kim} and \cite{ElSa}), and $F_\alpha$ is non-linear and holomorphic in its arguments. The full definition of the system is technical, so we defer it to Section \ref{sec: proof of Theorem C}. We just clarify that a Bers Laplacian is a so-called complex elliptic operator, and locally takes the form
\begin{equation}\label{eq: bers Laplacian loc}
    \Delta_h = -\frac{4}{\lambda}\partial_{\overline{z}}(\partial_z-\overline{\mu}\partial_{\overline{z}}),
\end{equation}
where $\lambda$ is a function to $\C^*$, and $\mu$ is a function with $||\mu||_{L^\infty}<1$.
 \begin{thmc}
Let $G$ be a simple split real Lie group of adjoint type with Coxeter number $d$. Let $c_1$ and $\overline{c_2}$ be oppositely oriented complex structures on $S$ and $q_1\in H^0(S,\mathcal{K}_{c_1}^d), \overline{q_2}\in H^0(S,\mathcal{K}_{\overline{c_2}}^d).$ A solution to the complex affine Toda equations \eqref{eq: CAT} for the data $(G,c_1,\overline{c_2},q_1,\overline{q_2})$ produces an
equivariant complex harmonic map $$f: (\widetilde{S},c_1,\overline{c_2})\to G^{\C}/K^{\C}$$ whose bi-Hitchin basepoint is $((0,\dots, 0, q_1),(0,\dots, 0,\overline{q_2})).$ 
\end{thmc}

When $c_1=c_2$, $q_1=q_2$, and $U$ is real, the complex harmonic map is the composition of an ordinary map to $G/K$ with the inclusion $G/K\to G^{\C}/K^{\C}$. In this case, (\ref{eq: CAT}) becomes an ordinary affine Toda system and is a special case of the equation (1) from \cite{ST}; in a local coordinate, it reduces to Baraglia's version of Hitchin's equations from \cite{Baraglia2010CyclicHB}. 

In low rank, (\ref{eq: CAT}) generalizes well-known equations from geometry. For instance, when $G=\mathrm{PSL}(2,\R),$ one requires two quadratic differentials $q_1$ and $\overline{q_2}$ and the equation is
\begin{equation}\label{eq: rank 1 case}
   \Delta_h \log U = 2U-2\frac{q_1\overline{q_2}}{h^2}U^{-1}-2. 
\end{equation}
Taking $c_1=c_2$ and $q_1=q_2,$ (\ref{eq: rank 1 case}) returns a classical Bochner formula for harmonic maps. Similarly, for $G=\mathrm{PSL}(3,\R)$, Theorem C provides a generalization of the Tziteica equation for hyperbolic affine spheres (see \cite{Lof}), and for $G=\mathrm{PSp}(4,\R),$ a generalization of the structural equations for spacelike maximal surfaces in the pseudo-hyperbolic space $\mathbb{H}^{2,2}$ (see \cite{Nie}).

As alluded to in Theorem B, we find a connection between certain complex harmonic maps and known meaningful objects: $G^{\C}$-opers. 
The modern theory of $G^{\C}$-opers stems from Drinfeld-Sokolov \cite{DS} and the work of Beilinson-Drinfeld on the Geometric Langlands Conjecture \cite{BD}.
\begin{thmd}
Let $G$ be a split real simple Lie group of adjoint type. Let $c_1$ and $\overline{c_2}$ be oppositely oriented complex structures on $S$, and let $(q,0)\in bH(c_1,\overline{c_2},G).$ 
\begin{enumerate}
    \item There exists an equivariant complex harmonic map $f_{q}: (\widetilde{S},c_1,\overline{c_2})\to G^{\C}/K^{\C}$ with bi-Hitchin basepoint $(q,0)$.
    \item  The holonomy of $f_{q}$ is that of a $G^{\C}$-oper on $(S,c_1)$.
\end{enumerate}
    The analogous statement holds if we reverse the roles of $c_1$ and $\overline{c_2}$.
\end{thmd}
The oper in part (2) is written out explicitly in the proof. When $c_1=c_2,$ the answer is clean: Beilinson-Drinfeld parametrized $G^{\C}$-opers on $(S,c_1)$ by the Hitchin base $H(c_1,G),$ and the oper corresponds under their parametrization to $q.$ Calling back to Theorem B, $\mathcal{B}_G$ takes a point $([c_1,q_1],[\overline{c_2},0])\in \mathcal{M}(S,G)\times \mathcal{T}(\overline{S})$ to the holonomy of the oper from Theorem D associated with the data $c_1,\overline{c_2},$ and $(q_1,0).$


\begin{remark}
Holonomy aside, the bundles from Theorem D are worthy of further study--they point to a generalization of $G^{\C}$-opers where one involves two complex structures on a surface. 
\end{remark}

In \cite{Gai}, Gaiotto conjectured, motivated by physics, that the Hitchin component and the space of opers are holomorphically identified by taking a ``conformal limit." This conjecture was proved in \cite{DFKMMN} and \cite{collier2024conformallimitsparabolicslnchiggs} (see also \cite{CW}). In Section \ref{sec: proof of Theorem D}, we point out a connection between complex harmonic maps and Gaiotto's conjecture, which we're eager to explore further.

\begin{remark}\label{rem: explanation}
One could try to construct a moduli space of complex harmonic $G$-bundles (over a fixed pair $(c_1,\overline{c_2})$, or a joint moduli space where the complex structures can vary), which, at least near the ordinary Hitchin section, should be a complex manifold. With such a moduli space, we propose using Theorem \ref{thm: biHitchin} to holomorphically extend Hitchin's section (over a pair $(c,\overline{c})$ or on the joint moduli space). We would call this extension the bi-Hitchin section. This would all be quite technical, and although interesting, isn't one of our aims in this paper.
\end{remark}

\subsection{Partial generalizations for Theorems A and B}\label{sec: A' and B'}
Regarding higher rank analogues for Theorems A and B, note that the second author proved with Smillie in \cite{SS} that for all $G$ of rank at least $3,$ there exist Hitchin representations that admit multiple minimal surfaces. Hence, there is no known analogue of Labourie's parametrizations. However, in \cite{Lab3}, Labourie does have a partial result, which we build on here. 

Let $G$ be a general split real simple Lie group with Coxter number $d_G$. As in the rank $2$ case, let $\mathcal{M}(S,G)$ be the holomorphic vector bundle over $\mathcal{T}(S)$ whose fiber over $[c]\in\mathcal{T}(S)$ is represented by the space of holomorphic degree $d_G$ differentials on $(S,c)$. This bundle parametrizes minimal surfaces for Hitchin representations with cyclic $G$-Higgs bundles. Labourie proves in \cite{Lab3} that the holonomy map $\mathcal{L}_G:\mathcal{M}(S,G)\to \mathrm{Hit}(S,G)$ is an immersion. 

The Goldman form pulls back to $\mathcal{M}(S,G)$ via $\mathcal{L}_G$, but it might degenerate far away from the Fuchsian locus. Still, we can ask about the compatibility of $\mathcal L_G^*\omega_G$ with $\mathcal{J}'_{G}$. The form $\mathcal L_G^*\omega_G$ is non-degenerate on an open subset of $\mathcal{M}(S,G)$, and we set $\mathcal{M}_0(S,G)\subset \mathcal{M}(S,G)$ to be the connected component of that subset that contains the zero section.
\begin{thma'}
Let $G$ be a split real simple Lie group of adjoint type. On the space $\mathcal{M}(S,G)$, the pullback of Goldman's symplectic form and the almost complex structure $\mathcal{J}'_{G}$ are compatible. On $\mathcal{M}_0(S,G)$, they determine a pseudo-K{\"a}hler structure of signature $(6\mathrm g-6,2(2d_G-1)(\mathrm g-1))$.
\end{thma'}
We similarly generalize Theorem B to this context. The statement of Theorem B' is exactly that of Theorem B, except now $G$ is of higher rank, and rather than holomorphically extending the diagonal identification, we holomorphically extend Labourie's immersion. Similar to the rank $2$ situation, we identify $\mathcal{M}(S,G)$ with the diagonal inside $\mathcal{M}(S,G)\times \mathcal{M}(\overline{S},G).$ 
 \begin{thmb'}
Let $G$ be a split real simple Lie group of adjoint type. There exists a $\mathrm{MCG}(S)$-invariant connected subset $\Omega_G'\subset \mathcal{M}(S,G)\times \mathcal{M}(\overline{S},G)$ such that
\begin{itemize}
    \item $\Omega_G'$ contains the diagonal, $ \mathcal{M}(S,G)\times \mathcal T(\overline S)$, and $\mathcal T(S)\times \mathcal{M}(\overline{S},G)$,
    \item  $\mathrm{int}(\Omega'_{{G}})$ intersects $\mathcal{M}(S,G)\times \mathcal T(\overline S)$ and $\mathcal T(S)\times \mathcal{M}(\overline S,G)$ in connected and dense subsets of the latter two spaces, 
\end{itemize}
and on which Labourie's immersion extends uniquely to a $\mathrm{MCG}(S)$-equivariant continous map $$\mathcal{L}_G^{\C}: \Omega_G'\to \chi^{\mathrm{an}}(\pi_1(S), G^{\C})$$ that is a holomorphic immersion on $\mathrm{int}(\Omega_G').$ The map $\mathcal{L}_G^{\C}$ associates points to holonomies of conformal complex harmonic maps to $G^{\C}/K^{\C}$. 

On $ \mathcal{M}(S,G)\times \mathcal T(\overline S)$ and $\mathcal T(S)\times \mathcal{M}(\overline{S},G)$, ${\mathcal{L}}_G^{\C}$ takes points to holonomies of $G^{\C}$-opers, and on the subset $\mathcal{T}(S)\times \mathcal{T}(\overline{S})$, $\mathcal{L}_G^{\C}$ agrees with Bers' simultaneous uniformization map.
 \end{thmb'}

\begin{remark}
    By taking products, the results can easily be extended to semisimple Lie groups.
\end{remark}

Rather than working in rank $2,$ we prove Theorems A' and B'. Knowing that Labourie's maps are bijections in rank $2,$ Theorems A and B are more or less realized as special cases.

\subsection{Outline of paper and proofs}\label{sec: outline}
In Section 2 we give the necessary preliminaries on Lie groups, character varieties, Teichm{\"u}ller space, Hitchin components, Higgs bundles, and Labourie's parametrizations. In Section 3, we give preliminaries on the complex side of the theory: we introduce holomorphic Riemannian symmetric spaces, discuss a notions developed in \cite{BEE}, \cite{ElE}, and \cite{ElSa}, and we set up an analytic framework used for complex harmonic maps.

In Sections \ref{sec: def of ch maps} and \ref{sec: building ch G-bundles} we define complex harmonic maps and complex harmonic $G$-bundles and discuss basic properties. {In Section \ref{sec: bi-Hitchin} we prove Theorem \ref{thm: biHitchin} about constructing complex harmonic $G$-bundles}, and in Section \ref{sec: proof of Theorem C} we prove Theorem C. The strategy in Section \ref{sec: proof of Theorem C} is similar to that of Theorem A from \cite{ST}, which is about Hitchin's self-duality equations for cyclic $G$-Higgs bundles, although we have new complications and we don't have analogs for some features of the Higgs bundle theory used in \cite{ST}. Briefly, Theorem \ref{thm: biHitchin} reduces finding complex harmonic maps to solving an equation for an operator on a $G^{\C}$-bundle. In the setting of Theorem C, we simplify the situation by demanding that the operator respects a reduction of the structure group to a Cartan subgroup. After a careful computation, we arrive the complex affine Toda equations.  In the remainder of Section \ref{sec: ch G-bundles}, we find cases where we can solve the equations (\ref{eq: bi-Hitchin flatness}) and (\ref{eq: CAT}) and we prove the first part of Theorem D. 

In Section \ref{sec: opers} we discuss $G^{\C}$-opers and prove the second part of Theorem D. For $c_1=c_2,$ we see directly that the flat connection associated with the complex harmonic $G$-bundle is the same as one constructed by Beilinson-Drinfeld. For $c_1\neq c_2,$ it is a direct computation.

To prove Theorem B' (and Theorem B) in Section 6, we follow a strategy similar to the proof of Theorem A from \cite{ElSa}, but we work in a more general context and simplify certain aspects. The basic method from global analysis is to solve the equation (\ref{eq: CAT}) on certain loci and to prove that on these loci, the linearization of (\ref{eq: CAT}) is an isomorphism. To study the linearization on the real locus, we use the maximum principle for systems from \cite{Dai2018}, and for the loci $\mathcal{M}(S,G)\times \mathcal{T}(\overline{S}),$ we use the Analytic Fredholm Theorem. 

There is one major snag that makes the situation in Section 6 more complicated: the space $\mathcal{M}(S,G)\times \mathcal{M}(\overline S,G)$ is the quotient of an infinite dimensional Fréchet manifold under a difficult group action, namely, that of the group $\mathrm{Diff}_0(S)\times \mathrm{Diff}_0(S),$ where $\mathrm{Diff}_0(S)$ is the group of diffeomorphisms of $S$ isotopic to the identity. In order to deal with this action, we use complex Lie derivatives, as developed in \cite{ElSa}.

In Section \ref{sec: Goldman form thma A} we prove Theorem A' (and Theorem A).  Let $\omega_{G}^{\C}$ be the complex Goldman symplectic form on $\chi(\pi_1(S),G^{\C})$ (see Section \ref{sec: G symplectic form}  for the definition), which determines a holomorphic symplectic form $(\mathcal{L}_G^{\C})^*\omega_{G}^{\C}$ on $\mathrm{int}(\Omega_G'),$ where $\Omega_G'$ is the set from Theorem B'. Beginning with Theorem A, we prove the compatibility condition by stablishing that the intersections of $\Omega_G'$ with submanifolds of the form $\mathcal M(S,G)\times \{[\overline{c_2}, \overline{q_2}]\}$ and $\{[c_1, q_1]\}\times \mathcal M(\overline S,G)$ are $(\mathcal{L}_G^{\C})^*\omega_{G}^{\C}$-Lagrangian (Proposition \ref{prop:lagrangians}). By general principles in complex geometry, this property is equivalent to the compatibility of the (real) form $\mathcal{L}_G^*\omega_G$ (see Proposition \ref{prop:compatible}). Note that we only need to prove Proposition \ref{prop:lagrangians} in a neighbourhood of the diagonal. Thus, the only part that we need from Theorem B' is that the holomorphic map $\mathcal{L}_G^{\C}$ associates points to holonomies of complex harmonic maps (in particular, nothing about $\mathcal{M}(S,G)\times \mathcal{T}(\overline{S}),$ etc.). Proving Proposition \ref{prop:lagrangians} is more approachable than proving Theorem A' directly because we have explicit and very tractable descriptions of flat connections corresponding to holonomies of complex harmonic maps. To compute the signature of the resulting pseudo-K{\"a}hler metric, we again  making use of the explicit flat connections. 
\subsection{Acknowledgements}
We thank Brian Collier, Oscar Garc{\'i}a-Prada, Fran\c{c}ois Labourie, and Daniel Reyes Nozaleda for helpful discussions. N.S. is funded by the FNR grant O20/14766753, \textit{Convex Surfaces in Hyperbolic Geometry}. C.E. is funded by the European Union (ERC, \textit{GENERATE}, 101124349). Views and opinions expressed are however those of the author(s) only and do not necessarily reflect those of the European Union or the European Research Council Executive Agency. Neither the European Union nor the granting authority can be held responsible for them.

\section{Preliminaries I}
In this section, let $S$ be a closed oriented surface of genus $\mathrm g\geq 2,$ and let $\overline{S}$ be that same surface with the opposite orientation.

\subsection{Lie groups, principal bundles, and symmetric spaces}\label{sec: generalities}
In this section, we set $G$ to be a semisimple Lie group with no compact factors and with Lie algebra $\g$. While a lot of the constructions in this paper go through for reductive Lie groups, to keep things simpler we will often specialize.

\subsubsection{Lie groups and Lie algebras}\label{sec: Lie preliminaries} Let $\nu$ be an $\mathrm{Ad}(G)$-invariant bilinear form on $\g$ that is negative definite on Lie subalgebras of compact subgroups. A Cartan involution $\theta$ for $\nu$ is a Lie algebra involution of $\g$ such that $\nu(\cdot,-\theta(\cdot))$ is positive definite. We always write the $\nu$-orthogonal $+1$ and $-1$ eigenspace decomposition, i.e., the Cartan decomposition, as $\g=\kk\oplus \p$. The subalgebra $\kk\subset \g$ is the Lie algebra of a maximal compact subgroup $K\subset G,$ and in fact any maximal compact subgroup determines a Cartan involution. Since $\nu$ is unique up to independently rescaling its restriction to each simple factor (note each such factor is a positive multiple of the corresponding Killing form), $\p$ is independent of the choice of $\nu.$ 

Let $\mathrm{ad}:\g\to \mathrm{End}(\g)$ be the adjoint representation. An element $X\in\g$ is semisimple if $\mathrm{ad}(X)\in \mathrm{End}(\g)$ is a semisimple endomorphism, and nilpotent if $\mathrm{ad}(X)$ is nilpotent. The rank of $G$ is the dimension of a maximal abelian subalgebra $\mathfrak{h}$ of $\g$ contained in $\p$. Note that such a subalgebra necessarily consists of semisimple elements. We always write $\g^{\C},\p^{\C},\kk^{\C}, \mathfrak{h}^{\C}, G^{\C},$ etc., for complexifications.
A Cartan subalgebra of $\g^{\C}$ is a maximal abelian semisimple subalgebra. The Lie group $G$ is split if the complexification of a maximal abelian subalgebra of $\p\subset \g$ is a Cartan subalgebra of the complexification $\g^{\C}.$ For reference, the adjoint forms of the rank $2$ split real Lie groups are $\mathrm{PSL}(2,\R)^2$, $\mathrm{PSL}(3,\R),$ $\mathrm{PSp}(4,\R),$ and $G_2'$ (split form of the exceptional $G_2$).  

For $g\in G$, let $L_g:G\to G$ be the left multiplication map. The Maurer-Cartan form of $G$ is the section $\theta$ of $\Omega^1(G,\g)$ defined on vectors $v\in T_gG$ by $$\theta(v) = (L_{g^{-1}})_*v.$$
The Maurer-Cartan form provides an isomorphism of vector bundles from $TG$ to $G\times \g$.

\subsubsection{Principal bundles and connections}
To us, a principal $G$-bundle over a manifold $M$ is a fiber bundle $P\to M$ together with a continuous right action $P\times G\to P$ that preserves fibers and such that for each $p\in S,$ the map $\{p\}\times G\to P|_{\{p\}}$ is a homeomorphism.

Given a subgroup $H<G$ and a space $V$, a left action $\alpha$ of $H$ on $V$ determines an equivalence relation via $(g,v)\sim (g\cdot h,\alpha(h)v)$ for all $h\in H.$ The associated bundle is $G\times_{\alpha} V=(G\times V)/\sim.$ When the action $\alpha$ is implicit, we might write $G\times_H V$. If $H$ preserves a structure on $V$ (vector space structure, Lie bracket, etc.), then the associated bundle inherits that structure. For example, the adjoint action of $G$ on $\g$, $g\cdot X = \mathrm{Ad}_g(X),$ yields the adjoint bundle $\mathrm{ad}P=P\times_{G}\g.$ 

\begin{defn}\label{def: principal connection}
    A \textbf{principal connection} on $P$ is a $(S\times \g)$-valued $1$-form $A$ such that
    \begin{enumerate}
        \item for $g\in G$, if $R_g: P\times \{g\} \to P$ is the restriction of the $G$-action, then $R_g^*A = \mathrm{Ad}_{g^{-1}}\circ A$.
        \item If $X\in\g$ and $X^*$ is the {corresponding} $G^{\C}$-invariant vector field on $P$ induced by the infinitesimal $G$-action, then $A(X^*)=X.$
    \end{enumerate}
\end{defn}
A connection $A$ on $P$ determines an affine connection on $\mathrm{ad}P.$ In general, the converse is not quite true: a connection $\nabla$ on $\mathrm{ad}P$ induces a principal connection on $P$ if and only if $\nabla$ preserves the Lie bracket on $\mathrm{ad}P.$

For computational purposes, we often make use of connection forms. Given $P\to M$ with connection $A$ and an open set $V\subset M$ with local section $s:V\to P$, the pullback of $A$ along $s$ is a $(S\times \g)$-valued $1$-form on $V$ called a \textbf{connection form}. On $\mathrm{ad}P$, in the linear frames induced by $s$, the induced connection is expressed as $d+\mathrm{ad}(s^* A).$ 

At a number of points in this paper, we use connections in order to construct complex structures on bundles. If $G$ and $M$ are complex, an \textbf{almost complex structure} for the bundle $P\to M$ is a $G$-equivariant almost complex structure on $P$ such that the $G$-action and the projection $P\to M$ are pseudo-holomorphic. Given a connection $A$ on $P,$ there is a unique almost complex structure such that $A$ is of type $(1,0)$ (see \cite{KM}). A special case of a theorem of Koszul and Malgrange says that when $M$ is a Riemann surface, any almost complex structure is integrable \cite{KM}.

Fix a reduction of $P$ to a compact real form $K$, which induces a Hermitian metric on $\mathrm{ad}P$. The Chern correspondence gives a one-to-one correspondence between almost complex structures on $P$ and connections $A$ that restrict to connections on the reduction (see \cite{Sandon}). When $M$ is a Riemann surface, by the Koszul-Malgrange theorem, the term ``almost complex structure" can be replaced by ``complex structure." As explained carefully in Sandon's thesis \cite[Chapter 4]{Sandon}, if we have an integrable almost complex structure, then relative to $K$ and the induced complex structure on $\mathrm{ad}P$, the corresponding connection $A$ produces the affine Chern connection on $\mathrm{ad}P$.

\subsubsection{Symmetric spaces}
Consider the quotient $G/K$. The Maurer-Cartan form yields an isomorphism between $T(G/K)$ and the associated bundle  $G/K\times_{\mathrm{Ad}|_K} \mathfrak{p}.$ Moreover, the bilinear form $\nu$ from above determines a Riemannian metric on $G/K$, which we still denote by $\nu,$ and conversely, any $G$-invariant Riemannian metric on $G/K$ is identified with some $\mathrm{Ad}(G)$-invariant bilinear form on $G$. The manifold $(G/K,\nu)$ is a Riemannian symmetric space of non-compact type, by which we mean a Riemannian manifold with an isometric inversion symmetry about each point and whose De Rham decomposition has no compact or Euclidean factors.

The projection $G\to G/K$ gives $G$ the structure of a principal $K$-bundle over $G/K$ on which the projection of the Maurer-Cartan form to $\kk$ defines a principal connection $A_K$. Via the adjoint action of $K$ on $\p$, $A_K$ induces a connection on the associated vector bundle $G/K\times_{\mathrm{Ad}|_K}\p.$ Under the identification of $G/K\times|_{\mathrm{Ad}|_K}\p$ with the tangent bundle $T(G/K),$ this connection is the Levi-Civita connection on $(G/K,\nu)$ (see \cite{BR}). Evidently, this connection is independent of the specific choice of $\mathrm{Ad}(G)$-invariant $\nu$.

\subsection{Character varieties}
We recall just a few aspects of the theory around character varieties; see \cite{Golform} and \cite{Sik} for more in-depth treatments.

Let $G$ be a (real) reductive Lie group. A homomorphism $\rho:\pi_1(S)\to G$ is called irreducible if the image $\rho(\pi_1(S))$ is not contained in any proper parabolic subgroup of $G$. The quotient $\chi(\pi_1(S), G)$ of the set of irreducible homomorphisms by the conjugation action of $G$ is a Hausdorff space, usually called the \textbf{character variety} (we keep this name, even though we are not concerned with algebraic aspects). Let $\mathrm{Hom}^*(\pi_1(S),G)$ be the set of homomorphisms that are irreducible and also simple, the latter condition meaning that the centralizer of $\rho(\pi_1(S))<G$ is the center of $G$. The quotient by conjugation $$\mathrm{Hom}^*(\pi_1(S),G)/G=: \chi^{\mathrm{an}}(\pi_1(S),G) \subset \chi(\pi_1(S),G)$$ is an analytic smooth manifold. When $G$ is a complex Lie group, $\chi^{\mathrm{an}}(\pi_1(S),G)$ admits the structure of a complex manifold, and we view it as such.

 Recall that isomorphism classes of principal $G$-bundles are classified by characteristic classes in the group cohomology $H^2(S,\pi_1(G))\simeq \pi_1(G)$. The curvature of a connection $A$ is $F(A)=dA+[A,A]$, and $A$ is flat if $F(A)=0.$ Denote by $\mathcal{F}(G)$ the space of pairs $(P,A),$ where $P$ is a principal $G$-bundle in the trivial cohomology class and $A$ is a flat connection on $P$. For any $\rho\in \mathrm{Hom}(\pi_1(S),G),$ we can construct the principal $G$-bundle $G\times_\rho \widetilde{S}$, which lies in the trivial cohomology class and which has a canonical flat connection. That is, we get a point in $\mathcal{F}(G).$ Conversely, fixing a point $p\in S$ and a point of $P$ in the fiber over $p$, any $(P,A)\in \mathcal{F}(G)$ determines a (pointed) holonomy representation from $\pi_1(S)$ (based at $p$) to $G$ (see \cite[Chapter 13.9]{Taubes}). When saying ``the holonomy" of a flat irreducible connection, we ignore the basepoints and are really referring to the corresponding point in $\chi(\pi_1(S),G).$

Denoting by $\mathcal{F}^*(G)\subset \mathcal{F}(G)$ the subspace on which the holonomy lands in $\mathrm{Hom}^*(\pi_1(S),G)$, we set $\mathcal{A}(S,G)$ to be the set of isomorphism classes of flat $G$-bundles inside $\mathcal{F}^*(G)$. As is well known, the pointed holonomy map $\mathcal{F}(S,G)\to \mathrm{Hom}^*(\pi_1(S),G)$ descends to a bijection $\mathcal{A}(S,G)\to \chi^{an}(S,G)$ (see \cite[Theorem 13.2]{Taubes}). It is common to create maps to $\chi^{an}(S,G)$ by passing through this bijection.

We'll often embed real Lie groups inside complex ones, so the following notion will be useful. Let $G$ be a real reductive Lie group with complexification $G^{\C}$. Let $V$ be an open subset of $\chi^{\mathrm{an}}(\pi_1(S),G)$ such that every element in $V$ is represented by a homomorphism $\rho:\pi_1(S)\to G$ for which the composition with the inclusion $G<G^{\C}$ defines a point in $\chi^{\mathrm{an}}(\pi_1(S),G^{\C})$, and such that the map $V\to \chi^{\mathrm{an}}(\pi_1(S),G^{\C})$ is a (totally real) embedding. We refer to such a subset $V$ as \textbf{admissible}, and from here on out we identify such $V$ with its image in $\chi^{\mathrm{an}}(\pi_1(S),G^{\C})$.

\subsection{Teichm{\"u}ller space and Hitchin components}\label{sec: definition teichmuller} Fix a basepoint complex structure $c_0$ on $S$ that's compatible with the orientation and let $\mathcal{K}$ be the canonical bundle. A \textbf{Beltrami form} $\mu$ on $S$ is an $L^\infty$-measurable section of the bundle $\mathcal{K}^*\otimes \mathcal{K}^{-1}$ with $L^\infty$ norm strictly less than $1$. 

Recall that the space of complex structures on $S$ 
is parametrized by Beltrami forms.  Explicitly, given a Beltrami form $\mu$ on $(S,c_0)$, the Measurable Riemann Mapping Theorem produces a unique pair consisting of a compatible complex structure $c_\mu$ together with a quasiconformal map $f_\mu:(S,c_0)\to (S,c_\mu)$. We say that a complex structure is $C^\infty$ if the Beltrami form is $C^\infty$. We write $\mathcal{C}(S,c_0)$ for the complex Fr{\'e}chet space of $C^\infty$ Beltrami forms, and $\mathcal C(S)$ for the space of smooth complex structures. The parametrization above gives $\mathcal C(S)$ a complex Fr{\'e}chet manifold structure whose biholomorphism class does not depend on the choice of $c_0$ (see \cite[page 26]{EE}).

 Let $\mathrm{Diff}_+(S)$ be the group of orientation preserving $C^\infty$ diffeomorphisms of $S$ endowed with the topology of $C^\infty$ uniform convergence and let $\mathrm{Diff}_0(S)$ be the normal subgroup of diffeomorphisms isotopic to the identity. The quotient ${\mathrm{Diff}_+(S)}/{\mathrm{Diff}_0(S)}$ is the mapping class group $\mathrm{MCG}(S).$ Precomposing charts with diffeomorphisms allows one to build an action of $\mathrm{Diff}_0(S)$ on $\mathcal{C}(S)$, which we point out has a natural description in any parametrization of $\mathcal{C}(S)$ by Beltrami forms (see \cite[page 27]{EE}).
\begin{defn}
    The \textbf{Teichm{\"u}ller space} of marked complex structures on $S$ is the quotient $\mathcal{T}(S)=\mathcal{C}(S)/\mathrm{Diff}_0(S)$.
\end{defn}
\begin{thm}[Earle-Eells \cite{EE}]\label{thm: earle eells}
  Under the quotient topology, $\mathcal{T}(S)$ inherits a complex structure such that the map from $\mathcal{C}(S)\to \mathcal{T}(S)$ is holomorphic and a principle $\mathrm{Diff}_0(S)$-fiber bundle.
\end{thm}

Note that the complex structure on the Teichm{\"u}ller space $\mathcal{T}(\overline{S})$ of the oppositely oriented surface $\overline{S}$ identifies with the complex structure conjugate to the usual one on $\mathcal{T}(S)$. 

Given a complex structure on $S$, the Uniformization Theorem provides a biholomorphism $\mathrm{dev}$ between the universal covering Riemann surface and the upper half-plane $\mathbb{H}\subset \C$. The map $\mathrm{dev}$ is unique up to post-composing with elements in $\PSL(2,\R)=\mathrm{Aut}(\mathbb{H})$ and is \textbf{equivariant}, i.e., there exists a homomorphism $\rho: \pi_1(S)\to \PSL(2,\R)$ such that $\mathrm{dev}\circ \gamma=\rho(\gamma)\circ \mathrm{dev}$ for all $\gamma\in \pi_1(S)$. The homomorphism $\rho$ is discrete and faithful, i.e., it is a \textbf{Fuchsian representation}. This association from complex structures to representations descends to a diffeomorphism between $\mathcal T(S)$ and one of the two connected components of Fuchsian representations inside $\chi^{an}(\pi_1(S),\PSL(2,\R))$. The same construction for the oppositely oriented surface $\overline S$ determines a diffeomorphism between $\mathcal T(\overline S)$ and the other connected component of Fuchsian representations.

The definition of Hitchin representations stems from the above perspective on Teichm{\"u}ller space. Let $G^{\C}$ be a complex simple Lie group of adjoint type. There is a special type of embedding $\mathrm{PSL}(2,\C)\to G^{\C}$ called a principal homomorphism, which is unique up to automorphisms; we defer further explanation to later on, since we'll construct one explicitly in Section \ref{sec: bi-Hitchin}. We fix a principal homomorphism $\iota_G: \mathrm{PSL}(2,\C)\to G^{\C}$, which restricts to an embedding $\mathrm{PSL}(2,\R)\to G,$ where $G$ is a split real form of $G^{\C}.$ Fixing a Fuchsian representation $\rho:\pi_1(S)\to \mathrm{PSL}(2,\R)$, the \textbf{Hitchin component} of $G$ is the component of $\chi(\pi_1(S),G)$ containing $\iota_G \circ \rho$. Note we had to choose $\iota_G$, so really there are a finite number of Hitchin components (see \cite{Hi}), but we will be lazy in this terminology.

Any component $\mathrm{Hit}(S,G)$ lies in $\chi^{\mathrm{an}}(\pi_1(S),G)$ and is admissible in the sense of the previous subsection. Justifications for these assertions can be gleaned from Hitchin's paper \cite{Hi} and the later work \cite{GPNR}. Indeed, first note that if $\rho:\pi_1(S)\to G$ post-composed with the inclusion from $G\to G^{\C}$ lies in $\mathrm{Hom}^*(\pi_1(S),G^{\C}),$ then $\rho$ lies in $\mathrm{Hom}^*(\pi_1(S),G).$ As well, under the non-abelian Hodge correspondence, that representations in $\mathrm{Hom}^*(\pi_1(S),G^{\C})$ correspond to stable and simple $G^{\C}$-Higgs bundles (see \cite{GPNR} for definitions, especially Section 5.4). The result then follows from \cite[\S 5]{Hi} and \cite[Section 6]{GPNR}, where the authors show, under the non-abelian Hodge correspondence (see the subsection below), Hitchin representations correspond to stable and simple $G^{\C}$-Higgs bundles, and that the map $\mathrm{Hit}(S,G)\to \chi^{\mathrm{an}}(\pi_1(S),G^{\C})$ is an embedding. 

\subsection{Harmonic maps and $G$-Higgs bundles}\label{sec: ordinary harmonic maps}
Ordinary harmonic maps are the most basic example of the complex harmonic maps that we'll introduce later on. See Section \ref{sec: def of ch maps} for the definition, which we won't really need in this subsection.

In the rest of this subsection, we fix a semisimple Lie group $G$ with no compact factors and bilinear form $\nu$ on the Lie algebra $\g$ as in Section \ref{sec: generalities}, and consider harmonic maps to a Riemannian symmetric space $(G/K,\nu).$ For such spaces, harmonicity does not depend on the choice of invariant metric $\nu$: for a simple Lie group, the notion is independent of the scaling on $\nu,$ and in the semisimple case, a map is harmonic if only if the map to each simple factor is harmonic. So, in such contexts, we'll often write about harmonic maps while omitting the $\nu$. 

As is well known, a harmonic map to a Riemannian symmetric space $(G/K,\nu)$ gives rise to a $G$-Higgs bundle (Definition \ref{def: G-Higgs} below). We explain the construction here. For relevant subspaces of $\g$, complexifications, etc., we follow the notation from Section \ref{sec: generalities}. 

Let $S$ be a surface with complex structure $c$, $\rho:\pi_1(S)\to G$ a representation, and $f:(\widetilde{S},c)\to G/K$ a $\rho$-equivariant map. Associated with $\rho$ we have the principal $G$-bundle $Q_\rho=\widetilde{S}\times_\rho G$ over $S$, which comes equipped with a flat connection $A$. The map $f$ is equivalent to a reduction of structure group to $K,$ $Q_f\subset Q_\rho$. The reduction induces, on the adjoint bundle $\mathrm{ad}Q_\rho=Q_\rho\times_G \g$, an involution $\sigma_f:\mathrm{ad}Q_\rho\to \mathrm{ad}Q_\rho$ that produces a Cartan involution in each fiber.

Recalling the principal connection $A_K$ on the principal $K$-bundle $G\to G/K$ from Section \ref{sec: Lie preliminaries}, we pull it back via $f$ to get a principal connection $A_{f}$ on $Q_f$. Under the splitting $\g=\p\oplus \kk$, via the adjoint action, $A_{f}$ induces a connection $\nabla$ on $\mathrm{ad}\p_f:=Q_f\times_{\mathrm{Ad}|_{K}}\p.$ The Maurer-Cartan isomorphism identifies the pullback bundle $f^*T(G/K)$ with $\mathrm{ad}Q_f,$ and the pullback of the Levi-Civita connection with $\nabla$. Furthermore, under this isomorphism, $df$ is a $1$-form valued in $\mathrm{ad}Q_f$, which we denote by $\psi_f.$ Upon extending the structure group of $Q_f$ to $K^{\C}$ via $P_f:=Q_f\times_{K}K^{\C}$ and viewing $\psi_f$ as a $1$-form valued in $\mathrm{ad}\p_f^{\C}:=P_{f}\times_{ \mathrm{Ad}|_{K^{\C}}}\p^{\C}$, we can decompose $\psi_f$ into $(1,0)$ and $(0,1)$ forms as $\psi_f=\psi_f^{1,0}+\psi_f^{0,1}.$ Similarly, decompose $\nabla$ into $(1,0)$ and $(0,1)$ components. It is well known that the definition of $f$ \textbf{harmonic} is equivalent to $\nabla^{0,1}\psi_f^{1,0}=0$ (see \cite{Li}). 

Henceforth assuming that $f$ is harmonic, by the Chern correspondence, $A_f$ determines a holomorphic structure on $P_{f}$ in which $\psi_f^{1,0}$ is a holomorphic $\mathrm{ad}\p_f^{\C}$-valued $1$-form.
The extension of structure group of $P_f$ to $G^{\C}$, $P_{f}\times_{K^{\C}}G^{\C}$, identifies with $P_\rho:= Q_\rho\times_G G^{\C}$. The Cartan involution $\sigma_f:\mathrm{ad}Q_\rho\to \mathrm{ad}Q_{\rho}$ then extends sesquilinearly to an isomorphism from the adjoint bundle to its conjugate bundle $$\widehat{\sigma}_f:\mathrm{ad}P_{\rho}\to \overline{\mathrm{ad}P_{\rho}}.$$  
Since $\sigma_f(\psi_f)=-\psi_f$, if $\phi_f=\psi_f^{1,0},$ then $\widehat{\sigma}_f(\phi_f)=-\psi_f^{0,1}.$ The flatness of the original connection $A$ on $P$ is equivalent to saying $$F(A_{f})-[\phi_f,\sigma_f(\phi_f)]=0.$$ If we drop the information of the reduction $P_K$, we obtain the notion of a $G$-Higgs bundle. Let $\mathcal{K}$ be the canonical bundle of $(S,c).$ For a general $K^{\C}$-bundle $P_{K^{\C}},$ we always write $\textrm{ad}\p^{\C}:=P_{K^{\C}} \times_{ \mathrm{Ad}|_{K^{\C}}}\p^{\C}$.
\begin{defn}\label{def: G-Higgs}
    A \textbf{$G$-Higgs bundle} is the data $(P_{K^{\C}},\phi)$, where $P_{K^{\C}}$ is a holomorphic principal $K^{\C}$-bundle over $(S,c)$ and $\phi$ is a holomorphic section of $\mathrm{ad}\p^{\C}\otimes \mathcal{K}$, called the \textbf{Higgs field}.
\end{defn}
A $G$-Higgs bundle together with a suitable reduction is called a harmonic $G$-bundle. 
\begin{defn}
    A \textbf{harmonic $G$-bundle} is the data $(P_{K^{\C}},\phi,P_K)$, where 
    $(P_{K^{\C}},\phi)$ is a $G$-Higgs bundle, and $P_K$ is a reduction of structure group to $K$ with Chern connection $A_{P_K}$ and, for $P_{G^{\C}}=P_{K^{\C}}\times_{K^{\C}}G^{\C},$
 involution $\widehat{\sigma}: \mathrm{ad}P_{G^{\C}}\to \overline{\mathrm{ad}P_{G^{\C}}}$ such that the connection on the principal $G$-bundle $P_{K}\times_{K} G$ defined by $$A_{P_K}+\phi-\widehat{\sigma}(\phi)$$ is flat.
\end{defn}

A harmonic $G$-bundle contains the information of a conjugacy class of representations to $G$ together with, for each representation, an equivariant harmonic map to $G/K,$ in a way that reverses the procedure carried out above. Indeed, given the data $(P_{K^{\C}},\phi,P_K),$ we extend the structure group to the principal $G$-bundle $P_{K}\times_{G} G$ as above, which inherits a flat connection $A=A_{P_K}+\phi-\widehat{\sigma}(\phi)$. Picking a point on $P$, the data of the inclusion $P_{K}\subset P_{K}\times_{G} G$ is equivalent to a map $f$ to $G/K$ that is equivariant for the pointed holonomy of the flat connection. Clearly, if we pick our point correctly, then this reverses the construction of a harmonic bundle from a harmonic map, up to conjugation and translation. We do want to verify that if we start with an arbitrary harmonic $G$-bundle $(P_{K^{\C}},\phi,P_K),$ then reduction induces a map $f$ that's harmonic. Just note that the decomposition of $A$ over $\p\oplus \mathfrak{k}$ is unique, and hence $\phi-\widehat{\sigma}(\phi)$ identifies with the tangent map of $f$ and $A_{P_K}$ comes from the Levi-Civita connection on the holomorphic Riemannian symmetric space. Holomorphicity of $\phi$ is then equivalent to harmonicity of $f$.

Naturally, one is led to ask, given a $G$-Higgs bundle, whether one can find a reduction that turns it into a harmonic $G$-bundle. There are notions of stability and polystability for a $G$-Higgs bundle defined in terms of an antidominant character for a parabolic subgroup of $G$ and a holomorphic reduction of the structure group of $P_{K^{\C}}$. We need these notions to state the next result, but we
omit the definitions since we won’t make use of them in any meaningful way (see \cite{G-P} for this theory).
\begin{thm}[Hitchin \cite{Hitchin:1986vp}, Simpson \cite{Si}]\label{thm: Hitchin simpsons}
    A $G$-Higgs bundle $(P_{K^{\C}},\phi)$ is polystable if and only if there exists a reduction of the structure group of $P_{K^{\C}}$ to $K$, say, $P_K$, which makes $(P_{K^{\C}},\phi,P_K)$ a harmonic $G$-bundle. If the $G$-Higgs bundle is stable, then the reduction is unique.
\end{thm}
Working in the adjoint representation, if $\overline{\partial}_E$ is the del bar operator, the theorem can be proved by finding a reduction with involution $\widehat{\sigma}$ that solves \textbf{Hitchin's self-duality equations}: 
\begin{equation}\label{eq: self-duality real case}
    F(\overline{\partial}_E+\widehat{\sigma}(\overline{\partial}_E))+[\phi,-\widehat{\sigma}(\phi)]=0.
\end{equation}
(In the adjoint bundle, $\overline{\partial}_E+\widehat{\sigma}(\overline{\partial}_E)$ is the Chern connection of the reduction.)

The counterpart to Theorem \ref{thm: Hitchin simpsons}, due to Donaldson \cite{D} and Corlette \cite{Co}, is that if $\rho:\pi_1(S)\to G$ is a reductive representation, then there exists a $\rho$-equivariant harmonic map $f:(\widetilde{S},c)\to G/K$ if and only if $\rho$ is reductive (and if $\rho$ is irreducible, it is unique). The results of \cite{Hitchin:1986vp}, \cite{D}, \cite{Co}, and \cite{Si} constitute a version of the non-abelian Hodge correspondence, which is an equivalence between classes of reductive representations to $G$ and classes of polystable $G$-Higgs bundles.

\subsection{Labourie's complex structures}\label{sec: Labourie paper}
Here we recall Labourie's parametrizations of the rank $2$ Hitchin components, as well as Labourie's immersions for cyclic $G$-Higgs bundles. Let $(S,c)$ be a Riemann surface with canonical bundle $\mathcal{K}.$ As in Section \ref{sec: generalities}, we fix $G$ to be a semisimple Lie group with no compact factors and Lie algebra $\g$. 

\subsubsection{Holomorphic differentials}
 We denote tensor powers of $\mathcal{K}$ by $\mathcal{K}^r:=\mathcal{K}^{\otimes r}.$ Elements of $H^0(S,\mathcal{K}^r)$ are called holomorphic differentials.

As is well known, Teichm{\"u}ller space carries a natural complex structure. Given a finite ordered set of integers $(d_i)_{i=1}^k,$ let $\mathrm{M}(S,(d_i))$ be the the bundle over the space $\mathcal C(S)$, whose fiber over $c$ is $\oplus_i H^0(S,\mathcal{K}^{d_i}),$ and set $\mathcal{M}(S,(d_i))$ to be the quotient of $\mathrm{M}(S,(d_i))$ by the natural action of $\mathrm{Diff}_0(S)$, which is a holomorphic vector bundle over $\mathcal{T}(S).$ As explained in \cite[Section 6.2]{ElSa} (in the case $k=1$, but there is no difference in the general construction), it can be seen using the classical theory developed in \cite{Bhd} that $\mathrm{M}(S,(d_i))$ and $\mathcal{M}(S,(d_i))$ are holomorphic vector bundles. There is a natural action of the mapping class group on on $\mathrm{M}(S,(d_i))$, $[\upphi]\cdot [c,(q_1,\dots, q_k)]=[\upphi^*c,(\upphi^*q_1,\dots, \upphi^* q_k)],$ where $\upphi^*c$ is the pullback complex structure.

\subsubsection{Labourie's work}
Let $\mathcal{O}(\g)^{G}$ be the algebra of $\mathrm{Ad}(G)$-invariant polynomials on the Lie algebra $\g$ of $G$. Let $\{p_1,\dots, p_l\}$ be a homogeneous minimal generating set for $\mathcal{O}(\g)^{G}$, listed so that the degrees $m_i$ (called the exponents of $G$) are ascending. When $G$ is simple, the number $m_l$ is called the \textbf{Coxeter number} of $G,$ and we denote it by $d_G,$ or just $d$ when the context is clear. The Hitchin fibration, defined on isomorphism equivalence class of $G$-Higgs bundles (see \cite{Hi}), sends (the isomorphism class of) a $G$-Higgs bundle $(P_{K^{\C}},\phi)$ to $(p_1(\phi),\dots, p_l(\phi))\in H(c,G)=\oplus_{i=1}^l H^0(S,\mathcal{K}^{m_i})$.

We now set $G$ to be simple and of adjoint type, as in \cite{Lab3}. As the Killing form is the unique adjoint invariant bilinear form on $\g$ up to scale, the polynomial $p_1\in \mathcal{O}(\p^{\C})^{K^{\C}}$ is a multiple of the Killing form, and no other $p_i$ has degree $2$. Given an equivariant harmonic map to $G/K,$ the $(2,0)$ component of the pullback metric (extended to the complexified tangent bundle) is a degree $2$ invariant polynomial in the Higgs field $\phi$, and hence the harmonic map is weakly conformal if and only if $p_1(\phi)=0$ (see Section \ref{sec: def of ch maps} or \cite{Li} for more explanation). Using Hitchin's section and the non-abelian Hodge correspondence, a minimal immersion $\widetilde{S}\to G/K$ that's equivariant for a Hitchin representation is equivalent to the data of a complex structure $c$ and a point in $(q_2,\dots, q_l) \in \oplus_{i=2}^l H^0(S,\mathcal{K}^{m_i})$ (identified with $(0,q_2,\dots, q_l) \in H(c,G)$)).

The space of such minimal immersions, up to the conjugation action of $G,$ can thus be packaged as the holomorphic bundle $\mathcal{M}(S,(m_2,\dots, m_l))$ over $\mathcal{T}(S).$ We point out that the mapping class group $\mathrm{MCG}(S)$, along with acting on this bundle, also acts by outer automorphisms on $\mathrm{Hit}(S,G)$. The map that associates a minimal surface to its holonomy yields a map $$\mathcal{L}_{G}: \mathcal{M}(S,(m_2,\dots, m_l))\to \mathrm{Hit}(S,G)$$ that intertwines the two actions of $\mathrm{MCG}(S).$ In \cite{L1} and \cite{Lab3}, when $G$ has rank $2,$ Labourie proved that every Hitchin representation comes with a unique equivariant minimal immersion (see \cite{SS} for the history of this result). Moreover, we have the following.
\begin{thm}[\cite{Lab3}]
    Assuming $G$ has rank $2$, the map $\mathcal{L}_{G}$ is a diffeomorphism.
\end{thm}

\begin{defn}
    Let $G$ be a split real simple Lie group of rank $2$. The complex structure induced on $\mathrm{Hit}(S,G)$ is called \textbf{Labourie's complex structure}, or the Labourie complex structure.
\end{defn}

For arbitrary rank at least $2$, we consider the subbundle of $\mathcal{M}(S,(m_2,\dots, m_l))$ of points of the form $[c,(0,\dots, 0, q_l)].$ As we mentioned in the introduction, Labourie proved in \cite[Theorem 1.5.1]{Lab3} that the restriction of $\mathcal{L}_{G}$ to this subbundle is an immersion. We package this subbundle as its own bundle, which we give its own notation: from now on, set $\mathrm{M}(S,G)=\mathrm{M}(S,m_l)$ and $\mathcal{M}(S,G)=\mathcal{M}(S,(m_l)).$ In rank $2,$ $\mathrm{M}(S,G)$ identifies with the whole bundle of minimal surfaces.

Finally, a word on the rank $2$ semisimple group $G=\mathrm{PSL}(2,\R)^2.$ 
So as not to complicate things too much, we set $\nu$ to be the product of two hyperbolic metrics of constant curvature $-1$, although one could also treat every invariant metric $\nu$ without any added difficulty. An equivariant map $f$ is minimal if and only if it is harmonic and the factors have Higgs fields $\phi^1$ and $\phi^2$ such that $p_1(\phi^1)+p_1(\phi^2)=0$ (again, see Section \ref{sec: def of ch maps} or \cite{Li}). Via $q\mapsto (q,-q)$, we identify the space of minimal surfaces with $\mathcal{M}(S,G):=\mathcal{M}(S,(2))$.  It follows from work of Schoen \cite{Sc} that when we postcompose with the holonomy map, we get a bijection. 



\section{Preliminaries II}

\subsection{Holomorphic Riemannian manifolds}
\label{section: holo Riem mflds}
Let $\mathbb X$ be a complex manifold with almost complex structure $\mathbb J$. A \textbf{holomorphic Riemannian metric} $\inners$ on $\mathbb X$ is a holomorphic assignment of a non-degenerate $\C$-bilinear form on the holomorphic tangent bundle. The pair $(\mathbb X, \inners)$ is called a \textbf{holomorphic Riemannian manifold}. Such objects have been studied by many authors; see, for instance, \cite{DZ}. 

Holomorphic Riemannian manifolds come with enough structure to formulate a covariant differential calculus. A holomorphic affine connection $D$ on $T\mathbb{X}$ is an affine connection that satisfies $D\mathbb J=0$ and, for local holomorphic vector fields $Z_1,Z_2$, $D_{\mathbb JZ_1} Z_2=\mathbb J\left( D_{Z_1} Z_2\right)$. Holomorphic Riemannian manifolds $(\mathbb X, \inners)$ have a natural notion of \textbf{Levi-Civita connection $D$}, which is a holomorphic affine connection on $T\mathbb X$ that makes the metric parallel. Mimicking the Riemannian setting, one can use the Levi-Civita connection to build a $\C$-multilinear curvature tensor and, for all $2$-dimensional complex vector subspaces $V< T_p \mathbb X$ on which $\inners|_{V}$ is non-degenerate, one can define the sectional curvature. 

The most basic example of a holomorphic Riemannian manifold is $\mathbb C^n$ equipped with the canonical symmetric $\C$-bilinear form given by $\langle\underline z, \underline w\rangle=\sum_{k=1}^n z_k w_k$, for $\underline z =(z_1,\dots, z_n),\underline w =(w_1,\dots, w_n)\in \C^n$. 
$(\mathbb \C^n, \inners)$ is the unique simply connected complete holomorphic Riemannian metric with constant sectional curvature zero \cite[Theorem 2.6]{BEE}. 

In this paper, we are most interested in holomorphic Riemannian manifolds constructed as follows. Let $G$ be a semisimple Lie group with no compact factors and Lie algebra $\g$. As well, let $K<G$ be a maximal compact subgroup with Lie algebra $\kk$ and corresponding splitting $\g=\p\oplus \kk.$ Recalling that $G^{\C}$, $K^{\C}$, etc., are complexifications, we consider the quotient $G^{\C}/K^{\C}$, which is a complex manifold because the action of $K^{\C}$ on $G^{\C}$ is holomorphic. The Lie algebra $\g^{\C}$ has the splitting $\g^{\C}=\p^{\C}\oplus \kk^{\C},$ and the Maurer-Cartan form identifies $T (G^{\C}/K^{\C})$ with $G^{\C}/K^{\C}\times_{\textrm{Ad}|_{K^{\C}}} \p^{\C}.$ As in Section \ref{sec: generalities}, let $\nu$ be an $\mathrm{Ad}(G)$-invariant bilinear form on $\g$. We extend it complex linearly to a $\mathrm{Ad}(G^{\C})$-invariant complex bilinear form $\nu^{\C}$ on $\g^{\C}.$ Via the Maurer-Cartan isomorphism, the restriction of $\nu^{\C}$ to $\p^{\C}$ determines a holomorphic Riemannian metric on $G^{\C}/K^{\C},$ which we continue to denote by $\nu^{\C}$, such that the immersion $(G/K,\nu)\to (G^{\C}/K^{\C},\nu^{\C})$ is isometric, totally geodesic, and totally real. 
\begin{defn}
   We refer to $(G^{\C}/K^{\C},\nu^{\C})$ as a \textbf{holomorphic Riemannian symmetric space.} 
\end{defn}
One thing we will need to observe, similar to the Riemannian setting, is the characterization of the Levi-Civita connection in terms of principal connections. The projection $G^{\C}\to G^{\C}/K^{\C}$ makes $G^{\C}$ a principal $K^{\C}$-bundle over $G^{\C}/K^{\C}$, and the projection of the Maurer-Cartan form to $\kk^{\C}$ is a principal connection $A_{K}^{\C}$. Using $\mathrm{Ad}|_{K^{\C}}$ on $\p^{\C}$, $A_{K}^{\C}$ induces a holomorphic affine connection on the associated complex vector bundle $G^{\C}/K^{\C}\times|_{\mathrm{Ad}|_{K^{\C}}}\p^{\C},$ which identifies with the Levi-Civita connection for $\nu$ on $T(G^{\C}/K^{\C}).$ This last assertion can be checked directly (by adapting some content from \cite{BR}), or it can be seen by observing that it is equal to the Levi-Civita connection on $G/K\subset G^{\C}/K^{\C}$ and using holomorphicity. 

Perhaps the most well-studied holomorphic Riemannian symmetric spaces are hyperboloids, \[
\mathbb X_n =\{\underline z \in \C^{n+1}\ |\ \langle \underline z, \underline z \rangle =-1 \},
\]
which inherit complex Riemannian metrics from their inclusions into $\C^{n+1}.$ It is easily seen that $\mathbb X_n$ identifies with $\mathrm{SO}(n+1,\C)/\mathrm{SO}(n,\C)$, with a suitably normalized Killing form. For $n\ge 2$, $\mathbb X_n$ is the unique complete simply connected holomorphic Riemannian manifold with constant sectional curvature $-1$ \cite[Theorem 2.6]{BEE}. It should be clear that the real $n$-dimensional hyperbolic space embeds in $\mathbb X_n$, in a unique way, up to post-composition with ambient isometries. 

Most studies on harmonic maps to symmetric spaces focus on $Y_n=\mathrm{SL}(n,\R)/\mathrm{SO}(n,\R)$, which models positive definite symmetric matrices in $\mathrm{SL}(n,\R)$, and $Y_n'=\mathrm{SL}(n,\C)/\mathrm{SU}(n)$,  which models positive definite Hermitian matrices in $\mathrm{SL}(n,\C)$. The holomorphic Riemannian manifold associated with $Y_n$ is $\mathbb{Y}_n = \mathrm{SL}(n,\C)/\mathrm{SO}(n,\C)$, which models symmetric matrices in $\mathrm{SL}(n,\C)$. If we associate matrices with bilinear forms (so that matrices in $Y_n$ become inner products on $\R^n$), then the inclusion $Y_n\to Y_n'$ can be seen as the map induced by taking sequilinear extensions of bilinear forms, while the inclusion $Y_n\to\mathbb{Y}_n$ can be seen as taking bilinear extensions. The metric on all of these spaces at a matrix $A$ is, up to normalization, $\nu_A(X,Y)=\frac{n}{2}\mathrm{tr}(A^{-1}XA^{-1}Y).$ 

Of particular importance to us will be the space $\mathbb{X}_2=\mathbb{Y}_2.$ Along with the two models given above, $\mathbb{X}_2$ identifies with the space of oriented geodesics of hyperbolic space $\mathbb{H}^3$. Indeed, $\mathrm{SO}(3,\C)$ acts transitively on the space of oriented geodesics, and the stabilizer of any oriented geodesic is conjugate to $\mathrm{SO}(2,\C)$. By identifying geodesics in $\mathbb{H}^3$ with their endpoints in $\mathbb{CP}^1$, we obtain a biholomorphism from $$\mathbb X_2\to \mathbb{G}:=\mathbb{CP}^1\times \mathbb{CP}^1\backslash \Delta.$$ Under this biholomorphism,  the pushed-forward holomorphic Riemannian metric has the following description: given any complex affine chart $(V,z)$ on $\CP^1$, the metric on $(V\times V\setminus \Delta, (z_1,z_2))$  has the form \[
 \inners_{\mathbb G}= -\frac 4 {(z_1-z_2)^2} dz_1\cdot dz_2 \ .
 \]
See \cite[Section 2.4]{BEE} for a rigorous treatment of the discussion above. 
Several relations between immersions of surfaces into $\mathbb H^3$ and $\mathbb G$ are treated in \cite{EleSeppi}.

\subsection{Complex metrics}\label{subsection: complex metrics}
Our study will require some notions from the theory of complex metrics, and more generally of \cite{BEE} and the works that followed. Throughout the paper, given a smooth manifold $M,$ we denote the complexified tangent bundle $TM\otimes_{\R}\C$ by $\C TM,$ and likewise we denote the complexified cotangent bundle by $\C T^*M$.
\begin{defn}
    A \textbf{complex metric} on a smooth (real) manifold $M$ is a smooth section of $\mathrm{Sym}^2(\mathbb C T^*M)$ that determines a non-degenerate symmetric bilinear form in each fiber of $\C TM$.
\end{defn}
The bilinear extension of a Riemannian metric to the complexified tangent bundle is a complex metric. Throughout the paper, we will view a complex metric on a manifold $M$ as the same thing as an invariant complex metric on the universal cover $\widetilde{M},$ and we won't distinguish notation.

A lot of the usual constructions in Riemannian geometry extend to the setting of complex metrics and the proofs apply verbatim (see \cite{BEE} for all of the details). For instance, every complex metric $g$ has a natural \textbf{Levi-Civita connection}, namely, an affine connection $\nabla^g: \Gamma( \C TM) \to \Gamma(End(\C TM))$ uniquely determined by the fact that it is torsion-free and compatible with the metric, i.e., $\nabla^g g=0$. Using $g$ and $\nabla^g$, one can define and discuss gradients, the curvature tensor, Gauss curvature, Hessians, and Laplacians (always denoted $\Delta_g$) totally analogous to the Riemannian manifolds. 

In the following, we restrict to the case of surfaces $M=S$. The fact that a complex metric $g$ on $S$ is non-degenerate implies that $g$ has two isotropic directions at each point. In other words, for each point $p\in S$, the set $\{v\in \C T_p S\ |\ g_p(v,v)=0 \}$ consists of two complex lines of $\mathbb C T_p S$, corresponding to two points in $\mathbb P \C T_pS$. Observe that $\mathbb P T_p S$ embeds into the sphere $\mathbb P \C T_pS$ as an equatorial circle, which cuts $\mathbb P \C T_pS$ into two connected components.

\begin{defn}\label{def: positive}
    A complex metric $g$ is \textbf{positive} if for each $p\in S$, no isotropic direction of $g_p$ lies in $\mathbb P T_p S$, and the two isotropic directions are contained in distinct components of $\mathbb P \C T_pS\setminus \mathbb P T_p S$.
\end{defn}
Riemannian metrics complexify to positive complex metrics: the isotropic directions are the eigen-lines of the almost complex structure and are hence antipodal. 
The isotropic directions of a positive complex metric also related to the notion of a bi-complex structure.
\begin{defn}
    A \textbf{bi-complex structure} on $S$ is a tensor $J\in \Gamma(End(\C TS))$ such that $J^2=-id$ and whose eigenspaces, for each $p\in S,$ do not intersect $\mathbb P T_p S$ and are contained in distinct components of $\mathbb P \C T_pS\setminus \mathbb P T_p S$.  
\end{defn}
As shown in \cite[Section 6]{BEE}, a bi-complex structure is determined by its eigenspaces. As a result, the data of $J$ is equivalent to a pair of oppositely oriented complex structures $c_1$ and $\overline{c_2}$ on $S$, uniquely defined by the fact that their holomorphic tangent bundles are the linear subbundles of $\C TS$ corresponding to the eigenspaces for $i$ and $-i$ of $J$ respectively.

By associating isotropic directions to eigenlines, every positive complex metric $g$ uniquely determines a bi-complex structure $J$ such that $g(J\cdot,J\cdot)=g(\cdot,\cdot).$ If $z$ and $\overline w$ are local holomorphic coordinates for the corresponding complex structures $c_1$ and $\overline{c_2}$, then $\partial_{\overline z}$ and $\partial_w$ are the isotropic directions for $g$, and $g$ can locally be written as $g=\lambda dz d\overline w$, for some non-vanishing complex valued function $\lambda$ (see \cite[Section 6.1]{BEE} for details).
\subsection{Bers' theorem and Bers metrics}
 Given a quasi-Fuchsian representation $\rho$ preserving a Jordan curve $\gamma$ in $\mathbb{CP}^1$, $\rho$ also preserves and acts properly discontinuously on the two connected domains $\Omega_\rho^+,\Omega_\rho^-\subset \mathbb{CP}^1$ that are complementary to $\gamma$, which are called domains of discontinuity. The resulting quotients $\Omega_\rho^+/\rho$ and $\Omega_\rho^-/\rho$ represent points in $\mathcal{T}(S)$ and $\mathcal{T}(\overline{S})$ respectively. Bers'  Simultaneous Uniformization Theorem says that the process is reversible.
\begin{thm} [Bers' Simultaneous Uniformization Theorem \cite{SU}]
\label{thm: Bers thm}
	For all $(c_1, \overline{c_2})\in \mathcal{C}(S)\times \mathcal{C}(\overline{S})$, there exists a quasi-Fuchsian representation $\rho: \pi_1(S)\to \PSL(2, \mathbb C)$, unique up to conjugation, with domains of discontinuity $\Omega_\rho^+$ and $\Omega_\rho^-$, together with unique $\rho$-equivariant holomorphic  diffeomorphisms $$\vb*{f_+}(c_1,\overline {c_2}): (\widetilde S, c_1) \to \Omega_\rho^+,\qquad \qquad  \vb*{\overline{f_-}}(c_1,\overline {c_2}): (\widetilde S, \overline{c_2}) \to \Omega_\rho^-.$$
	This correspondence determines a biholomorphism \[ \TSTS\xrightarrow{\sim}\mathcal{QF}(S).\]
\end{thm}
Fix a pair $(c_1,\overline{c_2})\in \CSCS$. The two maps $f_1= \f+\ccpair$ and $\overline{f_2}=\fm \ccpair$ define an immersion $(f_1, \overline{f_2}): \widetilde S\to \mathbb G\supset \Omega^+_\rho\times \Omega^-_\rho$, and the pull-back of $\inners_{
\mathbb G}$ is the positive complex metric 
\begin{equation*}
\hpair:= -\frac 4 {(f_1-\overline{f_2})^2}df_1d\overline{f_2}.
\end{equation*}
When $c_1=c_2$, $\hpair$ is the complexification of the Riemannian metric of constant curvature $-1$ in the conformal class determined by $c_1$. In general, as a result of Theorem 6.7 in \cite{BEE}, $\hpair$ has constant curvature $-1$.  We refer to the metrics obtained in this fashion as \textbf{Bers metrics}. 

As proved in \cite{BEE}, in the smooth category, Bers metrics form a connected component of the subset of smooth complex metrics of constant curvature $-1$, and $\ccpair \mapsto \vb*h{(c_1,\overline{c_2})}$ defines a bijective correspondence between pairs of complex structures and Bers metrics. 

Finally, we record the following formula 
for future use. It follows immediately from the combination of \cite[Theorem 6.7]{BEE} and \cite[Theorem 3.17]{RT}. 

\begin{prop}\label{prop: curvature equality}
    For $h$ a Bers metric locally of the form $h=\lambda dz d\overline{w}$, $-\frac{1}{2}\Delta_h \log \lambda = -1$.
\end{prop}

\subsection{Bi-complex calculus}
\label{subsec: Bicomplex structures}

Let $J$ be a bi-complex structure on $S,$ equivalent to complex structures $c_1$ and $\overline{c_2}$ with 
local holomorphic coordinates $z$ and $\overline w$ respectively. Let $\C TS=T_{c_1}^{0,1}S\oplus T_{c_2}^{1,0}S$ be the eigenspace decomposition for $J$. 
Locally, this decomposition is $\mathrm{Span}(\partial_{\overline{z}})\oplus \mathrm{Span}(\partial_w
)$. This decomposition comes with projections $\Pi_-^J,\Pi_+^J$ of $\C TS$, mapping to $T_{c_1}^{0,1}$ and $T_{c_2}^{1,0}$ respectively. Using the same symbols $\Pi_-^J$ and $\Pi_+^J$ for the induced projections on $\C T^* S,$  they induce a further splitting on $\Omega^1(S)$, $$\Omega^1(S)=\Omega_{c_1}^{1,0}(S)\oplus \Omega_{c_2}^{0,1}(S).$$  To understand these operators, we record that for a $1$-form $\omega$, locally,
\begin{equation}\label{eqn: projected forms}
    \Pi_-^J \omega = \frac{\omega(\partial_w)}{\partial_w z}dz, \hspace{1mm} \Pi_+^J \omega = \frac{\omega(\partial_{\overline{z}})}{\partial_{\overline{z}} \overline{w}}d\overline{w},
\end{equation}
which can be verified directly.
Note that, when $c_1=c_2,$ the two splittings agree and are the ordinary splitting of $\Omega^1$ into $(1,0)$ and $(0,1)$ forms. We also extend $\Pi_-^J$ and $\Pi_+^J$ to bundle valued forms in the obvious way. As well, writing $d$ for the ordinary exterior derivative on functions, we split it as $$d=\Pi_-^J\circ d + \Pi_+^J\circ d =: \partial_{J}+\overline{\partial}_{J}.$$
Both $\partial_{J}$ and $\overline{\partial}_{J}$ extend to $\Omega^1(S)$ via the extension of $d$. 
Explicitly, in local coordinates these mappings satisfy
$$\partial_{J}(fdz) = 0, \qquad \partial_{J}(fd\overline{w})=\partial_J f\wedge d\overline{w}=d(fd\overline{w})$$ and 
$$\overline{\partial}_{J}(fd\overline{w})=0, \qquad \overline{\partial}_{J}(fdz) = \overline{\partial}_{J}f \wedge dz = d(fdz).$$
More generally, if we have a complex vector bundle $E$ over $S$ with a connection $\nabla,$ we can split $\nabla$ using $J$. Note that the pieces $\partial_2:=\Pi_-^J\circ \nabla$ and $\overline{\partial}_1=\Pi_+^J\circ \nabla$ are not del bar operators or connections in the usual sense. 
These operators extend naturally to operators $\Omega^1(E)\to \Omega^2(E)$, which we still denote by $\partial_2$ and $\overline{\partial}_1$ respectively, by 
\begin{align*}
    \partial_2(\omega\otimes s)&= \partial_{J}\omega\otimes s - \omega \wedge\partial_2s,\\
    \overline \partial_1(\omega\otimes s)&= \overline{\partial}_{J}\omega\otimes s - \omega\wedge \overline \partial_1 s, 
\end{align*}
which satisfy $\partial_2\circ \partial_2=0$ and $\overline \partial_1 \circ \overline \partial_1=0$.

Finally, we state a formula on the Laplacian of a positive complex metric $g$ with bi-complex structure $J$, which we will use in the proof of Theorem C in Section \ref{sec: proof of Theorem C}. Let $dA_g$ denote the area form of $g$. Locally, if $g=\lambda dzd\overline{w},$ then $dA_g=\frac{i}{2}\lambda dz \wedge d\overline{w}$.
\begin{prop}\label{prop: Laplace formula}
    For any $C^2$ function $f:S\to \C$, 
    \[
    \overline{\partial}_J\partial_J f= \frac i 2 \Delta_g f dA_g.
    \]
\end{prop}
We will use Proposition \ref{prop: Laplace formula} in the form of the following more specific identity: when $f$ lands in $\C^*$, then using Proposition \ref{prop: Laplace formula} and (\ref{eqn: projected forms}), we have $\frac i 2 \Delta_g \log f dA_g = \overline{\partial}_J\Big ( \frac{1}{f}\partial_J f\Big ).$  Note that, while $\log f$ is only locally defined, 
$\Delta_g \log f$ is a globally defined function.
\begin{proof}
   We work in local coordinates $z$ and $\overline{w}$ for the complex structure of $f$. 
Specifically, we identify an open subset of $S$ with the unit disk $\mathbb{D},$ with standard holomorphic coordinate $z,$ and we view $w$ as a reparametrization of $\mathbb{D},$ $w=w(z).$ Let $\mu$ denote the Beltrami form of $w$, defined by $\mu =\frac{\partial_{\overline z} w}{\partial_z w},$ which satisfies $|\mu|<1.$ We write out $$\partial_w=\partial_wz\cdot \partial_z+\partial_w\overline{z}\cdot\partial_{\overline{z}}=\partial_wz(\partial_z-\overline{\mu}\partial_{\overline{z}}).$$ In \cite[Proposition 5.2]{ElSa}, we proved that 
\begin{equation*}
    \frac{ 4}{\lambda \overline{\partial_z w}} (\partial_{\overline z } (\partial_z- \overline \mu \partial_{\overline z})).
\end{equation*}
Substituting in $\frac{\partial_w}{\partial_w z} = \partial_z - \overline{\mu}\partial_{\overline{z}},$ we obtain
$$\Delta_g = \frac{4}{\lambda\overline{\partial_z w}}\partial_{\overline{z}}\left(\frac 1 {\partial_w z} \partial_w\right).$$
We then apply the formulas of (\ref{eqn: projected forms}) for the projection operators to get the result.
\end{proof}

\section{Complex harmonic $G$-bundles}\label{sec: ch G-bundles}
Throughout this section, let $S$ be a surface, which we do not assume to be compact. 

\subsection{Complex harmonic maps}\label{sec: def of ch maps}
Let $M$ be a manifold with a complex metric $g$, let $(\mathbb{X},\inners)$ be a holomorphic Riemannian manifold as in Section \ref{section: holo Riem mflds}, and let $f:(\widetilde{M},g)\to (\mathbb{X},\inners)$ be an equivariant map. The derivative $df$ defines a section of the endomorphism bundle $\C T^*\widetilde{M}\otimes_{\mathbb C} f^* T\mathbb{X}$ defined by $df(\mathsf X+i \mathsf Y)= df(\mathsf X)+\mathbb J_{\mathbb X} (df(\mathsf Y))$ for all $\mathsf X, \mathsf Y\in \Gamma(TS)$, where $\mathbb J_{\mathbb X}$ denotes the almost-complex structure associated to $\mathbb X$. 

Denote by $\Omega^k(M,f^* T\mathbb X)$ the space of $\C$-multilinear $k$-forms on $\C TM$ with values in the bundle $f^*T\mathbb X$. Let $\nabla^f$ be the $\C$-linear extension on $\C TM$ of the pull-back of the Levi-Civita connection $D^{\mathbb X}$ of $(\mathbb{X},\inners)$ on $TM$. We see it as an operator
\[
\nabla^f\colon\Gamma( f^*{T\mathbb X})\to \Omega^1(M, f^*T\mathbb X)
\]
satisfying, $$\nabla^f_X (f^*\xi):= \nabla^f_{Re(X)}(f^*\xi)+\mathbb J\left( \nabla^f_{Im(X)}(f^*\xi)\right)= f^*\left((D^{\mathbb X}_{df(Re(X))} \xi)) +\mathbb J \left((D^{\mathbb X}_{df(Im(X))} \xi\right)\right) $$
for all $X\in \C TM$ and $\xi \in \Gamma(T\mathbb X)$. Notice that,  in general, $\nabla^f_{X}  f^*{\xi}\ne f^*(D^{\mathbb X}_{df(X)} f_*\xi)$ for $X\in \C TM$, but the equality holds when $\xi$ is the restriction of a $\mathbb J-$holomorphic vector field.

As would work in a Riemannian setting, the dual of the Levi-Civita connection $\nabla^{g*}$ together with $\nabla^f$ determine a combined connection $\widehat \nabla^f:=\nabla^{g}\otimes \nabla^f$ on the endomorphism bundle $\C T^*\widetilde{M}\otimes_{\mathbb C} f^* T\mathbb{X}$ that makes $g^*\otimes f^*\inners$ parallel. Concretely,
\[
\widehat \nabla^f _X (\alpha\otimes \xi) =  - \alpha(\nabla_X^g\ \cdot) \otimes \xi + \nabla^f_X (\alpha(\cdot)  \xi).
\]

\begin{defn}\label{harmonicmaps}
An equivariant map $f: (\widetilde{M},g)\to (\mathbb{X},\inners)$ is \textbf{complex harmonic} if
    \begin{equation}\label{hmapeqncomplex}
        \mathrm{tr}_{g} \widehat \nabla^{f} (df) =0.
    \end{equation}
\end{defn}
When $g$ is the complexification of a Riemannian metric and $f$ lands in a real submanifold on which the restriction of $\inners$ is Riemannian, this is the ordinary harmonic map equation. An alternative definition, which we won't use but might be meaningful to some, is that $f$ is complex harmonic if and only if it is a critical point for the complex Dirichlet energy $$\mathcal{E}(f) = \int_{M} e_g(f) dA_g,$$ for any $g$ as above, and where $e_g(f)$ is the complex energy density $e_g(f)=\mathrm{tr}_g f^*\inners,$ which has been descended to a function on $M$. The equivalence of definitions can be verified using the standard Euler-Lagrange calculation. 
\begin{remark}\label{remark: totally geodesic isomorphism}
    Here is a natural example. If $\mathbb{X}$ has complex dimension $n,$ then any map from an $n$-manifold $M\to \mathbb{X}$ that induces an isomorphism $\C TM\to T\mathbb{X}$ is totally geodesic in the sense of \cite[page 23]{BEE}, and equipping $M$ with the pullback metric makes the map harmonic.
\end{remark}

From now on, take $M$ to be the surface $S$. The ordinary harmonic map equation is conformally invariant in dimension $2$. Analogously, (\ref{hmapeqncomplex}) depends only on the conformal class of $g$. Thus, if we take $g$ to be a positive complex metric on $S$ giving complex structures $c_1$ and $\overline{c_2}$, then the equation (\ref{hmapeqncomplex}) depends only on $c_1$ and $\overline{c_2}$. We will write things like ``let $f:(\widetilde{S},c_1,\overline{c_2})\to (\mathbb{X},\inners)$ be complex harmonic," without specifying a metric in the associated conformal class.

In the Riemannian setting, a map $f$ from a Riemann surface is harmonic if and only if its holomorphic derivative lies in the kernel of $(\nabla^f)^{0,1}$ (see \cite{Li}). Using the complex bi-grading, we now obtain the analogous result for complex harmonic maps.

As in the real case, the pull-back connection $\nabla^f$ defines an exterior covariant differential 
\[
  d^f\colon \Omega^1(S,f^*T\mathbb X)\to \Omega^2(S,f^*T\mathbb X)\\
\]
which is defined by the property
\begin{equation}
    \label{eq: exterior covariant diff of f}
d^f(\alpha\otimes \xi)= d\alpha\otimes \xi -\alpha \wedge \nabla^f\xi
\end{equation}
for all $\alpha\in \Gamma(\C T^*M)$ and $\xi\in \Gamma(f^*T\mathbb X)$, or equivalently characterized by 
\[
(d^f\eta)(X,Y)= \nabla^f_X (\eta(Y))-\nabla^f_Y(\eta(X))-\eta([X,Y])
\]
for all $\eta\in \Omega^1(S,f^*T\mathbb X)$ and $X,Y\in \Gamma(\C TM)$.

\begin{lem}
\label{lemma: symmetry of del and del bar in harmonicity}
    $d^f(\Pi_-^J df)= -d^f(\Pi_+^J df)$, 
\end{lem}
\begin{proof}
Since $df= \Pi_-^J df+\Pi_+^Jdf$, we just need to show that $d^f (df)=0$. This just follows from \eqref{eq: exterior covariant diff of f} and the fact that $D$ is torsion-free: indeed, for all $\mathrm X,\mathrm Y\in \Gamma(TS)$,
\begin{align*}
    (d^f(df))(\mathrm X,\mathrm Y)=& \nabla^f_{\mathrm X} (df(\mathrm{Y}))- \nabla_\mathrm{Y}^f(df(\mathrm{X})) - df([\mathrm X,\mathrm Y])\\
    &=f^*\left(D^{\mathbb X}_{df(\mathrm X)} df(\mathrm Y) - D^{\mathbb X}_{df(\mathrm Y)} df(\mathrm X) - df([\mathrm X, \mathrm Y]  )   \right)=0
    \end{align*}
\end{proof}

We can now show the following characterization of complex harmonicity using $d^f$.

\begin{prop}\label{prop: harmonichol}
Let $(\mathbb{X},\inners)$ be a holomorphic Riemannian manifold, and let
    $f:(\widetilde{S},c_1,\overline{c_2})\to (\mathbb{X},\inners)$ be an equivariant map. Let $J$ be the bi-complex structure associated with the pair $\ccpair$. The following are equivalent.
    \begin{enumerate}
        \item $f$ is complex harmonic;
        \item $d^f (\Pi_-^J df) =0$;
        \item $d^f (\Pi_+^J df) =0$.
    \end{enumerate}
\end{prop}
\begin{proof}
The equivalence between $2)$ and $3)$ just follows from Lemma \ref{lemma: symmetry of del and del bar in harmonicity}. We prove that $1)$ is equivalent to $2)$.

Let $X$ and $Y$ be local nowhere zero sections of $T^{1,0}_{c_2} S$ and $T^{0,1}_{c_1} S$ respectively, namely the eigenspaces for $i$ and $-i$ of $J$. Since $g(X,X)=g(Y,Y)=0$, metric compatibility implies that, for all $Z\in \Gamma(\C TS)$, the vector fields $\nabla^g_Z X$ and $\nabla^g_Z Y$ are local sections of $T^{1,0}_{c_2} S$  and $T^{0,1}_{c_1} S$ respectively. Therefore
\[
J(\nabla^g_Z X)=i \nabla^g_Z X\qquad \qquad J(\nabla^g_Z Y)= -i \nabla^g_Z Y.
\]

Writing $\Pi_-^J df = \frac{1}{2}(df-i df\circ J)$, we see that 
\begin{equation}
    \begin{split}
      -2i(  d^f(\Pi_-^J df)) (X,Y)= &(d^f(df\circ J))(X,Y)= \nabla^f_X (df(J Y))- \nabla^f_y (df(JX))- df(J[X,Y])=\\
        &= -i \nabla^f_X (df (Y))- i\nabla^f_Y (df(X))- df(J ( \nabla^g_X Y-\nabla_Y^g X))=\\
        &=-i(\nabla^f_X (df (Y)) + \nabla^f_Y (df(X))- df( \nabla^g_X Y + \nabla_Y^g X) ).
    \end{split}
\end{equation}
Observe that, since $(X,Y)$ determines a local basis for $\C T S$, $d^f(\Pi_-^J df)=0$ if and only if $( d^f(\Pi_-^J df))(X,Y)=0$.

On the other hand, using again that $g(X,X)=g(Y,Y)=0$, computing the trace of $(\widehat \nabla^f df)$ with respect to the basis $(X,Y)$ gives
\[
tr_g (\widehat \nabla^f df)= \frac {1}{g(X,Y)}\left( \nabla^f_X (df(Y)) - df(\nabla^g_X Y) + \nabla_Y^f(df(X) -df(\nabla^g_Y X) \right).
\]
We conclude that $1)$ is equivalent to $d^f(\Pi_-^J df)=0$.
\end{proof}

Continuing in our setting, $f^*\inners$ descends to $S$. By analogy with the Riemannian case, we say that $f$ is \textbf{conformal} if the pullback $f^*\inners$ is a non-zero multiple of $g$ at every point. This occurs if and only if $f^*\inners$ is non-degenerate and, splitting $df=\Pi_-^Jd f+ \Pi_+^Jd f$, if $q_1:=\langle \Pi_-^Jd f, \Pi_-^Jd f\rangle$ and $\overline{q_2}:=\langle  \Pi_+^Jd f,  \Pi_+^Jdf\rangle$ are identically zero. In general, Proposition \ref{prop: harmonichol} shows that $q_1$ and $q_2$ are holomorphic, and we refer to them as the two Hopf differentials.



We end this subsection with a motivating example. Consider $\mathbb{G}=\mathbb{X}_2= \PSL(2,\C)/\mathrm{PSO}(2,\C)$. By Remark \ref{remark: totally geodesic isomorphism}, the maps $f:(\widetilde{S},c_1,\overline{c_2})\to \mathbb{G}$ produced from Bers' Simultaneous Uniformization Theorem are conformal and complex harmonic. They correspond to taking $c_1$ and $\overline{c_2}$ arbitrary and $q_1=q_2=0$ in Theorem C, solving (\ref{eq: CAT}) with $U=1$, and then lifting the corresponding complex harmonic maps. From these observations, we can view Bers' simultaenous uniformization as an analogue of Labourie's parametrizations: every quasi-Fuchsian representation carries a special equivariant conformal complex harmonic map into $\mathbb{G}$ and these maps lead to a parametrization of the space of quasi-Fuchsian representations. 

\subsection{Complex harmonic $G$-bundles}\label{sec: building ch G-bundles}  
In this paper, we take the target space for our complex harmonic maps to be a holomorphic Riemannian symmetric space $(G^{\C}/K^{\C},\nu)$. As we saw, ordinary harmonic maps to Riemannian symmetric spaces of non-compact type give rise to harmonic $G$-bundles and $G$-Higgs bundles.  When we adapt to $(G^{\C}/K^{\C},\nu)$, we find complex harmonic $G$-bundles. We resume all notations from Sections \ref{sec: generalities} and \ref{sec: ordinary harmonic maps}.

\subsubsection{Definition}
The definition of a complex harmonic $G$-bundle requires a bit of setup. To keep things tidier, we first state this definition independently of complex harmonic maps. We then explain the relation with complex harmonic maps in Section \ref{subsubsec: equivalence} below.  

Let $c_1$ and $\overline{c_2}$ be oppositely oriented complex structures over $S$. Fix a maximal compact subgroup $K\subset G$ and let $P_{K^{\C}}$ be a principal $K^{\C}$-bundle over $S$ with a connection $A_{P_{K^{\C}}}$. Recalling the splitting $\mathfrak g^{\C}=\mathfrak k^{\C}\oplus \mathfrak p^{\C}$, as in Section \ref{sec: ordinary harmonic maps}, we set $\mathrm{ad}\p^{\C}:=P_{K^{\C}}\times_{\mathrm{Ad}|_{K^{\C}}}\p^{\C}.$ Via the Koszul-Malgrange theorem, $A_{P_{K^{\C}}}$ produces two complex structures on $P_{K^{\C}}$, one so that $P_{K^{\C}}\to (S,c_1)$ is holomorphic, and the other so that $P_{K^{\C}}\to (\overline{S},\overline{c_2})$ is holomorphic. 
\begin{defn}
\label{def: complex harmonic G-bundle}
    A \textbf{complex harmonic $G$-bundle} over $(S,c_1,\overline{c_2})$ is data $(P_{K^{\C}},A_{P_{K^{\C}}},\phi_1,\overline{\phi_2}),$ where $K$ is a maximal compact subgroup of $G,$ $P_{K^{\C}}$ is a principal $K^{\C}$-bundle over $S$ with connection $A_{P_{K^{\C}}}$, and $\phi_1\in \Omega_{c_1}^{1,0}(\mathrm{ad}\p^{\C})$ and $\overline{\phi_2}\in \Omega_{\overline{c_2}}^{0,1}(\mathrm{ad}\p^{\C})$ are holomorphic for the complex structures induced by $(A_{P_{K^{\C}}},c_1)$ and $(A_{P_{K^{\C}}},\overline{c_2})$ respectively, such that the connection on the $G^{\C}$-bundle $P_{K^{\C}}\times_{K^{\C}} G^{\C}$ given by $$A_{P_{K^{\C}}} + \phi_1+\overline{\phi_2}$$ is flat.
\end{defn}
For the most part, we don't specify the choice of $K$. By analogy with ordinary harmonic bundles, we refer to $\phi_1$ and $\overline{\phi_2}$ as Higgs fields. The flatness condition can be expressed
$$F(A_{P_{K^{\C}}}) + [\phi_1,\overline{\phi_2}]=0,$$ resembling Hitchin's self-duality equations. Note that $A_{P_{K^{\C}}}$ induces a connection $\nabla$ on $\mathrm{ad}\p^{\C}$. If $J$ is the bi-complex structure of $c_1$ and $\overline{c_2}$, with projections $\Pi_-^J$ and $\Pi_+^J$, we set $\partial_2=\Pi_-^J\circ \nabla$ and $\overline{\partial}_1=\Pi_+^J\circ \nabla$.
Note the lemma below, which follows immediately from the formulas (\ref{eqn: projected forms}).
\begin{lem}\label{lem: holomorphicity equivalence}
    A section $X$ of $\mathrm{ad}\p^{\C}$ satisfies $\overline{\partial_1}X=0$ if and only if it is holomorphic for $(A_{P_{K^{\C}}},c_1)$ and a section $Y$ satisfies ${\partial}_2 Y=0$ if and only if it is holomorphic for $(A_{P_{K^{\C}}},\overline{c_2}).$
\end{lem}


There is a linear version of complex harmonic $G$-bundles, which could be helpful to have in mind. 
We define a complex harmonic (vector) bundle to be data $(E,Q,\nabla,\phi_1,\overline{\phi_2})$, where $E$ is a complex vector bundle over $S$ with a bilinear pairing $Q$, $\nabla$ is a $Q$-compatible connection, $\phi_1\in \Omega_{c_1}^{1,0}(\mathrm{End}(E))$, $\overline{\phi_2}\in \Omega_{c_2}^{0,1}(\mathrm{End}(E))$ are $\textrm{End}(E)$-valued forms satisfying $(\Pi_-^J\circ \nabla)\overline{\phi_2}=(\Pi_+^J\circ \nabla)\phi_1=0,$ and the connection $$\nabla+\phi_1+\overline{\phi_2}$$ is flat.
Any linear representation $T:G\to \mathrm{GL}(V)$, with $V$ a real vector space, turns any complex harmonic $G$-bundle $(P_{K^{\C}},A_{K^{\C}},\phi_1,\overline{\phi_2})$ into a complex harmonic bundle.

\subsubsection{Equivalence with complex harmonic maps}\label{subsubsec: equivalence}
The group $G^{\C}$ acts on the set of pairs $(\rho,f)$, where $\rho:\pi_1(S)\to G^{\C}$ is a representation and $f$ is a $\rho$-equivariant harmonic map to $(G^{\C}/K^{\C},\nu)$, via $g\cdot (\rho,f)=(g\rho g^{-1},g\cdot f).$ An isomorphism of complex harmonic $G$-bundles is an isomorphism of the underlying principal bundles that identifies the connections and the Higgs fields. Note that we're allowing isomorphisms to change the choice of maximal compact subgroup of $G$. We establish the following. 
\begin{thm}\label{prop: equivalence of maps and bundles}
    A pair $(\rho,f)$ as above, up to the $G^{\C}$-action, is equivalent to an isomorphism class of complex harmonic $G$-bundles.
\end{thm}
As in the real case, we omit the $\nu,$ since harmonicity does not depend on the specific choice of $\nu.$ One key difference compared to the real case: in Section \ref{sec: ordinary harmonic maps}, we used the harmonic map to get a $K$-bundle $Q_f$, and then we complexified to get a $K^{\C}$-bundle $P_f$ and a Higgs field. In our complex setting, there is no complexification step: the harmonic map gives a $K^{\C}$-bundle automatically. As above, let $c_1$ and $\overline{c_2}$ be complex structures on $S$ with bi-complex structure $J$.
\begin{proof}
   Starting the same way as in Section \ref{sec: ordinary harmonic maps}, let $\rho:\pi_1(S)\to G^{\C}$ be a representation and $f:(\widetilde{S},c_1,\overline{c_2})\to G^{\C}/K^{\C}$ a $\rho$-equivariant map. We form the associated principal $G^{\C}$-bundle $P_{\rho}=\widetilde{S}\times_\rho G^{\C}$ over $S$, which comes with a flat connection $A$. The map $f$ is equivalent to a reduction of structure group to $K^{\C},$ $P_{f}\subset P_{\rho}$, which induces a linear involution of the adjoint bundle $\sigma_f:\mathrm{ad}P_{\rho}\to \mathrm{ad}P_{\rho},$ $\mathrm{ad}P_{\rho}=P_{\rho}\times_{G^{\C}} \g^{\C}$.
As in the case of ordinary harmonic maps, we use $f$ to pull back the principal connection $A_{K^{\C}}$ from the bundle $G^{\C}\to G^{\C}/K^{\C}$ to a connection $A_{f}$ on $P_{f}.$
Writing $\g^{\C}=\p^{\C}\oplus \kk^{\C},$ $A_{K^{\C}}$ induces an affine connection $\nabla$ on $\mathrm{ad}\p_f^{\C}:=P_{f}\times_{\mathrm{Ad}|_{K^{\C}}}\p^{\C}.$ Using the Maurer-Cartan form, the last bundle identifies with the pullback bundle $f^*T(G^{\C}/K^{\C})$, and the pullback of the Levi-Civita connection with $\nabla.$ Furthermore, under this isomorphism, $df$ is a $1$-form valued in $\mathrm{ad}\p_f^{\C}$, and we denote it by $\psi_f.$

 Decomposing $\psi_f$ and $\nabla$ according to the splitting $\Omega_{c_1}^{1,0}(\mathrm{ad}\p_f^{\C})\oplus \Omega_{c_2}^{0,1}(\mathrm{ad}\p_f^{\C})$ as $\psi_f=\phi_1^f+\overline{\phi_2}^f$ and $\nabla=\overline{\partial}_1+ \partial_2$ respectively, by Proposition \ref{prop: harmonichol}, the harmonic map equation (\ref{hmapeqncomplex}) translates to the two equations
\begin{equation*}
    \overline{\partial}_1 \phi_1^f = 0, \hspace{1mm} \partial_2\overline{\phi_2}^f=0.
\end{equation*}
By Lemma \ref{lem: holomorphicity equivalence}, this implies that $\phi_1$ is holomorphic for $(A_f,c_1)$ and $\overline{\phi_2}$ is holomorphic for $(A_f,\overline{c_2})$.

Collecting our data, if $f$ is a complex harmonic map, then  $(P_f,A_f,\phi_1^f,\overline{\phi_2}^f)$ is a complex harmonic $G$-bundle. If we act on $(\rho,f)$ by an element $g$ of $G^{\C},$ then multiplication by $g$ induces an isomorphism between $P_f\subset P_\rho$ and $P_{g\cdot f}\subset P_{g\rho g^{-1}},$ and it is clear that this isomorphism identifies all of the complex harmonic $G$-bundle data.

For the other direction, totally analogous to the case of ordinary harmonic $G$-bundles, once we pick a point on $S,$ a complex harmonic $G$-bundle determines a representation to $G^{\C}$ (via the flat connection) and an equivariant complex harmonic map (via $P_{K^{\C}}$ sitting inside $P_{K^{\C}}\times_{K^{\C}} G^{\C}$), in a way that undoes our construction above. As in the real case, changing the point conjugates the representation and translates the harmonic map.
\end{proof}

Since an ordinary harmonic map gives a complex harmonic map, an ordinary harmonic $G$-bundle $(P_{K^{\C}},\phi,P_K)$ determines a complex harmonic $G$-bundle. To see how the various components interact, let $P_G=P_K\times_K G$, $P_{G^{\C}}=P_{K^{\C}}\times_{K^{\C}}G^{\C}$, and let $\sigma:\mathrm{ad}P_G\to \mathrm{ad}P_G$ be the linear involution inherited through the inclusion $P_K\subset P_G$. Then the bilinear extension $\sigma: \mathrm{ad}P_{G^{\C}}\to \mathrm{ad}P_{G^{\C}}$ is an artefact of the reduction of $P_{G^{\C}}$ to $K^{\C}$ and the sesquilinear extension $\widehat{\sigma}: \mathrm{ad}P_{G^{\C}}\to \overline{\mathrm{ad}P_{G^{\C}}}$ is the solution to the self-duality equations. The Chern connection $A_{P_K}$ of $P_K$ extends to $P_{K^{\C}}$, and $(P_{K^{\C}},A_{P_K},\phi,-\widehat{\sigma}(\phi))$ is a complex harmonic $G$-bundle.

\subsubsection{Bi-Hitchin fibration}

As in Section \ref{sec: Labourie paper}, let $\mathcal{O}(\g)^{G}$ be the algebra of $\mathrm{Ad}(G)$-invariant polynomials on $\g$, and let $\{p_1,\dots, p_l\}$ be a minimal generating set comprised of homogeneous elements with degrees $m_1,\dots, m_l$, ordered so that $m_i\leq m_{i+1}.$ 
\begin{defn}
    We define the \textbf{bi-Hitchin base} to be $$bH(c_1,\overline{c_2},G)=\bigoplus_{i=1}^l H^0(S,\mathcal{K}_{c_1}^{m_i})\oplus \bigoplus_{i=1}^l H^0(S,\mathcal{K}_{\overline{c_2}}^{m_i}).$$ The \textbf{bi-Hitchin fibration} takes a complex harmonic $G$-bundle $(P_{K^{\C}},A_{K^{\C}},\phi_1,\overline{\phi_2})$ to $$(p_1(\phi_1),\dots, p_l(\phi_1),p_1(\overline{\phi_2}),\dots, p_l(\overline{\phi_2}))\in bH(c_1,\overline{c_2},G).$$
\end{defn}
We sometimes write points in the bi-Hitchin base as $(q^1,\overline{q}^2)\in bH(c_1,\overline{c_2},G),$ $q^1=(q_1^1,\dots, q_l^1),$ $\overline{q}^2=(\overline{q}_1^2,\dots, \overline{q}_l^2).$ When only the $l$-differentials are non-zero (a case we will consider often), we might write $q^1=(0,\dots, 0, q_1),$ $\overline{q}^2=(0,\dots, 0, \overline{q_2}).$ The fact that the differentials indeed land in the bi-Hitchin base is justified by the proposition below.
\begin{prop}\label{prop: diffs are holomorphic}
    Each differential $p_i(\phi_1)$ is $c_1$-holomorphic, and each differential $p_i(\overline{\phi_2})$ is $\overline{c_2}$-holomorphic.
\end{prop}
\begin{proof}
 Identifying $\mathcal{O}(\p^{\C})$ with the symmetric algebra of the dual space $(\p^{\C})^*$, any $p_i\in \mathcal{O}(\p^{\C})^{K^{\C}}$ is an $i$-multilinear map, and hence it takes holomorphic sections to holomorphic sections.
\end{proof}
Analogous to the real case, a complex harmonic map to $G^{\C}/K^{\C}$ with Higgs fields $\phi_1$ and $\overline{\phi_2}$ is conformal if and only if $p_1(\phi_1)=p_1(\overline{\phi_2})=0.$

\subsection{Excursion on Lie theory}\label{sec: excursion}
Looking toward Theorem C, we recall some Lie theory. For a source on all of the content below, see \cite[Chapter 3]{OV}. We begin with root systems. Let $\g^{\C}$ be a semisimple complex Lie algebra and let $\mathfrak{h}^{\C}\subset \g^{\C}$ be a Cartan algebra. Simultaneously diagonalizing the adjoint action of $\mathfrak{h}^{\C}$ on $\g^{\C}$ yields the root space decomposition $$\g^{\C}=\oplus_{\alpha\in \Delta} \g_{\alpha},$$ where $\alpha\in (\mathfrak{h}^{\C})^*$ is a root, i.e., a weight of the adjoint representation restricted to $\mathfrak{h}^{\C}$, and $\g_{\alpha}$ is the root space $$\g_{\alpha}=\{X\in\g^{\C}: \text{ for all } h\in \mathfrak{h}^{\C}, [h,X]=\alpha(h)X\}.$$
 The roots of the Lie algebra form a root system $\Delta$ on $((\mathfrak{h}^{\C})^*,\nu)$. Starting from $\Delta,$ we can choose a set of positive roots $\Delta^+$, by which we mean a subset such that
\begin{itemize}
    \item for each $\alpha\in \Delta,$ exactly one of $\alpha$ and $-\alpha$ is contained in $\Delta^+$, and
    \item if $\alpha,\beta\in \Delta^+$ and $\alpha+\beta$ is a root, then $\alpha+\beta\in \Delta^+.$
\end{itemize}
The choice of positive roots determines a subset $\Pi$, the set of simple roots, which is characterized by being linearly independent and by the property that every element in $\Delta^+$ is a non-negative linear combination of elements in $\Pi$.

Given a root $\beta = \sum_{\alpha \in \Pi} m_\alpha \alpha,$ $m_\alpha\in\mathbb{Z}$, the height of $\beta$ is $\sum_{\alpha \in \Pi} m_\alpha.$ The highest root $\delta$ is the positive root with the largest height. We will write $\delta=\sum_{\alpha \in \Pi} n_\alpha \alpha$. It is well known that the height of $\delta$ is $d_G-1,$ where we recall that $d_G$ is our notation for the Coxeter number of $G$ (see \cite[Chapter 3.6]{OV}). The extended simple root system $\mathcal{Z}$ is $\Pi\cup\{-\delta\}.$

For every root $\alpha$ we have the vector $h_\alpha\in \mathfrak{h}^{\C}$, defined by the relation that, for all roots $\beta$, $\beta(h_\alpha) = 2\frac{\nu(\alpha,\beta)}{\nu(\alpha,\alpha)}$, where $\nu$ here is the dualized Killing form on $\g^*.$ We call the numbers $a_{\beta\alpha}= \beta(h_\alpha)$, for $\alpha,\beta\in \mathcal{Z}$, the affine Cartan numbers. One can choose root vectors $x_\alpha \in \g_\alpha$, which, together with the $h_\beta$'s, $\beta\in \Pi,$ form a \textbf{Chevalley basis} for $\g$, $$\{x_\alpha,x_{-\alpha},h_\beta:\alpha\in \Delta^+, \beta\in \Pi\},$$ characterized by $[x_\alpha,x_{-\alpha}]=-h_\alpha$ and $[x_\alpha,x_\beta]=N_{\alpha,\beta}x_{\alpha+\beta}$, for some integers $N_{\alpha,\beta}$ (see \cite[Chapter 3, Theorem 1.19]{OV}). From here on out, we fix the above Chevalley basis. We take note of the following fact about $\mathcal{Z}$.
\begin{lem}\label{lem: commutators}
    Let $\epsilon_{\alpha\beta}=1$ for $\alpha=\beta$ and $0$ otherwise. If $\alpha,\beta\in\mathcal{Z}$, then, in a Chevalley basis as above, $[x_\alpha,x_{-\beta}]=-\epsilon_{\alpha\beta} h_\alpha.$  
\end{lem}
\begin{proof}
This is well known and quite simple (the same proof can be found in \cite{ST}).
The commutator $[\g_\alpha,\g_\beta]$ is non-zero if and only if $\alpha+\beta$ is a root. If $\alpha$ and $\beta$ are simple, then $\alpha-\beta$ cannot be a root, since if it is positive it contradicts the simplicity of $\beta$, and if it is negative then it contradicts the nsimplicity of $\alpha$. For any positive root $\alpha$, $\alpha+\delta$ is not a root, since the height would be larger than that of $\delta$.
\end{proof}

\subsection{Strategy for bi-Hitchin section}\label{sec: bi-Hitchin}
Now we assume, as in \cite{Hi}, that $G$ is a split real simple Lie group of adjoint type.
{In this subsection, the goal is to establish Theorem \ref{thm: biHitchin}, which allows us to produce complex harmonic $G$-bundles out of a point in the bi-Hitchin base and a solution to an equation. 
Throughout, notations $K,\g,\p,\kk$, etc., are as in Section \ref{sec: generalities}. 

\subsubsection{The Hitchin section}
We first recall the Hitchin section over a fixed Riemann surface $(S,c)$, which is used in Theorem \ref{thm: biHitchin}. Starting as in \cite[\S 5]{Hi}, let $e\in \g^{\C}$ be a principal nilpotent element. One characterization of the word ``principal" here is that the centralizer of $e$ in $\g^{\C}$ has dimension equal to the rank $l$. By the Jacobson-Morozov theorem, we can complete $e$ to a principal $\mathfrak{sl}(2,\C)$-triple $\mathfrak{s}^{\C}=\mathrm{Span}_{\C}\{x,e,\widetilde{e}\}\subset \g^{\C}$, where $x$ is semisimple, $\widetilde{e}$ is nilpotent, and we have the relations $$[x,e]=e, \hspace{1mm} [x,\widetilde{e}]=-\widetilde{e}, \hspace{1mm} [e,\widetilde{e}]=x.$$ 
 The inclusion from $\mathfrak{s}^{\C}$ to $\g^{\C}$ exponentiates to an inclusion $\iota_G:\mathrm{PSL}(2,\C)\to G^{\C}$; the $\iota_G$ from Section \ref{sec: definition teichmuller} can be constructed in this fashion, so we use the same notation. 

In the proof of Theorem C, it will be helpful to have an explicit $\mathfrak{s}^{\C}$. We choose, using the Cartan subalgebra $\mathfrak{h}^{\C}\subset \g^{\C}$ and the Chevalley base from the previous subsection (and as in \cite{Baraglia2010CyclicHB}), $$x=\frac{1}{2}\sum_{\alpha\in \Delta^+} h_\alpha = \sum_{\alpha\in \Pi} r_\alpha h_\alpha,$$ for some positive half integers $r_\alpha$ that we call the coefficients of $\Pi$, and $$e=\sum_{\alpha\in \Pi}r_\alpha^{1/2}x_{\alpha}, \hspace{1mm} \widetilde{e}=\sum_{\alpha\in \Pi}r_\alpha^{1/2}x_{-\alpha}.$$



To construct a family of $G$-Higgs bundles, Hitchin specifies an underlying holomorphic $K^{\C}$-bundle. Hitchin first builds a $G^{\C}$-bundle $P_{G^{\C}}$ and then reduces the structure group. Choose a square root $\mathcal{K}_{c}^{1/2}$ of the canonical bundle $\mathcal{K}_{c},$ and define $P_{\mathrm{PSL}(2,\C)}$ to be the principal $\mathrm{PSL}(2,\C)$-bundle obtained by taking the frame bundle of $\mathcal{K}_{c}^{1/2}\oplus \mathcal{K}_{c}^{-1/2}.$  We then define the principal $G^{\C}$-bundle $P_{G^{\C}}$ over $S$ by $P_{G^{\C}}=P_{\mathrm{PSL}(2,\C)}\times_{\iota_G} G^{\C}$. Hitchin reduces the structure group to $K^{\C}$ by building an involution on the adjoint bundle. To describe that involution, we need a description of the adjoint bundle of $P_{G^{\C}}$. Toward this, it can be checked that for any root $\beta$ of height $m$ and $y\in \g_{\beta}$, $[x,y]=m y$ (see \cite[Section 2]{Baraglia2010CyclicHB} for justification). Hence, the weight space decomposition of $\g^{\C}$ for the adjoint action of the algebra generated by $x$ is 
\begin{equation*}
    \g^{\C}=\oplus_{m=-M}^{M}\g_m,
\end{equation*}
where $\g_m=\{y\in \g: [x,y]=my\}$ is the direct sum of all of the root spaces for $\mathfrak{h}^{\C}$ of height $m$ (so $M=d-1$, where $d=d_G$).
It is explained in \cite{Hi} that
\begin{equation}\label{eq: ad P}
    \mathrm{ad}P_{G^{\C}} =\oplus_{m=-M}^M\g_m\otimes \mathcal{K}_{c}^m
\end{equation}
as Lie algebra bundles (note $\g_m\otimes \mathcal{K}_{c}^m:=(S\times \g_m)\otimes \mathcal{K}_c^m$ as a tensor product of bundles, but we suppress the notation). Here, the Lie bracket on a fiber of the bundle on the right hand side is given by, on pure tensors, 
\begin{equation}\label{eq: the lie bracket}
    [f_m\otimes \beta_m, f_n\otimes \beta_n]=[f_m,f_n]\otimes (\beta_m  \otimes \beta_n).
\end{equation}
To define Hitchin's involution on $\mathrm{ad}P_{G^{\C}},$ we consider another decomposition of $\g^{\C}:$ the decomposition under the adjoint representation of $\mathfrak{s}^{\C}$, $$\g^{\C}=\oplus_{i=1}^l V_i,$$ arranged so that $\dim V_i\leq \dim V_{i+1}$ for all $i=1,\dots, l-1$. Note that the number of summands is indeed $l$: for every $i,$ choose a highest weight vector $e_i\in V_i$ for the adjoint action of the algebra generated by $x$. For $i=1,$ we specifically choose $e_1=e$. As $[e,e_i]\in V_i$ would lie in a higher weight space, $[e,e_i]=0,$ and since $e_i$ is principal it follows that there are exactly $l$ $e_i$'s and that they span the centralizer of $e$. Among other facts of interest from \cite[\S 4]{Hi}, we recall here that $e_i\in \mathfrak g_{m_i-1}$ ($m_i$ as in Sections \ref{sec: Labourie paper} and \ref{sec: ch G-bundles}), $ad^k(\widetilde e)(e_i)\in \g_{m_i-1-k}$, and that $V_i$ has dimension $2m_i-1$ and is generated by $\{ad^k(\widetilde e)(e_i):k=0,\dots, 2m_i-2\}$. In \cite[\S 6]{Hi}, Hitchin shows that there is a Lie algebra involution $\sigma_0:\g^{\C}\to \g^{\C}$ uniquely characterized by $$\sigma_0(e_i)=-e_i, \hspace{1mm} \sigma_0(\widetilde{e})=-\widetilde{e}$$ and whose fixed point set is a copy of $K^{\C}$. To see why these properties define a Lie algebra involution, the Lie algebra involution property locks the formula, for every $k,$ $$\sigma_0(\mathrm{ad}(\widetilde{e})^k(e_i))=(-1)^{k+1}\mathrm{ad}(\widetilde{e})^k(e_i),$$ which thus defines $\sigma_0$ on every $V_i.$

The involution $\sigma_0$ naturally gives rise to an involution $\sigma$ on $\mathrm{ad}P_{G^{\C}},$ defined on each component in the decomposition (\ref{eq: ad P}) via $\sigma(X\otimes dz^m)=\sigma_0(X)\otimes dz^m.$ Since $G$ is of adjoint type, the involution $\sigma$ is equivalent to a reduction of structure group $P_{K^{\C}}\subset P_{G^{\C}}$ to $K^{\C}.$ Since it commutes with the del bar operator of $\mathrm{ad}P_{G^{\C}},$ the reduction is holomorphic. 
\begin{defn}
       \textbf{Hitchin's section} from the Hitchin base $H(c,G)$ to the set of isomorphism classes of $G$-Higgs bundles on $(S,c)$ is defined by associating $(q^1,\dots,q^l)$ to the class of $(P_{K^{\C}},\phi_1),$ where $$\phi_1=\widetilde{e}+q_{1}^1e_1+\dots + q_{l}^1e_l.$$
\end{defn}

Above, we view $\widetilde{e}$ as a section of $(\g_{-1}\otimes \mathcal{K}_{c_1}^{-1})\otimes \mathcal{K}_{c_1}$ and similar for the others. It is routine to check holomorphicity and that $\sigma(\phi)=-\phi,$ so that $\phi$ defines a section of $\mathrm{ad}\p^{\C}.$ 



\subsubsection{Theorem \ref{thm: biHitchin}}

Returning to our main task, we fix a bi-Hitchin basepoint $$(q^1,\overline{q}^2) \in bH(c_1,\overline{c_2},G)$$ and try to build a complex harmonic $G$-bundle.
For motivation, consider the case where $S$ is closed, $c_1=c_2$, and $q^1=q^2$ (the real locus). Let $(P_{K^{\C}},\phi_1)$ be the (stable) $G$-Higgs bundle over $(S,c_1)$ in the Hitchin section for $q^1$, and let $(P_{K^{\C}},\overline{\psi_2})$ over $(S,\overline{c_2})$ be Hitchin's $G$-Higgs bundle for $\overline{q}^2$. We're using the same symbol $P_{K^{\C}}$, with the understanding that the complex structures are different. Let $\overline{\partial}_1$ and $\partial_2$ be the del-bar operators for $(P_{K^{\C}},\phi_1)$ and $(P_{K^{\C}},\overline{\psi_2})$ respectively.
By Theorem \ref{thm: Hitchin simpsons}, there exists a unique reduction of structure to $K,$ $P_K\subset P_{K^{\C}}$, inducing a Lie algebra bundle involution $$\mathcal{I}:\mathrm{ad}P_{G^{\C}}\to \overline{\mathrm{ad}P_{G^{\C}}}$$  (in the notation of Section \ref{sec: ordinary harmonic maps}, $\mathcal{I}=\widehat{\sigma})$ such that if 
$\partial_2^{\mathcal{I}}= {\mathcal{I}}^{-1}\circ \partial_2\circ \mathcal{I}$, $\overline{\phi_2}^{\mathcal{I}}=-{\mathcal{I}}^{-1}\circ \overline{\psi}_2,$ and $A^{\mathcal{I}}$ is the principal connection induced by the affine connection $\partial_2^{\mathcal{I}}+\overline{\partial}_1$ (it is clear that this affine connection preserves the Lie bracket), then $$F(A^{\mathcal{I}}) +[\phi_1, \overline{\phi_2}^{\mathcal{I}}]=0,$$ and moreover $(P_{K^{\C}},\phi_1,P_K)$ is a harmonic $G$-bundle. Furthermore, by the discussion in Section \ref{sec: building ch G-bundles}, $(P_{K^{\C}},A^{\mathcal{I}},\phi_1,\overline{\phi_2}^{\mathcal{I}})$ is a complex harmonic $G$-bundle. 

{In general, let $\phi_1$ and $\overline{\psi}_2$ be as above, and let $J$ be the bi-complex structure for $(c_1,\overline{c_2})$. We build an operator $\overline{\partial}_1$ on $\oplus_{m=-M}^{M}\g_m\otimes \mathcal{K}_{c_1}^m$ as follows: complete the del bar operator of Hitchin's holomorphic structure to a connection and then project using $\Pi_+^J$ (note: after projecting, the result is independent of the choice of connection). In the same fashion, we build an operator $\partial_2$ on $\oplus_{m=-M}^{M}\overline{\g_{m}}\otimes \mathcal{K}_{\overline{c_2}}^{m}$. We say that an isomorphism of Lie algebras $\mathcal I$ preserves polynomials if for all adjoint-invariant polynomials $p$, $p(X)=p(-\mathcal{I}(X)).$ The notion naturally extends to Lie algebra bundles.}
{
\begin{thm}\label{thm: biHitchin}
Let $(q^1,\overline{q}^2)\in bH(c_1,\overline{c_2},G)$ and let $\overline{\partial}_1$, $\partial_2$, $\phi_1$, $\overline{\psi_2}$ be as above. Suppose there exists a Lie algebra bundle isomorphism $$\mathcal{I}:\bigoplus_{m=-M}^{M}\g_m\otimes \mathcal{K}_{c_1}^m\to \bigoplus_{m=-M}^{M}\overline{\g_{m}}\otimes \mathcal{K}_{\overline{c_2}}^{m},$$ compatible with the reduction of $P_{G^{\C}}$ to $K^{\C}$, i.e., so that $\mathcal{I}\circ \sigma=\overline{\sigma}\circ \mathcal{I}$, preserving polynomials, and such that, with
$\partial_2^{\mathcal{I}}= {\mathcal{I}}^{-1}\circ \partial_2\circ \mathcal{I}$, $\overline{\phi_2}^{\mathcal{I}}=-{\mathcal{I}}^{-1}\circ \overline{\psi}_2$, and $A^{\mathcal{I}}$ as above, \begin{equation}\label{eq: bi-Hitchin flatness}
    F(A^{\mathcal{I}}) + [\phi_1,\overline{\phi_2}^{\mathcal{I}}]=0.
\end{equation} Then $(P_{K^{\C}},A^{\mathcal{I}},\phi_1,\overline{\phi_2}^{\mathcal{I}})$ is a complex harmonic $G$-bundle with bi-Hitchin basepoint $(q^1,\overline{q}^2).$
\end{thm}
\begin{proof}
    For any $\mathcal{I}$, it follows from the relation $\partial_2\overline{\psi_2}$ and the definitions that $\partial_2^{\mathcal{I}}\overline{\phi_2}^{\mathcal{I}}=0,$ so the holomorphicity conditions are satisfied. Since $\mathcal{I}$ is compatible, $A^{\mathcal{I}}$ is a connection on the $K^{\C}$-reduction $P_{K^{\C}}$. The equation (\ref{eq: bi-Hitchin flatness}) then says that $A_{P_{K^{\C}}}+\phi_1+\overline{\phi_2}^{\mathcal{I}}$ is flat, and hence $(P_{K^{\C}},A^{\mathcal{I}},\phi_1,\overline{\phi_2}^{\mathcal{I}})$ is a complex harmonic $G$-bundle. Since $\mathcal{I}$ preserves polynomials, the bi-Hitchin basepoint is $(q^1,\overline{q}^2).$
\end{proof}}

Stepping back for a moment, we know that, when $S$ is closed, it is possible to solve (\ref{eq: bi-Hitchin flatness}) on the entire real locus. In Section \ref{sec: existence results} below, we find solutions far away from the real locus. 
\begin{remark}
    When $q_1^1=\overline{q}_1^2=0,$ the complex harmonic maps constructed {from Theorem \ref{thm: biHitchin} are conformal}. Indeed, since $q_1^1=\overline{q}_1^2=0,$ conformality is equivalent to the nowhere vanishing of the derivative. The presence of the $\widetilde{e}$ term in $\phi_1$ shows that this holds.
\end{remark}

Preparing for later subsections, we set a terminology. For a number of computations, we will pick a local trivialization of $P_{G^{\C}}\to S$ by a local holomorphic section so that $P_{G^{\C}}\to (S,c_1)$ is holomorphic. We say that the trivialization is \textbf{holomorphic} for $P_{G^{\C}}\to (S,c_1).$ 

Finally, while we'll mostly work with $\oplus_{m=-M}^{M}\g_m\otimes \mathcal{K}_{c_1}^m$, note that we can just as well push $\overline{\partial}_1$ and $\phi_1$ via $\mathcal{I}$ to objects $\overline{\partial}_1^{\mathcal{I}}$ and $\psi_1^{\mathcal{I}}$ on  $\oplus_{m=-M}^{M}\overline{\g_{m}}\otimes \mathcal{K}_{\overline{c_2}}^{m}$, so that $P_{K^{\C}},\overline{\partial}_1^{\mathcal{I}}+\partial_2,$ $\psi_1^{\mathcal{I}}$, and $\overline{\psi_2}$ determine a complex harmonic bundle. Working in trivializations that are holomorphic for $P_{G^{\C}}\to (S,\overline{c_2})$, if we complex conjugate all of our data, namely, $c_1,\overline{c_2}$, the trivializations, the Higgs fields, the isomorphism $\mathcal{I}$, etc., then we arrive at the expression for a complex harmonic $G$-bundle over $(S,c_2,\overline{c_1})$ with bi-Hitchin basepoint $(q^2,\overline{q}^1)$.

\subsection{Cyclic locus and complex affine Toda equations}\label{sec: proof of Theorem C}
We continue as in the previous subsection, keeping all of the same notations from Section \ref{sec: excursion} onwards, and in particular working with the fixed Chevalley base $\{x_\alpha,x_{-\alpha},h_\beta: \alpha\in \Delta^+, \beta \in \Pi\}$. We point out that 
we could have chosen $e_l$ so that $e_l=x_\delta$, and in this subsection we specify that we've done so and we write $x_\delta$ in place of $e_l$. 


Theorem \ref{thm: biHitchin} can be applied to any point in the bi-Hitchin base. Here is the point at which we specialize and set our bi-Hitchin basepoint to be $$(q^1,\overline{q}^2)=((0,\dots, 0, q_1),(0,\dots, 0, \overline{q_2})).$$ Now $\phi_1$ takes a simpler form: $$\phi_1=\widetilde{e}+q_1x_\delta.$$ Our main goal is Theorem C, which shows that, in this case, (\ref{eq: bi-Hitchin flatness}) can be reduced to the complex affine Toda equations, defined shortly below. We just need one more ingredient: there is a permutation of the simple root set $\xi: \Pi\to \Pi$, which can be represented by an automorphism of the Dynkin diagram, such that the Lie algebra involution $\sigma_0$ satisfies $\sigma_0(h_{\alpha})=h_{\xi(\alpha)}$ and $\sigma_0(h_{-\delta})=h_{-\delta}$. The permutation $\xi$ is in fact trivial unless $G$ is of type $A_n$, $D_{2n+1}$, or $E_6$, and in all of these cases it is an involution (see \cite[Section 4]{Baraglia2010CyclicHB}). 

\begin{defn}\label{def: CAT}
    Let $G$ be a simple split real Lie group of adjoint type with Coxeter number $d$ and let $\Pi,\mathcal{Z}$, $(r_\alpha)_{\alpha\in \Pi},$ $(a_{\alpha\beta})_{\alpha,\beta\in\mathcal{Z}}$, and $\xi$ be as in Sections \ref{sec: excursion}-\ref{sec: proof of Theorem C}. Let $c_1$ and $\overline{c_2}$ be oppositely oriented complex structures on $S$ with Bers Laplacian $\Delta_h$, and $q_1\in H^0(S,\mathcal{K}_{c_1}^d), \overline{q_2}\in H^0(S,\mathcal{K}_{\overline{c_2}}^d).$
    
    We say that a $C^2$ function $U=(U_\alpha):S\to \C^l$ satisfying $U_\alpha=U_{\xi(\alpha)}$ for all $\alpha\in \Pi$ solves the \textbf{complex affine Toda equations} for $(G,c_1,\overline{c_2},q_1,\overline{q_2})$ if
\begin{equation}\label{eq: CAT}
\Delta_h \log U_\alpha=2\sum_{\beta\in \Pi}a_{\alpha\beta}r_\beta U_\beta -2a_{\alpha\delta}\frac{q_1\overline{q_2}}{h^d}U_{-\delta} - 2, \hspace{1mm} \alpha\in \Pi, 
\end{equation}
where $h=\hpair$ is the Bers metric corresponding to $\ccpair$.
\end{defn}


\subsubsection{Assumption \ref{ass: only ass}}\label{sec: on assumption}
 We impose a constraint on the class of Lie algebra isomorphisms $\mathcal{I}$ that we will look for. Note that $\g_0=\mathfrak{h}^{\C}$. We interpret $S\times \mathfrak{h}^{\C}$ as the subbundle $S\times \g_0$ of $\mathrm{ad}P_{G^{\C}}$.
 \begin{ass}\label{ass: only ass}
 We demand that the restriction of $\mathcal{I}$ to $S\times \mathfrak{h}^{\C}$ is multiplication by $-1$.
 \end{ass}
 \begin{remark}
 The assumption above might appear ad hoc. On a closed surface, for the case $c_1=c_2$ and $q_1=q_2$, the assumption holds; it is a consequence of the theory of cyclic $G$-Higgs bundles (see \cite{Baraglia2010CyclicHB} or \cite{ST}). One heuristic justification for making this assumption is that if $\mathcal{I}$ lives in a holomorphic family of solutions that includes the real solutions, the assumption should hold. 
 \end{remark} We record the following essential consequences.
\begin{prop}\label{prop: h valued}\
    Under Assumption \ref{ass: only ass}, 
\begin{enumerate}
    \item For a root $\alpha$ of height $m$, $\mathcal{I}$ restricts on $\g_\alpha\otimes \mathcal{K}_{c_1}^m$ to a map $\mathcal{I}_\alpha: \g_\alpha\otimes \mathcal{K}_{c_1}^m\to \overline{\g_{-\alpha}}\otimes \mathcal{K}_{\overline{c_2}}^{-m}.$
    \item $\mathcal{I}_\alpha$ is determined entirely by $\{\mathcal{I}_{\alpha}: \alpha\in \Pi\}$ or $\{\mathcal{I}_{-\alpha}: \alpha\in \Pi\}$.
    \item The principal connection $A^{\mathcal{I}}$ is valued in $S\times\mathfrak{h}^{\C}$.
\end{enumerate}
\end{prop}
\begin{proof}
Toward (1), let $\Delta_m$ be the set of roots of height $m$. It is easy to see using the definition (\ref{eq: the lie bracket}) that the adjoint action of $S\times \mathfrak{h}^{\C}$ decomposes $\g_m\otimes \mathcal{K}_{c_1}^m$ into weight spaces as $\g_m\otimes \mathcal{K}_{c_1}^m=\oplus_{\alpha \in \Delta_m} \g_\alpha\otimes \mathcal{K}_{c_1}^m$, with weight $\alpha$ on $\g_\alpha\otimes \mathcal{K}_{c_1}^m$. By the same token, the adjoint action of $S\times \overline{\mathfrak{h}}^{\C}$ splits $\overline{\g_m}\otimes\mathcal{K}_{\overline{c_2}}^m$ as $\overline{\g_m}\otimes \mathcal{K}_{\overline{c_2}}^m=\oplus_{\alpha \in \Delta_m} \overline{\g_\alpha}\otimes\mathcal{K}_{\overline{c_2}}^m$, with weight $\alpha$ on $\overline{\g_\alpha}\otimes \mathcal{K}_{\overline{c_2}}^m$. Thus, for $X$ a local section of $S\times \mathfrak{h}^{\C},$ and $Y$ a local section of $\g_\alpha\otimes \mathcal{K}_{c_1}^m$, $$[X,\mathcal{I}(Y)]=[-\mathcal{I}(X),\mathcal{I}(Y)]= -\mathcal{I}([X,Y]) = -\alpha(X)\mathcal{I}(Y),$$ and hence $\mathcal{I}(Y)$ is a local section of $\overline{\g_{-\alpha}}\otimes \mathcal{K}_{\overline{c_2}}^{-m}$. The result follows.

For (2), since the root spaces $\{\g_\alpha, \g_{-\alpha}:\alpha\in \Pi\}$ generate $\g$, the collection of subbundles $\{\g_{\alpha} \otimes \mathcal{K}_{c_1}, \g_{-\alpha} \otimes \mathcal{K}_{c_1}^{-1}: \alpha\in \Pi\}$, under the bracket (\ref{eq: the lie bracket}), generates $\textrm{ad}P_{G^{\C}}.$ Since $\mathcal{I}$ is a Lie algebra homomorphism, it is clear that $\mathcal{I}$ is determined by $\{\mathcal{I}_\alpha, \mathcal{I}_{-\alpha}:\alpha\in \Pi\}.$ We show that $\mathcal{I}_\alpha$ determines $\mathcal{I}_{-\alpha}$ and vice versa. Indeed, let $X_\alpha$ and $X_{-\alpha}$ be nowhere vanishing local sections of $\g_{\alpha} \otimes \mathcal{K}_{c_1}$ and $\g_{-\alpha} \otimes \mathcal{K}_{c_1}^{-1}$ respectively. Pick nowhere vanishing local sections $Y_\alpha$ and $Y_{-\alpha}$ of $\g_{\alpha} \otimes \mathcal{K}_{\overline{c_2}}$ and  $\g_{-\alpha} \otimes \mathcal{K}_{\overline{c_2}}^{-1}$ respectively such that $[X_\alpha,X_{-\alpha}]=[Y_\alpha,Y_{-\alpha}]$. Since all of the spaces in question are $1$-dimensional, there are locally defined functions $k_\alpha$ and $k_{-\alpha}$ such that $\mathcal{I}_\alpha(X_\alpha)=k_\alpha Y_{-\alpha}$ and $\mathcal{I}_{-\alpha}(X_{-\alpha})=k_{-\alpha} Y_{\alpha}$, which completely determine $\mathcal{I}_\alpha$ and $\mathcal{I}_{-\alpha}$. We have $$[k_\alpha Y_{-\alpha},k_{-\alpha}Y_{\alpha}]=[\mathcal{I}_\alpha(X_\alpha), \mathcal{I}_{-\alpha}(X_{-\alpha})]=\mathcal{I}([X_\alpha,X_{-\alpha}])=[X_{-\alpha},X_\alpha],$$ and hence $k_\alpha = k_{-\alpha}^{-1}$, and the result follows. 

For assertion (3), let $s:V\to P_{G^{\C}}$ be any local trivialization over an open subset $V\subset S$. We write, on the adjoint bundle, over $V$, $$\overline{\partial}_1+\partial_2^{\mathcal{I}} = d+\textrm{ad}(s^*A^{\mathcal{I}}).$$ Since $\mathcal{I}$ is the identity in $S\times \mathfrak{h}^{\C}$, on any local section $X$ of $S\times \mathfrak{h}^{\C}$, $\overline{\partial}_1+\partial_2^{\mathcal{I}} = dX$. Thus, $$[s^*A^{\mathcal{I}},X]=\textrm{ad}(s^*A^{\mathcal{I}})(X) = 0.$$ Since $\mathfrak{h}^{\C}$ is maximal abelian and $X$ is arbitrary, it follows that $s^*A^{\mathcal{I}}$ is valued in $ S\times \mathfrak{h}^{\C},$ and (2) is immediate.
\end{proof}

Under Assumption \ref{ass: only ass}, the compatibility condition $\mathcal{I}\circ \sigma=\overline{\sigma}\circ \mathcal{I}$ is equivalent to a relation between the $\mathcal{I}_\alpha$'s. This is where $\xi$ comes into play. 
We check that for $h\in\mathfrak{h}^{\C}$ and $X_\alpha$ in a root space $\g_\alpha,$ $$[h,\sigma_0(X_\alpha)]=\sigma_0([\sigma_0(h),X_\alpha]=(\alpha\circ \sigma_0)(h)\sigma_0(X_\alpha).$$ Hence, $\sigma_0$ defines a family of isomorphisms, for $\alpha\in \Pi,$ $\g_\alpha\to \g_{\alpha\circ \sigma_0}=\g_{\xi(\alpha)},$ and $\g_{-\alpha}\to \g_{-\alpha\circ \sigma_0}=\g_{-\xi(\alpha)}$. Returning to our adjoint bundles, the compatibility condition on $\mathcal{I}$ becomes 
\begin{equation}\label{compatibility I alpha}
    \mathcal{I}_{\xi(\alpha)}\circ \sigma = \overline{\sigma}\circ \mathcal{I}_\alpha, \hspace{1mm} \mathcal{I}_{\xi(-\alpha)}\circ \sigma = \overline{\sigma}\circ \mathcal{I}_{-\alpha}
\end{equation}
for all $\alpha\in \Pi$.

Finally, by the Chevalley theorem (see \cite[Section 23.1]{Hu}), to check that $-\mathcal{I}$ preserves $\mathrm{Ad}(G^{\C})$-invariant polynomials, one only needs to check that $-\mathcal{I}$ preserves polynomials on $S\times \mathfrak{h}^{\C}$ that are invariant under the Weyl group. Our assumption makes this trivial.
\subsubsection{The proof of Theorem C}
We will see that Assumption \ref{ass: only ass} leads to the promised system of equations of Theorem C. Here is the high-level overview (compare with \cite[Section 3]{ST}). The term $F(A^{\mathcal{I}})$ is a $2$-form valued in $S\times \mathfrak{h}^{\C}\subset \mathrm{ad}P_{G^{\C}}$. We will compute below that $[\phi_1,\overline{\phi_2}^{\mathcal{I}}]$ is valued in $S\times \mathfrak{h}^{\C}$ as well. We consider the adjoint action of $S\times \mathfrak{h}^{\C}$ on $\g_{-1}\otimes \mathcal{K}_{c_1}^{-1},$ and interpret $F(A^{\mathcal{I}})$ and $[\phi_1,\overline{\phi_2}^{\mathcal{I}}]$ as $\mathrm{End}(\g_{-1}\otimes \mathcal{K}_{c_1}^{-1})$-valued $2$-forms. The splitting $\g_{-1}\otimes \mathcal{K}_{c_1}^{-1}=\oplus_{\alpha\in \Pi}\g_{-\alpha}\otimes \mathcal{K}_{c_1}^{-1}$ gives a decomposition of $\mathrm{End}(\g_{-1}\otimes \mathcal{K}_{c_1}^{-1})$ in which we can find tractable expressions for the components of the $\mathrm{End}(\g_{-1}\otimes \mathcal{K}_{c_1}^{-1})$-valued forms. Since every $\mathrm{End}(\g_{-\alpha}\otimes \mathcal{K}_{c_1}^{-1})$ is trivial, the equation (\ref{eq: bi-Hitchin flatness}) becomes a system of equations for $|\Pi|=l$ $2$-forms. Dividing by a suitably chosen $2$-form, the equation (\ref{eq: bi-Hitchin flatness}) will be equivalent to a system of $l$ functions. 

Preparing for our computation, first observe that since $\mathcal{I}|_{\g_{-1}\otimes \mathcal{K}_{c_1}^{-1}}=\sum_{\alpha \in \Pi}\mathcal{I}_{-\alpha},$ $\partial_2^{\mathcal{I}}$ restricts on the subbundle to $\sum_{\alpha\in \Pi} \partial_2^{-\alpha}$, where $\partial_2^{-\alpha}=\mathcal{I}_{-\alpha}^{-1}\circ \partial_2\circ \mathcal{I}_{-\alpha}$ is an operator on $\g_{-\alpha}\otimes \mathcal{K}_{c_1}^{-1}$-valued forms. By its definition, $\overline{\partial}_1$ preserves the space of forms valued in $\g_{-\alpha}\otimes \mathcal{K}_{c_1}^{-1}$. Note that the Bers metric of $c_1$ and $\overline{c_2}$, seen as a section of $\mathcal{K}_{c_1}\otimes \mathcal{K}_{\overline{c_2}}$, determines an isomorphism $\kappa_0: \mathcal{K}_{c_1}^{-1}\to \mathcal{K}_{\overline{c_2}}.$ Explicitly, if the Bers metric is locally given by the symmetrization of $\lambda dz \otimes d\overline{w},$ then $\kappa_0$ takes $dz^{-1}$ to $\lambda d\overline{w}.$ Any other isomorphism of these line bundles differs by scalar multiplication by a non-vanishing function. It follows that for each $\alpha,$ since we can absorb multiplication by smooth functions from the $\mathcal{K}_{c_1}^{-1}\to \mathcal{K}_{\overline{c_2}}$ piece of $\mathcal{I}_{-\alpha}$ into the $\g_{-\alpha}\to \overline{\g_\alpha}$ piece, we can write 
\begin{equation}\label{eq: I alpha}
    \mathcal{I}_{-\alpha}=\mu_\alpha \otimes \kappa_0,
\end{equation}
where $\mu_\alpha$ is an isomorphism from $\g_{-\alpha}\to \overline{\g_{\alpha}}$. Since $x_{-\alpha}$ and $x_\alpha$ span $\g_{-\alpha}$ and $\g_{\alpha}$ respectively, there is a nowhere vanishing function $U_\alpha$ on $S$ such that $$\mu_\alpha(x_{-\alpha})=-U_\alpha x_\alpha.$$ That is, after writing (\ref{eq: I alpha}), $\mathcal{I}_{-\alpha}$ is completely determined by a function $U_\alpha$. We choose $-U_\alpha$ because, in the real case, this makes $U_\alpha$ positive. From the proof of (2) in Proposition \ref{prop: h valued}, $$\mathcal{I}_{\alpha}=\mu_\alpha^{-1}\otimes \frac{1}{\kappa_0}$$ (we point out that $\frac{1}{\kappa_0}$ is not $\kappa_0^{-1}$, it locally takes $dz\mapsto \frac{1}{\lambda}d\overline{w}^{-1})$.
Note that, by the discussion in Section \ref{sec: on assumption}, the compatibility condition (\ref{compatibility I alpha}) is equivalent to $U_\alpha=U_{\xi(\alpha)}.$ 

To warm up for the proof of Theorem C, we compute connection forms in terms of $\lambda$ and the $U_\alpha$'s.  Let $V\subset S$ be an open subset and $s:V\to P_{G^{\C}}$ a trivialization that's holomorphic for $P_{G^{\C}}\to (S,c_1)$. In the induced trivialization of $\g_{-1}\otimes K_{c_1}^{-1}=\oplus_{\alpha\in \Pi}\g_{-\alpha}\otimes \mathcal{K}_{c_1}^{-1},$ by part (3) of Proposition \ref{prop: h valued}, $A^{\mathcal{I}}$ yields a connection form $\mathrm{ad}(s^*A^{\mathcal{I}})$ that splits as $\mathrm{ad}(s^*A^{\mathcal{I}})=\sum_{\alpha \in \Pi}A^{\mathcal{I}_{-\alpha}}$.
\begin{prop}\label{prop: K^c connection form}
As $\g^{\C}$-valued $1$-forms on the open subset $V$, $A^{\mathcal{I}_{-\alpha}}=\partial_J \log (\lambda U_\alpha).$
\end{prop}
As in Section \ref{subsection: complex metrics}, recall that the logarithm is defined only locally, but the Laplacian of the logarithm is globally defined.
\begin{proof}
On  $\g_{-\alpha}\otimes \mathcal{K}_{c_1}^{-1}$, $A^{\mathcal{I}_{-\alpha}}$ is described by $$\overline{\partial}_1+\partial_2^{-\alpha}=d+ A^{\mathcal{I}_{-\alpha}}.$$ 
We just need to evaluate the connection on a local section of the form $x_{-\alpha}\otimes dz^{-1},$ where $x_{-\alpha}$ is taken from our Chevalley basis. Since both $d$ and $\overline{\partial}_1$ kill such a section, we are left to examine $\partial_2^{-\alpha}(x_{-\alpha}\otimes dz^{-1}).$ We compute
\begin{align*}
    \partial_2^{-\alpha}(x_{-\alpha}\otimes dz^{-1})&= \mathcal{I}_{-\alpha}^{-1}\circ \partial_2\circ \mathcal{I}_{-\alpha}(x_{-\alpha}\otimes dz^{-1}) \\
     &=  \mathcal{I}_{-\alpha}^{-1}\circ \partial_2(-U_\alpha x_\alpha\otimes \lambda d\overline{w})\\
     &=  \mathcal{I}_{-\alpha}^{-1}(\partial_{J}(-\lambda U_\alpha)\otimes x_\alpha)\otimes d\overline{w}\\
     &= \lambda^{-1} U_\alpha^{-1}(\partial_{J}(\lambda U_\alpha))\otimes (x_{-\alpha}\otimes dz^{-1}),
\end{align*}
which simplifies to $\partial_J \log (\lambda U_\alpha)\otimes (x_{-\alpha}\otimes dz^{-1}).$ This establishes the result.
\end{proof}
Finally, we come to the proof of Theorem C.
\begin{proof}[Proof of Theorem C]
We continue with the setup and notations exactly as above. From Theorem \ref{prop: equivalence of maps and bundles} and the constructions of the previous sections, we are required to show that, under Assumption \ref{ass: only ass}, the equation (\ref{eq: bi-Hitchin flatness}) becomes the equation (\ref{eq: CAT}). Refer to the strategy from the beginning of this subsubsection. Since the final equation (\ref{eq: CAT}) is coordinate invariant, we're welcome to work and show the equivalence of equations in a trivialization $s:V\to P_{G^{\C}}$ as above.

 Beginning our computation of the $\mathrm{End}(\g_{-\alpha}\otimes \mathcal{K}_{c_1}^{-1})$ components of a connection of the form $A^{\mathcal{I}}+\phi_1+\overline{\phi_2}^{\mathcal{I}}$, we start with the curvature term $F(A^{\mathcal{I}})$. As in the proof of Proposition \ref{prop: K^c connection form}, to compute the $\mathrm{End}(\g_{-\alpha}\otimes \mathcal{K}_{c_1}^{-1})$ component, we just need to evaluate $\overline{\partial}_1\circ \partial_2^{-\alpha}$ on a section of the form $x_{-\alpha}\otimes dz^{-1}$. Equivalently, we need to apply $\overline{\partial}_1$ to the connection form. Recalling that the Bers metric is locally $\lambda dz\otimes d\overline{w},$ by Proposition \ref{prop: Laplace formula},
     $$\overline{\partial}_1\circ \partial_2^{-\alpha}(x_{-\alpha}\otimes dz^{-1})=\overline{\partial}_1(\partial_J \log (\lambda U_\alpha)\otimes (x_{-\alpha}\otimes dz^{-1}))
     = \frac{\lambda}{4}(\Delta_h \log (\lambda U_\alpha))(dz\wedge d\overline{w})\otimes (x_{-\alpha}\otimes dz^{-1}).$$
For the commutator term, recall our choices for $e,\widetilde{e},$ and $x$, and write $$\phi_1=\sum_{\alpha\in \Pi}r_\alpha^{1/2} x_{-\alpha} + q_1 x_{\delta}, \hspace{1mm} \overline{\psi_2}=\sum_{\alpha\in \Pi}r_\alpha^{1/2} x_{-\alpha} + \overline{q_2} x_{\delta}.$$ To understand $\overline{\phi_2}^{\mathcal{I}}$, we need a good description of $\mathcal{I}^{-1}(\overline{q_2}x_\delta)$, or rather $\mathcal{I}_\delta^{-1}(\overline{q_2}x_\delta)$, where $\mathcal{I}_\delta$ is the restriction $\mathcal{I}_\delta:\g_{\delta}\otimes \mathcal{K}_{c_1}^{d-1}\to \overline{\g_{-\delta}}\otimes \mathcal{K}_{\overline{c_2}}^{-d+1}$ ($d=d_G$ is the Coxeter number). This is not hard to find: as in Section \ref{sec: excursion}, we write $\delta=\sum_{\alpha\in \Pi} n_\alpha \alpha$, with $n_\alpha>0$ and $\sum_{\alpha\in \Pi} n_\alpha = d-1.$ After factoring, similar to above, $\mathcal{I}_\delta$ can be written $$\mathcal{I}_\delta=\mu_{-\delta}\otimes \kappa_0^{-d+1},$$ for some isomorphism $\mu_{-\delta}:\g_{\delta}\otimes \mathcal{K}_{c_1}^{d-1}\to \overline{\g_{-\delta}}\otimes \mathcal{K}_{\overline{c_2}}^{-d+1}$. Using the Lie algebra homomorphism property for $\mathcal{I}$, together with the explicit expressions (\ref{eq: I alpha}) for the $\mathcal{I}_{\alpha}$'s, one can derive that $\mu_{-\delta}$ is given by $\mu_{-\delta}(x_{\delta})=U_{-\delta}x_{-\delta},$ where $$U_{-\delta}=\prod_{\alpha\in \Pi} (-U_\alpha)^{-n_\alpha}.$$ 
Now we compute $\overline{\phi_2}^{\mathcal{I}}$. Working in local coordinates, if $\overline{q_2}=\overline{q_2}(w)d\overline{w}^2$ and the Bers metric is $\lambda dz\otimes d\overline{w},$ recalling that $\kappa_0$ takes $dz^{-1}$ to $\lambda d\overline{w}$ and using our expression for $\mathcal{I}_{\alpha}$, we obtain
\begin{align*}
    \overline{\phi_2}^{\mathcal{I}} &=\sum_{\alpha\in \Pi} r_\alpha^{1/2}\lambda U_\alpha (x_\alpha\otimes dz) \otimes d\overline{w}-\lambda^{-d+1} U_{-\delta}\overline{q_2}(w)(x_{-\delta}\otimes dz^{-d+1} )\otimes d\overline{w}.
\end{align*}
Using the relations in the Chevalley base and Lemma \ref{lem: commutators}, we get
\begin{align*}
    [\phi_1,\overline{\phi_2}^{\mathcal{I}}] &=\sum_{\alpha\in \Pi} r_\alpha U_\alpha[x_{-\alpha},x_{\alpha}]\otimes\lambda dz\wedge d\overline{w}-U_{-\delta}\frac{q_1(z)\overline{q_2}(w)}{\lambda^{d}}[x_{\delta},x_{-\delta}]\otimes\lambda dz\wedge d\overline{w}\\
    &=\sum_{\alpha\in \Pi} r_\alpha  U_\alpha h_\alpha\otimes\lambda dz\wedge d\overline{w}+U_{-\delta}\frac{q_1(z)\overline{q_2}(w)}{\lambda^{d}} h_{-\delta}\otimes\lambda dz\wedge d\overline{w}.
\end{align*}
We have found an expression for $ [\phi_1,\overline{\phi_2}^{\mathcal{I}}]$ as a $(S\times \mathfrak{h}^{\C})$-valued $2$-form. For any $(S\times \mathfrak{h}^{\C})$-valued $2$-form $\Omega$, the $\mathrm{End}(\g_{-\alpha}\otimes \mathcal{K}_{c_1}^{-1})$ components of the induced $\mathrm{End}(\g_{-1}\otimes \mathcal{K}_{c_1}^{-1})$-valued form is given by $-\alpha(\Omega)$ (here, for $\eta$ a $2$-form, $\alpha(X\otimes \eta):=\alpha(X)\eta)$. This follows simply from the definition of the root space $\g_{-\alpha}$. Thus, applying $\alpha$ to $[\phi_1,\overline{\phi_2}^{\mathcal{I}}]$, which is easily done once we recall the characterization $a_{\beta\alpha}=\beta(h_\alpha) = 2\frac{\nu(\alpha,\beta)}{\nu(\alpha,\alpha)}$, and contracting with $\lambda dz\wedge d\overline{w},$ our system (\ref{eq: bi-Hitchin flatness}) becomes
$$\frac{1}{4}\Delta_h \log (\lambda U_\alpha)-\frac{1}{2}\sum_{\beta\in \Pi} a_{\alpha\beta}r_\beta U_\beta+\frac{1}{2}a_{\alpha\delta}\frac{q_1\overline{q_2}}{\lambda^d}U_{-\delta}=0, \alpha\in \Pi.$$ Rearranging and using  Proposition \ref{prop: curvature equality} ($\Delta_h \log \lambda = 2$)
we get $$\Delta_h \log U_\alpha=2\sum_{\beta\in \Pi}a_{\alpha\beta}r_\beta U_\beta -2a_{\alpha\delta}\frac{q_1\overline{q_2}}{\lambda^d}U_{-\delta} - 2, \hspace{1mm} \alpha\in \Pi,$$ as desired. 
\end{proof}

\begin{remark}\label{rem: cyclic lifts}
For a $G$-Higgs bundle, a solution to the self-duality equations reduces the underlying $K^{\C}$-bundle to a $K$-bundle. When $S$ is closed and the $G$-Higgs bundle is cyclic, the $K$-bundle reduces further to a $H$-bundle, where $H\subset K$ is a maximal torus. Equivalently, the harmonic map to $G/K$ lifts to a map to $G/H.$ This map is sometimes called the cyclic lift, and it plays an essential role in important works such as \cite{Lab3}.

For the complex harmonic maps constructed through Theorem C, the subbundle $S\times \mathfrak{h}^{\C}\subset \mathrm{ad}P_{G^{\C}}$ determines a reduction to $H^{\C}.$ Since the involution $\sigma$ restricts to $S\times \mathfrak{h}^{\C}$, the harmonic map to $G^{\C}/K^{\C}$ lifts to $G^{\C}/H^{\C}$. It would be interesting to study this lift further.
\end{remark}

\begin{remark}\label{rem: other immersions}
    As can be seen using the cyclic lift, minimal immersions for Hitchin representations into $\mathrm{PSL}(3,\R)$ are equivalent to hyperbolic affine spheres in $\R^3$, and there is a similar story for the other rank $2$ Lie groups (see also \cite{Nie} for $\mathrm{SO}(n,n+1)$ and \cite{RTpara} for $\mathrm{SL}(2n+1,\R)$). It seems probable that for certain $G$, the conformal harmonic maps from Theorem C are equivalent to some  special immersions inside complex homogeneous spaces. For $\textrm{PSL}(3,\R)$, this assertion is supported by our previous work on complex affine spheres \cite{ElSa}: by adapting Labourie's Blaschke lift \cite{Lab2}, complex affine spheres should be equivalent to the solutions to (\ref{eq: CAT}) for $G=\textrm{PSL}(3,\R)$ (we do not pursue a formal discussion and proof here).  
\end{remark}


\subsection{Existence of certain complex harmonic $G$-bundles}\label{sec: existence results}
In this subsection, we use Theorem C to construct complex harmonic maps. These are the complex harmonic maps that will come into play in the proofs of Theorems B and B'.

\begin{prop}
\label{prop: real Hitchin ex and un}
 In the setting of Theorem C, assume that $S$ is closed, and suppose $c_1=c_2$ and $q_1=q_2$. Then (\ref{eq: CAT}) admits a unique real solution.
\end{prop}
\begin{proof}
    By \cite{Baraglia2010CyclicHB} and \cite{ST}, or just following our proof, we see that Hitchin's self-duality equations for the (real) $G$-Higgs bundle $(P_{K^{\C}},\phi_1)$ is equivalent to (\ref{eq: CAT}). The result then follows from existence and uniqueness of solutions to Hitchin's self duality equations, i.e., Theorem \ref{thm: Hitchin simpsons}.
\end{proof}
The following shows that we have solutions far away from the real locus. 
\begin{prop}\label{prop: existence on marginal locus}
   In the setting of Theorem C, let $c_1$ and $c_2$ be arbitrary and suppose that at least one of $q_1$ or $\overline{q_2}$ is identically zero. Then there is a unique constant solution $U=(U_\alpha)_{\alpha\in \Pi}$ to the system (\ref{eq: CAT}). The solution is real and positive and does not depend on $c_1$, $\overline{c_2}$, $q_1$, or $\overline{q_2}$. 
\end{prop}
When $c_1=c_2,$ of course the solutions from Proposition \ref{prop: real Hitchin ex and un} and Proposition \ref{prop: existence on marginal locus} coincide.
\begin{proof}
For this data, it is helpful to transform the equation. Making the substitution $U_\alpha \mapsto U_\alpha r_\alpha$ (setting $r_{-\delta}=1)$, the equation (\ref{eq: CAT}) becomes 
\begin{equation}\label{eq: substituted Theorem C}
    \Delta_h \log U_\alpha=2\sum_{\beta\in \Pi}a_{\alpha\beta} U_\beta -2a_{\alpha\delta}\frac{q_1\overline{q_2}}{\lambda^d}U_{-\delta} - 2, \hspace{1mm} \alpha\in \Pi.
\end{equation}
Note that $U_{-\delta}$ is now $\prod_{\alpha\in \Pi} U_\alpha^{-n_\alpha} r_\alpha^{n_\alpha}$, but this is not consequential for us here. When $q_1=0$ or $\overline{q_2}=0,$ the equation (\ref{eq: substituted Theorem C})  is
    $$\Delta_h \log U_\alpha=2\sum_{\beta\in \Pi}a_{\alpha\beta} U_\beta - 2, \hspace{1mm} \alpha\in \Pi.$$ We put it in matrix form. Set $U=(U_\alpha)_{\alpha \in \Pi}$ and $\Delta_h \log U= (\Delta_h \log U_\alpha)_{\alpha\in \Pi}.$ Choosing an ordering for the simple roots allows us to build the Cartan matrix $A_{\g}$ (see \cite[Chapter 3, \S 1.7]{OV}). The equation can then be expressed as $$\Delta_h \log u = 2 A_{\g} U-2\Vec{e},$$ where $\Vec{e}=(1,\dots, 1)$.
    If we're looking for a constant solution, then the equation becomes $A_{\g}U=\Vec{e}.$ As is well known, $A_{\g}$ is positive definite, and hence there is a unique solution, which is real and positive. 
\end{proof}

\begin{prop}\label{prop: existence on bigger marginal locus}
Let $c_1$ and $c_2$ be arbitrary. Let $(q^1,\overline{q}^2)\in bH(c_1,\overline{c_2},G)$ and suppose that at least one of $q^1$ or $\overline{q}^2$ is identically zero. For this data, there exists an isomorphism solving the equation (\ref{eq: bi-Hitchin flatness}), thus producing a complex harmonic $G$-bundle with bi-Hitchin basepoint $(q^1,\overline{q}^2).$ The isomorphism does not depend on $(q^1,\overline{q}^2).$ When $S$ is closed, $c_1=c_2$, and $q^1=\overline{q}^2=0$, the complex harmonic $G$-bundle corresponds to the unique harmonic $G$-bundle coming from the self-duality equations (\ref{eq: self-duality real case}). 
\end{prop}
\begin{proof}
Given $c_1,c_2,$ let $\mathcal{I}$ be the common Lie algebra bundle isomorphism built from the solution vector $U=(U_\alpha)_{\alpha\in \Pi}$ from Proposition \ref{prop: existence on marginal locus}. We will show that this isomorphism solves (\ref{eq: bi-Hitchin flatness}) for the data in question, independent of $q^1,\overline{q}^2$. The curvature $F(A^\mathcal{I})$ of course does not depend on $(q^1,\overline{q}^2)$, and we already know by the proposition above that if we take $(q^1,\overline{q}^2)=(0,0),$ then $\mathcal{I}$ solves the equation (\ref{eq: bi-Hitchin flatness}). Hence, it suffices to show that for arbitrary $(q^1,\overline{q}^2),$ the commutator of the sections $\phi_1$ and $\overline{\phi_2}^{\mathcal{I}}$ built using $(q^1,\overline{q}^2)$ and $\mathcal{I}$ is identical to the commutator term for the case $(q^1,\overline{q}^2)=(0,0)$. We recall that for $(q^1,\overline{q}^2)=(0,0),$ this commutator is $[\widetilde{e},-\mathcal{I}^{-1}(\widetilde{e})].$

Recall the splitting $V=\oplus_{i=1}^l V_i$ from Section \ref{sec: bi-Hitchin}, and that $e_i$ is a highest weight vector for the adjoint action of $x$ on $V_i,$ with weight $m_i$. For each $i$, since $x$ preserves $V_i$, $V_i$ decomposes into a sum of root spaces, which, by basic representation theory, come in pairs $\g_\alpha$ and $\g_{-\alpha}$.

With this in mind, we calculate the commutator term. For a basepoint of the form $(0,\overline{q}^2),$ the commutator term is $$[\phi_1,-\mathcal{I}^{-1}(\overline{\psi_2})]=[\widetilde{e},-\mathcal{I}^{-1}(\widetilde{e})]+\sum_{i=1}^{l}[\widetilde{e},-\mathcal{I}^{-1}(\overline{q_i}e_i)].$$ For each $i$, by (1) in Proposition \ref{prop: h valued} and the facts mentioned above, in a trivialization that's holomorphic for $P_{G^{\C}}\to (S,c_1),$ $\mathcal{I}^{-1}(\overline{q_i}e_i)$ is a $1$-form valued in $V_i,$ and hence $[\widetilde{e},-\mathcal{I}^{-1}(q_ie_i)]$ is also valued in $V_i$ (since $\widetilde{e}$ is in the principal $\mathfrak{sl}(2,\C)$).
On the other hand, using (1) in Proposition \ref{prop: h valued} again, $[\widetilde{e},-\mathcal{I}^{-1}(q_ie_i)]$, if non-zero, has weight $-m_i-1$ for the action of $x$. Since $-m_i$ is the lowest weight, $[\widetilde{e},-\mathcal{I}^{-1}(q_ie_i)]=0$. We deduce that $[\phi_1,-\mathcal{I}^{-1}(\overline{\psi_2})]=[\widetilde{e},-\mathcal{I}^{-1}(\widetilde{e})]$, as desired.

For a basepoint $(q^1,0),$ the term in question is $$[\phi_1,-\mathcal{I}^{-1}(\overline{\psi_2})]=[\widetilde{e},-\mathcal{I}^{-1}(\widetilde{e})]+\sum_{i=1}^{l}[e_iq_i,-\mathcal{I}^{-1}(\widetilde{e})].$$ For every $i,$ writing $[e_iq_i,-\mathcal{I}^{-1}(\widetilde{e})]=\mathcal{I}^{-1}([\mathcal{I}(e_iq_i),\widetilde{e}]),$ the same reasoning as above shows that $[e_iq_i,-\mathcal{I}^{-1}(\widetilde{e})]$ is identically zero. Thus, yet again, $[\phi_1,-\mathcal{I}^{-1}(\overline{\psi_2})]=[\widetilde{e},-\mathcal{I}^{-1}(\widetilde{e})]$, and the proof is complete.
\end{proof}

\begin{remark}\label{rem: failure}
We do not attempt to solve (\ref{eq: CAT}) everywhere. As we'll see in Section 6, we are most interested in solutions to (\ref{eq: CAT}) such that the linearization is an isomorphism of the relevant Sobolev spaces. For $G=\mathrm{PSL}(2,\R)$, taking $c_1=c_2$ and $q:=q_1=-q_2$, the equation (\ref{eq: rank 1 case}) becomes
\begin{equation}\label{eq: n=2 case bad}
    \Delta_h \log U = 2U+2|q|_h^2U^{-1}-2.
\end{equation}
Given any $q$ and the path $t\mapsto tq$, $t\geq 0,$ there is no real solution to (\ref{eq: n=2 case bad}) for large $t$ and, for small $t,$ there is a unique real analytic path of solutions $U_t$ starting at $U_0=1$ such that the linearization of (\ref{eq: n=2 case bad}) eventually has a kernel (see \cite{Uh}). Morally, the issue is that, for $|q|^2$ large, the linearization resembles $\Delta - kI,$ $k<0$, which, for certain $k$, has kernel (eigenvalues of the Laplacian). For $G=\mathrm{PSL}(3,\R)$, we have the same bad behaviour (see \cite[Section 6.6]{ElSa}). 

\end{remark}

\section{Complex harmonic maps and opers}\label{sec: opers}
In this section, we show how $G^{\C}$-opers can be seen as complex harmonic maps, and we prove Theorem D. We keep all notations and conventions from the previous section. 
\subsection{$G^{\C}$-opers}\label{sec: introducing opers}
We first give a rapid introduction to the theory of $G^{\C}$-opers, including holomorphic connections. For more on this theory, we refer the reader to the paper \cite{CSopers}, which also highlights relations to higher Teichm{\"u}ller theory.

\subsubsection{Holomorphic connections}\label{sec: hol connections}
A holomorphic connection on a holomorphic principal $G^{\C}$-bundle $P_{G^{\C}}$ over a Riemann surface is just a connection in the sense of Definition \ref{def: principal connection}, i.e., a $\g^{\C}$-valued $1$-form with some properties, which is also holomorphic. Since there are no $(2,0)$-forms on a Riemann surface, a holomorphic connection on a Riemann surface is flat. 

We record for future use that the passage from holomorphic connections to flat connections is in some sense reversible. Let $P_{G^{\C}}$ be a smooth principal $G^{\C}$-bundle over a Riemann surface with a connection $A$. Via the Koszul-Malgrange theorem, equip $P_{G^{\C}}$ with the unique complex structure such that $A$ is of type $(1,0).$ Then it is easily checked that, for this complex structure, $A$ is holomorphic if and only if it is flat.
 
Since it might be more familiar to the reader, we often write out the associated data on the adjoint bundle. Let $P_{G^{\C}}$ be a holomorphic principal $G^{\C}$-bundle on a Riemann surface with canonical bundle $\mathcal{K}$. Given a representation on a complex vector space $\xi:G^{\C}\to \mathrm{GL}(V)$ with associated holomorphic vector bundle $V_\xi=P_{G^{\C}}\times_{\xi} V$, let $\mathcal{O}_{V_\xi}$ be the sheaf of holomorphic sections. Any holomorphic connection $A$ induces an affine holomorphic connection $\partial_A$ on $V_\xi$, which by definition is a $\C$-linear operator $\partial_A: \mathcal{O}_{V_\xi}\to \mathcal{O}_{V_\xi}\otimes \mathcal{K}$ such that for locally defined holomorphic sections $s$ and locally defined holomorphic functions $f$, $$\partial_A(fs) =\partial f \otimes s + f\otimes \partial_A(s).$$ In general, if $\overline{\partial}$ is the del-bar operator on $V_\xi$, then $\nabla_A:= \overline{\partial}+\partial_A$ is a flat connection. 

\subsubsection{$G^{\C}$-opers}
\label{subsubsec: opers}

Let $P_{G^{\C}}$ be a holomorphic principal $G^{\C}$-bundle, $H\subset G^{\C}$ a closed complex Lie subgroup, and $P_H\subset P_{G^{\C}}$ a reduction of structure group to $H.$ The map $A: TP_{G^{\C}}\to \g^{\C}$ fits into a composition $$TP_H\to TP_{G^{\C}}  \xrightarrow{A} S\times \g^{\C} \to S\times \g^{\C}/\mathfrak{h},$$  which yields a $(S\times \g^{\C}/\mathfrak{h})$-valued $1$-form $A_{\mathfrak{h}}$ on $P_H$ called the second fundamental form of $A$ relative to the reduction $P_H.$ The group $H$ acts on $\g^{\C}/\mathfrak{h}$ by quotienting the adjoint action, and $\C^*$ acts on $\g^{\C}/\mathfrak{h}$ by multiplication. Let $\mathcal{O}\subset \g^{\C}/\mathfrak{h}$ be a $H\times \C^*$ invariant subset.
\begin{defn}
    We say that $A$ has \textbf{relative position} $\mathcal{O}$, written $\mathrm{pos}_{P_H}(A)=\mathcal{O}$, if for all non-zero tangent vectors $v$ of $P_H,$ $A_{\mathfrak{h}}(v)\in \mathcal{O}$.
\end{defn}

Now choose a Borel subgroup $B\subset G^{\C}$ with Lie algebra $\mathfrak{b}\subset \g^{\C}$ and a Cartan subgroup $H^{\C}\subset G^{\C}$ contained inside $B$. Since all pairs $H^{\C}$ and $B$ are equivalent under the adjoint action of $G^{\C},$ we can specify that $H^{\C}$ is the same Cartan from Section \ref{sec: ch G-bundles}, with height grading $\g=\oplus_{m=-M}^M\g_m$ from Section \ref{sec: bi-Hitchin} (constructed as the weight decomposition for the element $x$), and that $\mathfrak{b}=\oplus_{m=0}^M \g_m$. This grading defines a $B$-invariant filtration 
\begin{equation}\label{eq: first filtration}
    \g^M\subset \g^{M-1}\subset \dots \subset \g^0\subset \g^{-1}\subset \dots \subset \g^{-M}=\g^{\C},
\end{equation}
 with $\g^j = \oplus_{i=j}^M \g_i$ (note $\g^0=\mathfrak{b}$). Quotienting by $\mathfrak{b}$ yields a $B$-invariant filtration: $$\g^{-1}/\mathfrak{b}\subset \g^{-2}/\mathfrak{b}\subset \dots \subset \g^{\C}/\mathfrak{b}.$$ Note that any choice of $H^{\C}\subset G^{\C}$ would have in fact given the same filtration \cite[Section 2.1 and Example 2.2]{CSopers}. The result below follows from a theorem of Vinberg (see \cite[Theorem 10.19]{Knapp}) and leads into the definition of a $G^{\C}$-oper.

\begin{thm}
    There is a unique open and dense $B$-orbit $\mathcal{O}\subset \g^{-1}/\mathfrak{b}$ with respect to the adjoint $B$-action.
\end{thm}
The most hands-on way to check the relative position of a holomorphic connection is to choose a trivialization and examine the connection form. By the definitions, $\g^{-1}/\mathfrak{b}$ identifies with the sum of the negative simple root spaces. An element in $\g^{-1}$ determines an element in $\mathcal{O}$ if and only if its projection onto every negative simple root space is non-zero (see \cite[Example 2.2]{CSopers}).

\begin{defn}
    A \textbf{$G^{\C}$-oper} $(P_{G^{\C}},P_B,A)$ over the Riemann surface $(S,c)$ is the data of a principal $G^{\C}$-bundle $P_{G^{\C}}$ over $(S,c)$ carrying a holomorphic connection $A$ and equipped with a holomorphic reduction $P_B\subset P_{G^{\C}}$ to a Borel subalgebra $B\subset G^{\C}$ such that $\mathrm{pos}_{P_B}(A)=\mathcal{O}$.
\end{defn}
Geometrically, the reduction $P_B$ gives a holomorphic map to $G^{\C}/B,$ which is equivariant with respect to the holonomy of $A$. The holonomy of $A$ is referred to as the \textbf{holonomy of the} $G^{\C}$\textbf{-oper}.
The following result from \cite[Section 3]{BD} (see also \cite[\S 1]{BD2}) will be used to show that our maps $\mathcal{B}_G$ and $\mathcal{L}_G^{\C}$ of Theorems B and B' map into the smooth locus of the character variety.
\begin{prop}\label{prop: smooth locus}
    Holonomy representations of $G^{\C}$-opers are irreducible and simple.
\end{prop}

As the main example, we construct Beilinson-Drinfeld's $G^{\C}$-opers (see \cite{BD} and \cite[Section 5.1]{CSopers}), where $G^{\C}$ is a complex simple Lie group of adjoint type. These $G^{\C}$-opers are parametrized by the Hitchin base for the split real form, which we denote by $G$. We return to the notations of Section \ref{sec: ch G-bundles}. Fixing a Riemann surface $(S,c)$ with canonical bundle $\mathcal{K}$, consider the bundle $P_{G^{\C}}$ from Section \ref{sec: ch G-bundles}. The subbundle of the adjoint bundle $$\oplus_{i=0}^{M}\g_i\otimes \mathcal{K}^i\subset \mathrm{ad}P_{G^{\C}}=\oplus_{m=-M}^M\g_m\otimes \mathcal{K}^m$$ is the adjoint bundle of a reduction to a Borel $P_B\subset P_{G^{\C}}$. Let $(P_{K^{\C}},\phi_0)$ be the unique nilpotent Higgs bundle in the Hitchin section (that is, $\phi_0=\Tilde{e}$). Recall from Section \ref{sec: ch G-bundles} that $P_{K^{\C}}$ is a reduction of $P_{G^{\C}}$. For the stable $G$-Higgs bundle $(P_{K^{\C}},\phi_0)$, let $P_K$ be the reduction from Theorem \ref{thm: Hitchin simpsons}, with Chern connection $A_{P_K}$ and involution $\widehat{\sigma}$ of $\mathrm{ad}P_{G^{\C}}$. For every $q=(q_1,\dots, q_{l})\in H(c,G)$, we set $A_q$ to be the flat connection $$A_q=A_{P_K}-\widehat{\sigma}(\phi_0)+\phi_0+q_1e_1+\dots + q_l e_l.$$ We then equip $P_{G^{\C}}$ with the unique holomorphic structure such that $A_q$ of type $(1,0)$ and thus holomorphic. We point out that in this holomorphic structure, $P_B$ is a holomorphic reduction. 
\begin{defn}\label{def: BD oper}
    \textbf{Beilinson-Drinfeld's $G^{\C}$-oper} associated with $q=(q_1,\dots, q_{m_l})\in H(c,G)$ is the above data $(P_{G^{\C}},P_B,A_q)$.
\end{defn}
On the adjoint bundle, the del bar operator is $\overline{\partial}_P = \overline{\partial}_1-\widehat{\sigma}(\phi_0)$, and the affine holomorphic connection is $\widehat{\sigma}(\overline{\partial}_1)+\phi_0+q_1e_1+\dots + q_l e_l$. From our description of elements of $\g^{-1}$ projecting to $\mathcal{O}$, and the explicit expression for $\widetilde{e}$, it is easily seen that $\mathrm{pos}_{P_B}(A_q)=\mathcal{O}$. 

In Definition \ref{def: BD oper}, the $q_1$ is always a quadratic differential, and when $G^{\C}=\mathrm{PSL}(2,\C),$ it is a positive scalar multiple of the Schwarzian derivative of the associated complex projective structure. Beilinson-Drinfeld in fact prove that the map $q\mapsto (P_{G^{\C}},P_B,A_q)$ defines a parametrization of the moduli space of $G^{\C}$-opers by the Hitchin base, although the identification depends on choices (see \cite{BD} or \cite[\S 3]{BD2} for the precise statement). Recall that $G$ is assumed to be of adjoint type; in general, the moduli space of $G^{\C}$-opers might be disconnected, but every component can be identified with the Hitchin base in a similar fashion. 

From the expressions for the flat connections, the following special case (the case $c_1=c_2$) of Theorem D is immediate.
\begin{prop}\label{prop: BD oper}
    Let $c:=c_1=c_2$, let $(q,0)\in bH(c,\overline{c},G).$ The holonomy of the equivariant complex harmonic map to $G^{\C}/K^{\C}$ obtained through Proposition \ref{prop: existence on bigger marginal locus} agrees with that of Beilinson-Drinfeld's $G^{\C}$-oper on $(S,c)$ associated with $q\in H(c,G)$.
\end{prop}
\begin{proof}
Let $\mathcal{I}$ be the isomorphism solving (\ref{eq: bi-Hitchin flatness}) from Proposition \ref{prop: existence on bigger marginal locus}, giving complex harmonic $G$-bundle $(P_{K^{\C}},A^{\mathcal{I}},\phi_1,\overline{\phi_2}^{\mathcal{I}})$. From Proposition \ref{prop: existence on bigger marginal locus}, $\mathcal{I}=-\widehat{\sigma}$, and we see that $A^{\mathcal{I}}=A_{P_K}$. Also, note that $\phi_1=\phi_0$ (both are $\widetilde{e}$). Therefore, the flat connection $A^{\mathcal{I}}+\phi_1+\overline{\phi_2}^{\mathcal{I}}$ is
\begin{equation}\label{eq: surprise oper}
A_{P_K} + \phi_0+q_1e_1+\dots+q_le_l-\widehat{\sigma}(\phi_0)=A_q
\end{equation}
\end{proof}
The geometric explanation behind the equality (\ref{eq: surprise oper}) is that, via the reduction $P_{H^{\C}}$ of the usual bundle $P,$ the complex harmonic map lifts to $G^{\C}/H^{\C},$ which then projects to $G^{\C}/B,$ giving the holomorphic map associated with the oper.
\begin{remark}
    In the language of \cite{CSopers}, our $G^{\C}$-opers are $(G^{\C},B)$-opers. We expect that all $(G^{\C},P)$-opers in the sense of \cite{CSopers} can be realized as complex harmonic maps.
\end{remark}

\subsection{Proof of Theorem D}\label{sec: proof of Theorem D}
Having introduced the necessary background, and informed by Proposition \ref{prop: BD oper}, we now prove Theorem D. 
\begin{proof}[Proof of Theorem D]
        Part (1) is Proposition \ref{prop: existence on bigger marginal locus}. For part (2), we begin with the $(q^1,0)$ case. Let $c_1,\overline{c_2}$ be oppositely oriented complex structures on $S$ and let $(P_{K^{\C}},A^{\mathcal{I}},\phi_1,\overline{\phi_2}^{\mathcal{I}})$ be any complex harmonic $G$-bundle over $(S,c_1,\overline{c_2})$ arising from a solution $\mathcal{I}$ to (\ref{eq: bi-Hitchin flatness}). Since the connection $A=A^{\mathcal{I}}+\phi_1+\overline{\phi_2}^{\mathcal{I}}$ on $P_{G^{\C}}=P_{K^{\C}}\times_{K^{\C}}G^{\C}$ is flat, as discussed at the beginning of Section \ref{sec: introducing opers}, it determines a holomorphic structure on $P_{G^{\C}}\to (S,c_1)$ in which $A$ is a holomorphic connection.

Let $P_B\subset P_{G^{\C}}$ be the reduction to a Borel from the subsection above.
We now assume that our Hitchin basepoint is $(q^1,0)\in bH(c_1,\overline{c_2},G)$ and that our complex harmonic $G$-bundle is the one obtained through Proposition \ref{prop: existence on bigger marginal locus}. We claim that $(P_{G^{\C}},P_B,A)$ is a $G^{\C}$-oper, and to do so, we only need to show that $\mathrm{pos}_{P_B}(A)=\mathcal{O}$. We check, in a trivialization that is holomorphic for $P_{G^{\C}}\to (S,c_1),$ that the connection form is in $\mathcal{O}$. In such a trivialization, the connection form is the sum of the connection form of $A^{\mathcal{I}}$ with the Higgs fields. We know by (3) in Proposition \ref{prop: h valued} that the connection form of $A^{\mathcal{I}}$ is valued in $S\times \mathfrak{h}^{\C}$. We also have from the explicit expression 
$$
    \overline{\phi_2}^{\mathcal{I}}=\sum_{\alpha\in \Pi} r_\alpha^{1/2}\lambda U_\alpha (x_\alpha \otimes dz)\otimes d\overline{w}$$
(here, $U_\alpha$ is a constant that does not depend on $c_1,\overline{c_2},$ or $q^1$) that $\overline{\phi_2}^{\mathcal{I}}\in \Omega^1(S,\g_1\otimes \mathcal{K}_{c_1})$. Recalling the filtration (\ref{eq: first filtration}), if $\mathfrak{b}$ is the Lie algebra of $B$, then both $\mathfrak{h}^{\C}$ and $\g_1$ lie in $\mathfrak{b}$. It follows that neither $A^{\mathcal{I}}$ nor $\overline{\phi_2}^{\mathcal{I}}$ contribute to the computation of $\mathrm{pos}_{P_B}(A).$ Writing $\phi_1=\widetilde{e}+q_1e_1+\dots + q_le_l,$ as with Beilinson-Drinfeld's opers, the explicit expression for $\widetilde{e}$ shows that $\mathrm{pos}_{P_B}(A)=\mathcal{O}$, and hence we have an oper. 
        
For the $(0,\overline{q}^2)$ case, we can argue similarly to above, except choosing a holomorphic structure on $P_{G^{\C}}$ that makes $P_{G^{\C}}\to (S,\overline{c_2})$ holomorphic. As at the end of Section \ref{sec: bi-Hitchin}, our complex harmonic $G$-bundle data is now written $(P_{K^{\C}},(A')^{\mathcal{I}},\psi_1^{\mathcal{I}},\overline{\psi_2})$. To see that we have an oper, similar to above, $(A')^{\mathcal{I}}$ and $\psi_1^{\mathcal{I}}$ do not factor into computing the relative position, and the explicit expression $\overline{\psi}_2=\widetilde{e}+\overline{q_1}e_1+\dots + \overline{q_l}e_l$ shows that the relative position is indeed $\mathcal{O}$.
\end{proof}

With the proof of Theorem D complete, as promised in the introduction, we discuss some related objects. There is an $S^1$-action on the set of harmonic $G$-bundles: $\theta\cdot (P_{K^{\C}},\phi,P_K)=(P_{K^{\C}},e^{i\theta}\phi,P_K)$. The $S^1$-action extends to a $\C^*$-action on the set of complex harmonic $G$-bundles: $\xi\cdot (P_{K^{\C}},A_{K^{\C}},\phi_1,\overline{\phi_2})=(P_{K^{\C}},A_{K^{\C}},\xi\phi_1,\xi^{-1}\overline{\phi_2})$. This action descends to the bi-Hitchin base: $\xi\cdot (q^1,\overline{q}^2)=(\xi\cdot q^1,\xi\cdot\overline{q}^2),$ where $$(\xi\cdot (q_1^1,\dots, q_l^1),\xi \cdot (\overline{q}_1^2,\dots, \overline{q}_l^2))=((\xi^{m_1} q_1^1,\dots, \xi^{m_l}q_l^1),(\xi^{-m_1}\overline{q}_1^2,\dots, \xi^{-m_l}\overline{q}_l^2)).$$ Note that if we can solve (\ref{eq: bi-Hitchin flatness}) for a point $(q^1,\overline{q}^2)\in bH(c_1,\overline{c_2},G)$, then, for all $\xi\in \C^*,$ the same isomorphism solves (\ref{eq: bi-Hitchin flatness}) for  $\xi\cdot (q^1,\overline{q}^2)$.

Intriguingly, when $\phi_1$ and $\overline{\phi_2}$ come from an ordinary harmonic $G$-bundle and $S$ is closed, the flat connections associated with a family $\{(P_{K^{\C}},A_{K^{\C}},\xi\phi_1,\xi^{-1}\overline{\phi_2}): \xi\in \C^*\}$ form (the restriction to $\C^*$ of) a real twistor line in the space of $\lambda$-connections (a model for the twistor space of the hyperK{\"a}hler moduli space of Higgs bundles). For definitions and exposition on the $\lambda$-connections and such, we suggest the paper \cite{CW}. Although we don't pursue the direction in this paper, the appearance of the twistor lines seems worthy of investigation.

Let $S$ be closed. In \cite{Gai}, starting from a stable $G$-Higgs bundle $(P_{K^{\C}},\phi)$, Gaiotto proposed (for $G=\textrm{SL}(n,\C)$) to keep $\hbar=R^{-1}\xi$ fixed and to consider the family of flat connections $$A_{R,\hbar}= A_R+\hbar^{-1}\phi - \hbar R^2\hat{\sigma}_R(\phi),$$ where $A_R$ and $\hat{\sigma}_R$ are the Chern connection and involution respectively corresponding to solving Hitchin's self-duality equations for $(P_{K^{\C}},R\phi)$. In fact, Gaiotto was mainly interested in punctured surfaces, and in that case the starting assumptions are different, but here we keep to closed surfaces. Gaiotto conjectured that if we start with a Higgs bundle in the Hitchin section and take $R\to 0$, then, up to gauge, the flat connections converge to one of Beilinson-Drinfeld's opers, with the Hitchin base parameter equal to the Hitchin basepoint of $\hbar^{-1} \phi$. This conjecture was proved in \cite{DFKMMN} (closed case, for any $G$) and \cite{collier2024conformallimitsparabolicslnchiggs} (punctures). Our work gives a new perspective on Gaiotto's conjecture: if $(P_{K^{\C}},\phi)$ is in the Hitchin section and $A_R$ and $\hat{\sigma}_R$ are as above, then each $(P_{K^{\C}}, A_R, \hbar^{-1}\phi, \hbar R^2\hat{\sigma}_R(\phi))$ is a complex harmonic $G$-bundle and the bi-Hitchin basepoints limit to one of the form $(q^1,0)$. One can see in the proof of the conjecture in Theorem 4.11 in \cite{DFKMMN} that as $R\to 0,$ up to gauge, the data $A_R$, $\hbar^{-1}\phi,$ and $\hbar R^2\hat{\sigma}_R(\phi)$ converge to the corresponding data for the complex harmonic $G$-bundle with bi-Hitchin basepoint $(q^1,0)$ obtained through Proposition \ref{prop: existence on bigger marginal locus}. In other words, the complex harmonic $G$-bundles converge to one of ours. Our work on the general points $(c_1,q^1,\overline{c_2},0)$ and $(c_1,0,\overline{c_2},\overline{q}^2)$ suggest a complex generalization of Gaiotto's Conjecture.

\section{Construction of the maps $\mathcal{B}_G$ and $\mathcal{L}_{G}^{\C}$}\label{sec: thm B}

In this section, we use complex harmonic $G$-bundles to construct the maps $\mathcal{B}_G$ and $\mathcal L_G^{\C}$ of Theorems B and B' respectively. The main focus is to construct $\mathcal{L}_G^{\C}$ for a general Lie group. In rank $2$, we get $\mathcal{B}_G$ from $\mathcal{L}_G^{\C}$ by using Labourie's theorem to identify $\mathrm{Hit}(S,G)$ with $\mathcal{M}(S,G).$ For the general strategy, we refer back to Section \ref{sec: outline}. 
Most of the constructions are quite general, so at certain points we work in more generality than is needed for the specific problem at hand.

Throughout, let $S$ be a closed oriented surface of genus $\mathrm g\geq 2$. We fix a Riemannian metric and a connection on $S$ in order to define $L^p$-spaces $L^p$ and Sobolev spaces $W^{k,p}$. We interpret spaces of $C^\infty$ functions as complex Fr{\'e}chet spaces, built using families of Sobolev norms.

\subsection{Spaces of solutions to complex equations}
Fix a basepoint complex structure $c_0$ on $S$ and identify $\mathcal{C}(S)$ with the complex Fréchet space $\mathcal{C}(S,c_0)$ of smooth Beltrami forms on $(S,c_0).$ Likewise, identify $\mathcal{C}(\overline{S})$ with $\mathcal{C}(\overline{S},\overline{c_0}).$ For each $s\in\mathbb{Z}_+,$ we further consider the complex Banach spaces $\mathcal{C}^s(S,c_0)$ and $\mathcal{C}^s(\overline{S},\overline{c_0})$ of $W^{s,\infty}$ Beltrami forms. The specific choice $W^{s,\infty}$ is informed by holomorphic dependence results for the Beltrami equation (see \cite{ESholodependence}).

Let $(d_i)$ be an ordered set of integers and recall the bundle $\mathrm{M}(S,(d_i))$ from Section \ref{sec: Labourie paper}. For each $s\in \mathbb{Z}_+,$ we also consider the analogous bundle $\mathrm{M}^s(S,(d_i))$ over $\mathcal{C}^s(S,c_0).$
Let $s,k,l\in \mathbb{Z}_+$, with $k>2$, and consider a holomorphic map between Banach manifolds 
\[
\mathscr T: \mathrm M^s(S,(d_i))\times \mathrm M^s(\overline S,(d_i)) \times W^{k,2}(S,\C^l)\to W^{k-2,2}(S,\C^l) .\ 
\]
We denote 
\begin{align*}
    \mathrm{SOL}_{ \mathscr T,s, k}(S)&=\{\sigma=(c_1, q_1, \overline {c_2}, \overline {q_2},  u)\in \mathrm M^s(S,(d_i))\times \mathrm M^s(\overline S,(d_i)) \times W^{k,2}(S,\C^l)\ |\ \mathscr T(\sigma)=0\}\\
    \mathrm{SOL}^*_{ \mathscr T,s, k}(S)&=\{\sigma\in \mathrm{SOL}_{ \mathscr T,s, k}(S)\ |\ \partial_{u} \mathscr T: W^{k,2}(S,\C^l)\to W^{k-2,2}(S,\C^l) \text{ is an isomorphism} \},
\end{align*}
where $\partial_{u} \mathscr T$ is the linearization in the direction of $u,$ $$\partial_{u} \mathscr T(\dot{u}) = \frac{d}{dt}\Big|_{t=0}\mathscr T(c_1,q_1, \overline {c_2},\overline{q_2},u+t\dot{u}).$$
We assume that there exists a sequence $(s_n, k_n)$ with $s_n, k_n \nnearrow  \infty$ such that
\begin{equation}
\label{eq: assumption nesting}
 \CAS^*_{ \mathscr T,s_n, k_n}(S)\supseteq \CAS^*_{ \mathscr T,s_n, k'}(S) \text{ for all $k'$}.
\end{equation}
Observe that, under the assumption \eqref{eq: assumption nesting}, we also have 
$$\mathrm{SOL}^*_{ \mathscr T, s_n, K_n}(S)\supseteq \mathrm{SOL}^*_{ \mathscr T, s_{n+1}, K_n}(S)\supseteq \mathrm{SOL}^*_{ \mathscr T, s_{n+1}, K_{n+1}}(S).$$
Finally, define 
\begin{align*}
    \mathrm{SOL}_{\mathscr T}(S)&=\{\sigma=(c_1, q_1, \overline {c_2}, \overline {q_2},  u)\in \mathrm M(S,(d_i))\times \mathrm M(\overline S,(d_i)) \times C^{\infty}(S,\C^l)\ |\ \mathscr T(\sigma)=0\} =\bigcap_{n}  \mathrm{SOL}_{ \mathscr T, s_n, K_n}(S) \\
    \mathrm{SOL}_{\mathscr T}^*(S)&=\{\sigma\in \mathrm{SOL}_{\Tau} |\ \partial_{u} \mathscr T: C^{\infty}(S,\C^l)\to C^{\infty}(S,\C^l) \text{ is an isomorphism} \}=\bigcap_{n }  \mathrm{SOL}^*_{ \mathscr T, s_n, K_n}(S).
\end{align*}

\begin{prop}
    In the above set-up, $\mathrm{SOL}_{ \mathscr T,s,K}^*(S)$ is a complex Banach manifold, $\mathrm{SOL}_{\mathscr T}^*(S)$ is a complex Fréchet manifold, and the projections \begin{align*}
\pi_0&:\mathrm{SOL}_{\mathscr T, s,K}^*(S)\to \mathrm M^s(S, (d_i))\times \mathrm M^s(\overline S, (d_i)) \\
\pi_0&:\mathrm{SOL}_{\mathscr T}^*(S)\to \mathrm M(S, (d_i))\times \mathrm M(\overline S, (d_i))
    \end{align*}
    are local biholomorphisms.
\end{prop}
\begin{proof}
The map \begin{align*}
\Phi: \mathrm M^{s}(S, (d_i))\times \mathrm M^{s}(\overline S, (d_i))\times W^{k,2}(S,\C^l)&\to \mathrm M^{s}(S,(d_i))\times \mathrm M^{s}(\overline S,(d_i))\times W^{k-2,2}(S,\C^l)\\
(c_1, q_1, \overline {c_2}, \overline{q_2},  u)&\mapsto (c_1, q_1, \overline {c_2},  \overline{q_2}, \mathscr T(c_1,  q_1, \overline {c_2}, \overline{q_2},  u) )
\end{align*} is holomorphic. By assumption, for all $\sigma\in \mathrm{SOL}_{\mathscr{T},s,K}^*(S)$, $\partial_u \mathscr T$ is an isomorphism, and hence so is $d_\sigma \Phi$. By the Transversality Theorem for Complex Banach manifolds, $\mathrm{SOL}_{\mathscr T, s,K}^*(S)$ is a complex Banach submanifold of $\mathrm M^{s}(S, (d_i))\times \mathrm M^{s}(\overline S, (d_i))\times W^{k,2}(S,\C^l)$.

Since $\mathscr T|_{\mathrm{SOL}_{\mathscr T,s_n, K_n}^*(S)}\equiv 0$, $d_\sigma\Phi= (d_\sigma \pi_0, 0)$, and hence $d_\sigma \pi_0: T\mathrm{SOL}_{\mathscr T, s_n,K_n}^*(S)\to  T(\mathrm M^{s_n}(S,(d_i))\times \mathrm M^{s_n}(\overline S, (d_i)))$ is an isomorphism, which implies that $\pi_0$ is a local biholomorphism.

By taking nested intersections of well-chosen charts for each $\mathrm{SOL}_{\mathscr T, s_n,K_n}^*(S)$ in order to define a Fréchet manifold structure, or going through a formal projective limit construction, one obtains the result for the Fréchet $C^{\infty}$ spaces (see \cite[Section 6.1]{ElSa} for full explanation on these constructions).
\end{proof}
We now pass to finite-dimensional quotients. As in Section \ref{sec: Labourie paper}, let $\mathcal M(S,(d_i))$ and $\mathcal M(\overline S,(d_i))$ be the quotients of $\mathrm M(S, (d_i))$ and $\mathrm M(\overline S, (d_i))$ by the actions of $\mathrm{Diff}_0(S)$, which are holomorphic bundles on $\mathcal T(S)$ and $\mathcal T(\overline S)$ respectively.
\begin{defn}
    Let $\mathrm{V}\subset \mathrm{SOL}^*_{\Tau}(S)$ be an open subset. We define $\mathcal{V}$ as the quotient of $\mathrm{V}$ that contracts to points the connected components of the fibers of the projection $[\pi_0]: \mathrm{V}\to \mathcal M(S,(d_i))\times \mathcal M(\overline S,(d_i))$. We denote the induced projection by $\pi: \mathcal{V}\to \mathcal M(S,(d_i))\times \mathcal M(\overline S,(d_i))$.
\end{defn}
\begin{prop}
    $\mathcal{V}$ has a complex manifold structure for which the projection $\pi: \mathcal{V}\to \mathcal M(S,(d_i))\times \mathcal M(\overline S,(d_i))$ is a local biholomorphism.
\end{prop}
\begin{proof}
We can construct a covering of $\mathrm M(S, (d_i))$ by open subsets that split as $A\times B\subset \mathcal M(S,(d_i))\times \mathrm{Diff}_0(S)$ and such that the projections $A\times B\to A$ are holomorphic (see \cite[Section 6.2]{ElSa}). Around each $\sigma\in \mathrm{SOL}_{\Tau}^*(S)$, we can choose a connected neighbourhood $\widehat W_{\sigma}$ so that $\pi$ restricts to a biholomorphism onto a product of open sets of the form above:  \[\pi: {\widehat W_\sigma}\xrightarrow{\sim}  A_1\times B_1\times A_2\times B_2\subset \mathrm M(S, (d_i))\times\mathrm M(\overline S,(d_i)), \] where $A_1\subset \mathcal M(S,(d_i))$, $A_2\subset \mathcal M(\overline S,(d_i))$, and $B_1, B_2\subset \mathrm{Diff}_0(S)$ are open subsets. 
   Denoting the quotient map $\mathrm{V}\to \mathcal{V}$ by $\mathrm p_\sim$, we have the following commutative diagram:
    \begin{equation}
    \label{eq: diagram quotient CAS}
    \begin{tikzcd}
    \widehat W_{\sigma} \arrow[r, " \pi \quad \sim"] \arrow[d, "\mathrm p_{\sim}"] \arrow[rd, , "{[ \pi]}" ] &A_1\times B_1\times A_2\times B_2 \arrow[d, "\text{projection}"]\\
        W_{\sigma}:=\mathrm p_\sim(\widehat W_{\sigma}) \arrow[r, " \pi'\quad \sim "] &A_1 \times A_2
    \end{tikzcd}
    \end{equation}
It is easy to check that the collection of subsets $\{W_\sigma\}_{\sigma\in \mathrm{V}}$ of $\mathcal{V}$ defines a base for a topology on $\mathcal{V}$. Moreover, using the one-to-one restrictions of $ \pi$ to each $W_\sigma$ and the fact that the projection is holomorphic, we get a holomorphic atlas on $\mathcal{V}$ by pulling back the complex structure on $\mathcal M (S,(d_i))\times \mathcal M(\overline S, (d_i))$: by the diagram \eqref{eq: diagram quotient CAS}, $ \pi$ is a local biholomorphism and $\mathrm p_{\sim}$ is holomorphic.
\end{proof}

\subsection{Complex Lie derivatives} \label{sec: complex lie derivatvies}
Proving Theorem B will require us to study the action of $\mathrm{Diff}_0(S)\times \mathrm{Diff}_0(S)$ on products such as $\mathrm{M}(S,(d_i))\times \mathrm{M}(\overline{S},(d_i))$. In our previous paper \cite{ElSa}, we developed tools to study this type action in detail. In particular, considerations on the infinitesimal level led us to complex Lie derivatives, which we recall here. 

For a complex structure $c\in \mathcal C(S)$ with corresponding almost complex structure $J$, recall that we denote the projections of $\C TS$ onto $T_c^{1,0}(S)$ and $T_c^{0,1}S$ as $\Pi_+^J$ and $\Pi_-^J$ respectively. For all tensors $\upalpha$ on $S$ and $X\in \Gamma(\C TS)$ we define the \textbf{complex Lie derivative} of $\upalpha$ in the direction of $X$ as the complex tensor $\mathscr L_X \upalpha= \mathscr L_{Re(X)}\upalpha+ i\mathscr L_{Im(X)}\upalpha$. The complex Lie derivative of a complex structure, as in \cite{ElSa}, is defined as follows. First, for all $X\in \Gamma(TS)$, define $\mathscr L_{\mathrm X} c\in T_c \mathcal C(S)$ as the element obtained by differentiating the orbit map at c from $\Diff_0(S) \to \mathcal C(S)$ and evaluating the derivative at $\mathrm X\in  \Gamma(TS)\cong T_{id}\Diff_0(S)$. For $X\in \Gamma(\C TS)$,
\[
\mathscr L_X c:= \mathscr L_{\Pi_+^J(X)+ \overline{\Pi_+^J(X)}}c.
\]
 One motivation for this definition is that, for all
$\upalpha\in H^0(S,\mathcal K^d_c)$, we have $\mathscr L_X \upalpha= \mathscr L_{\Pi_+^J(X)+ \overline{\Pi_+^J(X)}}\ \upalpha$.  A further motivation is that $\mathscr L_X c =0$ if and only if $X\in \Gamma(T^{0,1}_c(S))$, that is, if and only if $\Pi_+^J(X)=0$. In \cite[Proposition 4.3]{ElSa}, we proved the following.
\begin{prop}
\label{prop: tangent model CSCS}
For all $\ccpair\in\CSCS$, the map
    \begin{equation*}
    \begin{split}
        \Gamma(\C TS)&\to T_{c_1}\mathcal C(S)\times T_{\overline {c_2}}\mathcal C(\overline S)\\
        X &\mapsto (\mathscr L_X c_1, \mathscr L_X \overline{c_2})
        \end{split}
    \end{equation*}
    defines a $\C$-linear isomorphism onto the tangent space of the $\Diff_0(S)\times\Diff_0(S)$-orbit of $(c_1, \overline {c_2})$. As a result, we have a complex splitting $T_{\ccpair} \mathcal C(S)\times \mathcal C(\overline S)\cong T_{[c_1]}\mathcal T(S)\times T_{[\overline{c_2}]} \mathcal T(\overline S)\times \Gamma(\C TS)$. 
    \end{prop}
    \begin{cor}\label{cor: tangent to orbit in bundles}
    The map
     \begin{equation*}
    \begin{split}
        \Gamma(\C TS)&\to T_{(c_1, q_1)}\mathrm M(S,(d_i))\times T_{(\overline {c_2}, \overline{q_2})}\mathrm M(\overline S,(d_i))\\
        Z &\mapsto (\mathscr L_Z c_1, (\mathscr L_Z q_{d_i}^1), \mathscr L_Z \overline{c_2}, (\mathscr L_Z \overline{q}_{d_i}^2)) 
        \end{split}
    \end{equation*}
is a linear isomorphism onto the tangent space of the $\Diff_0(S)\times\Diff_0(S)$-orbit of $\left(c_1, (q_{d_i}^1), \overline {c_2}, (\overline{q}_{d_i}^2)\right)$. 
      
    \end{cor}


Now, return to the setting and notations from the previous subsection.
Along with the action on bundles such as $\mathrm{M}(S,(d_i)),$ the group $\mathrm{Diff}_0(S)$ acts on spaces of functions $u:S\to \C$ by pre-composition. Thus, we can talk about $\Tau$ being $\mathrm{Diff}_0(S)$-equivariant, i.e., satisfying $\Tau(\upphi^*c_1, \upphi^*q_1,\upphi^*\overline {c_2}, \upphi^* \overline {q_2},  \upphi^* u)= \upphi^* \mathscr T(c_1, q_1, \overline {c_2}, \overline {q_2},  u)$ for all $\upphi\in \mathrm{Diff}_0(S)$. This property implies that $\mathrm{Diff}_0(S)$ acts on $\mathrm{SOL}_{\mathscr{T}}^*$. In this situation, Proposition \ref{prop: mathscr L and SOL} below provides a useful infinitesimal description of the fibers of $\mathrm{SOL}_{\Tau}^*(S)\to \mathcal{SOL}_{\Tau}(S)$. 
\begin{prop}
\label{prop: mathscr L and SOL}
Assume $\Tau$ is $\mathrm{Diff}_0(S)$-equivariant and let $V\subset \mathrm{SOL}^*_{\Tau}(S)$ be a $\mathrm{Diff}_0(S)$-invariant open subset. Let $\sigma= (c_1, q_1, \overline {c_2}, \overline{{q_2}}, u)\in V$ and let $\dot \sigma$ be a tangent vector in $T_{\sigma}V$. Then $\dot \sigma$ lies in the kernel of the differential of the projection $V\to \mathcal V$ if and only if there exists $X\in \Gamma(\C TS)$ such that $$\dot \sigma= (\mathscr L_{X}c_1,  \mathscr L_{X} q_1, \mathscr L_{X}\overline{c_2}, \mathscr L_{X}\overline {q_2}, \mathscr L_{X}u).$$
\end{prop}
\begin{proof}
Denote $\dot \sigma=(\dot c_1, \dot q_1,\dot{\overline{c_2}}, \dot{\overline{q_2}}, \dot u)$. By Corollary \ref{cor: tangent to orbit in bundles}, there exists a unique $X\in \Gamma(\C TS)$ such that $(\dot c_1,  \dot q_1, \dot{\overline{c_2}}, \dot{\overline{q_2}})= (\mathscr L_{X}c_1, \mathscr L_{X} q_1,  \mathscr L_{X}\overline{c_2},\mathscr L_{X}{q_2})$. The fact that $\Tau$ is equivariant for the action of $\Diff_0(S)$ implies that, for all $\mathrm X\in \Gamma(TS)$, $(\mathscr L_{\mathrm X} c_1,\mathscr L_{\mathrm X} q_1, \mathscr L_{\mathrm X} \overline{c_2},  \mathscr L_{\mathrm X} \overline {q_2}, \mathscr L_{\mathrm X} u)\in T_{\sigma}V$. Since $V$ is a complex submanifold, by seeing $X=Re(X)+i Im(X)$, we have that $(\mathscr L_{ X} c_1, \mathscr L_{ X} q_1, \mathscr L_{ X} \overline{c_2},\mathscr L_{ X} \overline {q_2}, \mathscr L_{ X} u)\in T_{\sigma}V$. Since the differential of the projection $\pi_0:V\to \mathrm M(S, (d_i))\times\mathrm M(\overline S, (d_i))$ is injective, we conclude that $\dot u= \mathscr L_{X}u$.
\end{proof}

Finally, we use complex Lie derivatives to prove a general diffeomorphism invariance result.
\begin{prop}\label{prop: descends through quotient}
Let $M$ be a complex manifold, $V\subset \mathrm{SOL}_{\mathscr{T}}^*(S)$ a $\mathrm{Diff}_0(S)$-invariant open subset, and $\widetilde{F}:V\to M$ a holomorphic map. Then $\widetilde{F}$ descends to a holomorphic map $F:\mathcal{V}\to M$ if and only if $\mathscr{T}$ is $\mathrm{Diff}_0(S)$-equivariant and $\widetilde{F}$ is invariant under the diagonal action of $\mathrm{Diff}_0(S)\times \mathrm{Diff}_0(S)$.
\end{prop}
\begin{proof}
    It is obvious that the equivariance and invariance are necessary, so assume that we have things. The invariance can be restated as saying that $\widetilde{F}$ is constant along the orbits of the diagonal in $\Diff_0(S)\times\Diff_0(S)$. Differentiating, we obtain that at $(c_1,q_1, \overline {c_2},q_2,u)\in \mathrm{SOL}_{\mathscr{T}}^*(S),$ for all $X\in \Gamma(TS),$ $d\widetilde{F} (\mathscr L_{\mathrm X}c_1,   \mathscr L_{\mathrm X}q_1,\mathscr L_{\mathrm X}\overline{c_2},\mathscr L_{\mathrm X}\overline{q_2}, \mathscr L_{\mathrm X}u)=0$. Since $\widetilde{F}$ is holomorphic, for all $X\in \Gamma(\C TS)$, still at the point $(c_1,q_1, \overline {c_2},q_2,u),$
\begin{align*}
d\widetilde{F}(\mathscr L_{ X}c_1, \mathscr L_{ X}q_1, \mathscr L_{ X}\overline{c_2}, \mathscr L_{ X}\overline{q_2}, &\mathscr L_{ X} u)
=d\widetilde{F}(\mathscr L_{ Re(X)}c_1, \mathscr L_{ Re(X)} q^1, \mathscr L_{ Re(X)} \overline{c_2},  \mathscr L_{Re(X)}\overline{q_2},\mathscr L_{ Re(X)} u) \\
&+id\widetilde{F}(\mathscr L_{ Im(X)}c_1, \mathscr L_{ Im(X)} q^1, \mathscr L_{ Im(X)} \overline{c_2},  \mathscr L_{Im(X)}\overline{q_2}, \mathscr L_{ Im(X)} u) =0.
\end{align*}
It follows that $\widetilde{F}$ is indeed constant along the $\mathrm{Diff}_0(S)\times \mathrm{Diff}_0(S)$ orbits, and hence descends to the quotient.
\end{proof}

\subsection{Complex affine Toda equations and holonomy}\label{sec: equations and holonomy}

We use the constructions of the above subsections in the context of complex harmonic $\GC$-bundles. We start with the following general fact, which we prove below.
\begin{prop}\label{prop: Fredholm and such}
Assume that $k>2$, $s>k+2$, and that $\mathscr T$ has the form 
 \begin{equation*}     
\mathscr{T} (c_1, q_1, \overline {c_2}, \overline{q_2}, u)= \Delta_{\hpair} u + F(c_1,q_1, \overline {c_2},\overline{{q_2}},u),
 \end{equation*} 
 for some holomorphic map $$F:  \mathrm M^s(S, (d_i))\times \mathrm M^s(\overline S, (d_i)) \times W^{k,2}(S,\C^l)\to W^{k-2,2}(S,\C^l).$$
 Then $\mathscr{T}$ is holomorphic, the assumption \eqref{eq: assumption nesting} holds, and for all $\sigma=(c_1, q_1, \overline {c_2}, \overline {q_2},  u)\in \mathrm{SOL}_{ \mathscr T,s, k}(S)$, the linearization $\partial_u\mathscr{T}$ is Fredholm of index $0$. 
\end{prop}
By Fredholmness, in the definition of  $\mathrm{SOL}^*_{ \mathscr T,s, K}(S)$, the ``is an isomorphism" can be replaced with ``is bijective," i.e., the bi-continuity is for free. Since the Fredholm index is $0,$ the word ``bijective" can be further replaced by ``injective." 

The proposition is a consequence of the holomorphicity result from the paper \cite{ESholodependence}, and the elliptic estimates for $\Delta_h$ that we proved in \cite{ElSa} (see \cite[Section 2]{ElSa} for more general statements). Denote $W^{k,p}$ norms on an open subset $V$ by $||\cdot||_{W^{k,p}(V)}$.
\begin{prop}\label{interiorestimate}
Let $V\subset V'\subset S$ be open disks with $\overline{V}\subset V'$ and set $k>2$ and $s>k+2$. Let $h=\vb*h(c_1, \overline{c_2})$ be a Bers metric with $c_1,c_2\in \mathcal{C}^s(S,c_0)$, and let $u\in L^{2}(V',\C)$ and $v\in W^{k,2}(V',\C)$, with $u$ satisfying $\Delta_h u=v$ in the weak sense. Then $u\in W^{k+2,2}(V,\C)$ and there exists $C=C(V,V',c_1,\overline{c_2},k)$ such that 
$$
    ||u||_{W^{k+2,2}(V)}\leq C(||u||_{W^{k,2}(V')}+||v||_{W^{k,2}(V')}).$$
    \end{prop}
By a basic covering argument, one can replace $V$ and $V'$ with $S$, at the cost of enlarging the constant $C$. In the proof below, we work with a constant $C$, always depending on specified parameters, which is allowed to change during the course of the proof. 
\begin{proof}[Proof of Proposition \ref{prop: Fredholm and such}]
To see that $\mathscr{T}$ is holomorphic, we just need to observe that $(c_1,\overline{c_2},u)\mapsto \Delta_{\hpair} u$ is holomorphic. Since, by \cite[Theorem B]{ESholodependence}, $\hpair$ varies holomorphically with $(c_1,\overline{c_2})$ in the relevant Fréchet space of complex metrics, holomorphicity of $(c_1,\overline{c_2},u)\mapsto \Delta_{\hpair} u$ follows from the Koszul formula for the Levi-Civita connection of $\hpair$.
 
For the other two claims we will show that there exists $C=C(\overline{c_1},q_1, \overline {c_2},\overline{q_2},k)>0$ such that for all $\dot{u}\in C^\infty(S,\C^l)$,
\begin{equation}\label{eq: elliptic estimate multivariable}
    ||\dot{u}||_{W^{k+2,2}}\leq C(||\dot{u}||_{W^{k,2}}+||(\partial_u\mathscr{T})(\dot u)||_{W^{k,2}}).
\end{equation}
Here, we write $||\cdot||_{W^{k,2}}$ for $||\cdot||_{W^{k,2}(S)}.$ By totally standard arguments (see, for instance, \cite[\S 10.4]{Ni}), the Fredholmness for $\partial_u \mathscr{T}$ will follow. As well, with \eqref{eq: elliptic estimate multivariable}, verifying the assumption (\ref{eq: assumption nesting}) is simple: let $\sigma=(c_1,q_1, \overline {c_2},\overline{q_2},u) \in \mathrm{SOL}_{ \mathscr T,s, k}(S)$, so that $u\in W^{k,2}(S,\C^l)$. Then \eqref{eq: elliptic estimate multivariable} shows that $u\in W^{k+1,2}(S,\C^l)$, and continuing inductively we obtain that $u\in W^{s-1,2}(S,\C^l).$ Hence, (\ref{eq: assumption nesting}) holds for $(s_n, k_n)=(n, n+1)$.

We're left to actually prove (\ref{eq: elliptic estimate multivariable}). Toward this, taking $\dot{u}\in C^\infty(S,\C^l$), since $F$ is holomorphic,
    $$||\partial_u F (c_1,q_1, \overline {c_2},\overline{{q_2}},\dot{u})||_{W^{k,2}}\leq C(c_1,q_1, \overline {c_2},\overline{q_2})||\dot{u}||_{W^{k,2}}.$$ 
Hence, if $u=(u_1,\dots, u_l)$ and $F=(F_1,\dots, F_l),$ by Proposition \ref{interiorestimate}, for every $i,$
    \begin{align*}
      ||\dot{u}_i||_{W^{k+2,2}}&\leq C(||\dot{u}_i||_{W^{k,2}}+||\Delta_h u_i||_{W^{k,2}}) \\
      &\leq C(||\dot{u}_i||_{W^{k,2}}+||\Delta_h u_i + \partial_u F_i(c_1,q_1, \overline {c_2},\overline{{q_2}},\dot{u})||_{W^{k,2}} + ||\partial_u F_i(c_1,q_1, \overline {c_2},\overline{{q_2}},\dot{u})||_{W^{k,2}})  \\
      &\leq C(||\dot{u}_i||_{W^{k,2}}+||\Delta_h u_i + \partial_u F_i(c_1,q_1, \overline {c_2},\overline{{q_2}},\dot{u})||_{W^{k,2}}),
    \end{align*}
    where in the last line we enlarged $C$. Summing up over the indices, we obtain the estimate (\ref{eq: elliptic estimate multivariable}), and hence the proposition follows.
\end{proof}

Now, let $G$ be a split real simple Lie group of adjoint type and of rank at least $2$, or the semisimple Lie group $\mathrm{PSL}(2,\R)^2$. Let $d=d_G$ be the Coxeter number of $G$ (recall that for $G=\mathrm{PSL}(2,\R)^2$, we set $d_G=2$). In our framework above, we now set $(d_i)_{i=1}^k=(d),$ so that the bundles in question are $\mathrm{M}(S,G)$ and $\mathcal{M}(S,G)$ (recall the notation from Section \ref{sec: Labourie paper}). Fix an ordering of the simple roots $\Pi=\{\alpha_1,\dots, \alpha_l\},$ so that we can consider an ordered tuple of functions $(u_\alpha)_{\alpha\in \Pi}$ as a function from $S$ to $\C^l.$ For $\mathrm{PSL}(2,\R)^2$, we make a little abuse of notation and just set $\Pi=\{\alpha\}$. Working in the setup of Proposition \ref{prop: Fredholm and such}, consider the holomorphic operator $$\mathscr T_G: \mathrm M(S,G) \times  \mathrm M(\overline S, G) \times W^{k,2}(S,\C^l)\to W^{k-2,2}(S,\C^l),$$  $$\mathscr{T}_G (c_1, q_1, \overline {c_2}, \overline{q_2}, u)= \Delta_{\hpair} u + F(c_1,q_1, \overline {c_2},\overline{{q_2}},u),$$ where
\begin{equation}\label{eq: the F}
    F(c_1,q_1, \overline {c_2},\overline{{q_2}},u)= (\Delta_h  u_\alpha-4\sum_{\beta\in \Pi}\nu(\alpha,\beta)r_\beta e^{u_\beta} +4\nu(\alpha,\delta)\frac{q_1\overline{q_2}}{h^d}\prod_{\gamma\in \Pi} (-\frac{e^{- n_\gamma u_{\gamma}}}{\nu(\gamma,\gamma)^{n_\gamma}} )+ 2)_{\alpha\in \Pi}.
\end{equation}
 Solving for $F=0$ gives solutions to (\ref{eq: CAT}) that admit global logarithms (of course, the real solutions have this property);
$F$ is related to (\ref{eq: CAT}) by the substitution $u_\alpha\mapsto e^{u_\alpha} \nu(\alpha,\alpha)=U_\alpha$. Normalizing by $\nu(\alpha,\alpha)$ is minor, but makes the equation easier to work with on the real locus.

We verify that $F$ is holomorphic. By Hartog's Theorem on separate holomorphicity, it suffices to check in $c_1,q_1, \overline {c_2},\overline{q_2}$, and $u$ separately. The expression above is clearly holomorphic in the input data $q_1,\overline{q_2},$ and $u,$ and for $c_1$ and $\overline{c_2}$, the result follows from the holomorphicity of the Bers metric map from \cite{ESholodependence}. Hence, we can apply Proposition \ref{prop: Fredholm and such} to $\mathscr{T}_G$ and define the corresponding spaces $\mathrm{SOL}_{\mathscr T_G, l, 
K}(S)$, $\mathrm{SOL}^*_{\mathscr T_G, l, 
K}(S)$ together with the $C^\infty$ analogues $\mathrm{SOL}_G(S):=\mathrm{SOL}_{\mathscr T_G}(S)$ and $\mathrm{SOL}^*_G(S):=\mathrm{SOL}_{\mathscr T_G}^*(S)$.
Observe that $\Tau_G$ is $\Diff_0(S)$-equivariant, and that $\Diff_0(S)$ acts on both $\mathrm{SOL}_G(S)$ and $\mathrm{SOL}^*_G(S)$.

Now, by Theorem C, every element in $\mathrm{SOL}_G(S)$ defines a complex harmonic $G$-bundle on the same underlying principal $K^{\C}$-bundle $P_{K^{\C}}$, and the holonomy of the corresponding flat connection on $P_{G^{\C}}=P_{K^{\C}}\times_{K^{\C}} G^{\C}$ defines a map $\widetilde{\mathrm{hol}}: \mathrm{SOL}_G(S)\to \mathrm{Hom}(\pi_1(S), G^{\C})/G^{\C}$ (non-Hausdorff conjugation quotient). The case $G=\mathrm{PSL}(2,\R)^2$ needs a bit more elaboration: for quadratic differentials $q_1,\overline{q_2}$, the equation (\ref{eq: CAT}) is the same for $(q_1,\overline{q_2})$ and $(-q_1,-\overline{q_2}),$ so in this case we take the complex harmonic $G$-bundle corresponding to the product of two maps with differentials $(q_1,\overline{q_2})$ and $(-q_1,-\overline{q_2})$ respectively. 
Let $V\subset \mathrm{SOL}^*_G(S)$ be the subset of points such that the holonomy of the corresponding complex harmonic $G$-bundle is irreducible and simple, and restrict $\widetilde{\mathrm{hol}}$ to $V$. 
\begin{prop}
\label{prop: hol descends to mathbb SOL}
    The subset $V$ is open, and the map $\widetilde{\mathrm{hol}}:V\to \chi^{\mathrm{an}}(\pi_1(S),G)$ is holomorphic.
\end{prop}
Knowing that $V$ is open, we denote $\mathcal{SOL}_G(S):=\mathcal{V}$.
By Proposition \ref{prop: descends through quotient}, the map $\widetilde{\mathrm{hol}}$ descends to a holomorphic map $\mathrm{hol}:\mathcal{SOL}_G(S)\to \chi^{\mathrm{an}}(\pi_1(S),G).$

To prove the proposition, we do not need to cite any general result asserting that flat connections identify biholomorphically with representations. It suffices to argue directly and show that, if we pick a cover by trivializations, the connection forms vary holomorphically as $C^1$ $1$-forms valued in $\g$. Indeed, picking a point in $P_{G^{\C}},$ given the connection forms, the pointed holonomy is determined by solving a well-known differential equation for a map from $[0,1]\to G$ that depends only on the connection form (see \cite[Chapter 11.6 and 13.8]{Taubes}), whose solutions are understood to depend holomorphically on the connection forms (see \cite[Chapter 8.3]{Taubes}). It follows that for every homotopy class $\gamma\in \pi_1(S)$, the parallel transport of our flat connection gives a holomorphic map to $G$. Patching these maps together over all of $\pi_1(S)$ gives a holomorphic map to $\mathrm{Hom}$. It follows that $V$ is open, since the irreducible and simple locus in $\mathrm{Hom}$ is open. We can then restrict to $V,$ and $\widetilde{\mathrm{hol}}$ is the composition with the quotient map to $\chi^{\mathrm{an}}.$

We record for use below that, in a trivialization that is holomorphic for $P_{G^{\C}}\to (S,c_1)$, for simple $G,$ the flat connection associated with $\sigma=(c_1,q_1, \overline {c_2},\overline{q_2},u) \in \mathrm{SOL}_G(S)$ is
\begin{equation}\label{eq: flat connection SOL}
   A_\sigma := A^{\mathcal{I}}+ (\widetilde{e}+q_1x_\delta) + (\sum_{\alpha \in \Pi} \lambda U_\alpha r_\alpha^{1/2}(x_\alpha \otimes dz) \otimes d\overline{w}+ \lambda^{-d+1} U_{-\delta}\overline{q_2}(w)(x_{-\delta}\otimes dz^{-d+1} )\otimes d\overline{w}),
\end{equation}
where $\mathcal{I}$ is the relevant isomorphism solving (\ref{eq: bi-Hitchin flatness}) and, as above, $U_\alpha=e^{u_\alpha}\nu(\alpha,\alpha)$. For $\mathrm{PSL}(2,\R)^2,$ we take a product of two flat connections of the form above.
\begin{proof}[Proof of Proposition \ref{prop: hol descends to mathbb SOL}]
 By fixing $c_0\in \mathcal C(S)$ and a global holomorphic coordinate $z_0:\widetilde S\to \mathbb H$, we biholomorphically identify $\mathcal C(S)$ with the $\pi_1(S)$-invariant Beltrami differentials on $(\widetilde S,c_0)$, which allows us to holomorphically associate to each $c_1\in \mathcal C(S)$ a global coordinate $z_{c_1}$ on $\widetilde S$. 

  To prove the proposition, we follow the procedure outlined directly above, showing that, in a fixed trivialization, the connection forms vary holomorphically. One minor complication is that, for each $c_1$, \eqref{eq: flat connection SOL} describes a connection form using a trivialization that depends on $c_1$. To move to a common trivialization, the isomorphism 
    \begin{equation*}
    \begin{split}
    \mathcal K_{c_1} &\to \mathcal K_{c_0}\\
dz_{c_1}&\mapsto \frac{\vb*{h}(c_0, \overline{c_0})}{\vb* h(c_1, \overline{c_0})} dz_{0},
\end{split}
\end{equation*}
induces an automorphism of $P_{G^{\C}}$ (recall the construction from Section \ref{sec: bi-Hitchin}) that takes holomorphic trivializations for $P_{G^{\C}}\to (S,c_1)$ to holomorphic trivializations for $P_{G^{\C}}\to (S,c_0)$. Note that, in the procedure above, we track the connection forms and a point on $P_{G^{\C}}$. The isomorphisms move the point on $P_{G^{\C}},$ but they do so in a holomorphic fashion, so we still just need to check that the connection forms vary holomorphically.

Working in the fixed trivialization, the full connection form is the connection form of $A^{\mathcal{I}}$ plus the Higgs field pieces. By \cite[Theorem B]{ESholodependence}, the map $\ccpair\mapsto \hpair=\lambda(z_{c_1},\overline{z_{c_2}})dz_{c_1}d\overline{z_{c_2}}$ is holomorphic in the relevant Fréchet space of $C^\infty$ tensors, and it follows that the Higgs field pieces vary holomorphically.

For $A^{\mathcal{I}},$ it follows from Proposition \ref{prop: h valued} part (2) that the connection forms on $\g_{-\alpha}\otimes \mathcal{K}_{c_0}^{-1}$, $\alpha \in \Pi,$ determine the connection forms on the whole bundle. Thus, it suffices to verify that the restrictions to $\g_{-\alpha}\otimes \mathcal{K}_{c_0}^{-1}$, $\alpha\in \Pi,$ described explicitly in Proposition \ref{prop: K^c connection form}, vary holomorphically.
Using \cite[Theorem B]{ESholodependence} again, one easily sees that $\ccpair\mapsto\lambda(z_{c_1},\overline{z_{c_2}})$ and 
$$
\partial_J \log(\lambda): \  \ccpair\mapsto \frac{\partial_{\overline{z_{c_2}}} \log(\lambda(z_{c_1},\overline{z_{c_2}}))}{\partial_{\overline{z_{c_2}}} z_{c_1} }
$$ are holomorphic. Using the explicit description for the connection forms from Proposition \ref{prop: K^c connection form}, we indeed deduce the holomorphicity. This completes the proof.
\end{proof}

\subsection{The real locus}\label{sec: the real locus}
Let $(c, q)\in \mathrm M(S,G)$. According to Proposition \ref{prop: real Hitchin ex and un}, there exists a unique $u\in C^{\infty}(S,\mathbb R^l)$ such that $(c,q,\overline c, \overline q, u)\in \mathrm{SOL}_G(S)$. By construction, the corresponding complex harmonic $G$-bundle has Hitchin holonomy in the real group $G.$ The map $\mathcal{L}_G$ from the introduction is obtained by associating  $(c, q)$ to the holonomy and then factoring through the $\mathrm{Diff}_0(S)$ action. In order to put this map into the framework with which we are proving Theorems B and B', we prove the following. 
\begin{prop}
\label{prop: real locus}
   With notations as above, $(c,q,\overline c, \overline q, u)\in \mathrm{SOL}^*_G(S)$. 
\end{prop}
The embedding $\mathrm{M}(S,G)\to \mathrm{M}(S,G)\times \mathrm{M}(\overline{S},G)$, $(c,q)\mapsto (c,q,\overline c, \overline q)$, descends to an embedding $\mathcal{M}(S,G)\to \mathcal{M}(S,G)\times \mathcal{M}(\overline{S},G)$, whose image we denote by $\Delta.$ Proposition \ref{prop: real locus} then shows that we have an embedding of $\Delta,$
\begin{equation*}
\mathcal M(S,G)\times \mathcal M(\overline S,G)\supset \Delta \xrightarrow[\sim]{\ \mathrm s_0\ } \mathcal{SOL}_G(S).
\end{equation*}
Since the solutions are real, we denote the image by $\mathbb R\mathcal {SOL}_G(S).$

In order to prove Proposition \ref{prop: real locus}, we use the following maximum principle for elliptic systems. The lemma below is just a special case of \cite[Lemma 3.1]{Dai2018}.
\begin{lem}[Lemma 3.1 in \cite{Dai2018}]\label{lem: maximum principle}
    Let $X$ be a closed Riemannian manifold and for $c_{ij}:X\to \R$, $1\leq i,j\leq n,$ $C^2$ functions such that
    \begin{enumerate}
        \item (cooperative) $c_{ii}\geq 0$ and $c_{ij}\leq 0$ for $i\neq j$,
        \item (column-diagonally dominant) $\sum_{i=1}^n c_{ij}\geq 0$ for all $j$,
        \item (fully coupled) there is no parition $\{1,\dots,n\}=A\cup B$ such that $c_{ij}=0$ for $i\in A$, $j\in B.$
    \end{enumerate}
    Suppose $u_i:X\to \R$ are $C^2$ functions such that $$\Delta u_i =\sum_{j=1}^n c_{ij}u_j.$$ If $\sum_{i=1}^n u_i\geq 0,$ then either: a) $u_i>0$ for all $i$ or b) $u_i\equiv 0$ for all $i$.
\end{lem}

\begin{proof}[Proof of Proposition \ref{prop: real locus}]
By Proposition \ref{prop: Fredholm and such} and the comments thereafter, it suffices to show that for some (and hence every) $k>2$, the linearization $\partial_u \mathscr{T}: W^{k,2}(S,\C^l)\to W^{k-2,2}(S,\C^l)$ is injective. Linearizing the expression in (\ref{eq: the F}),  the task is to show that, given a point $(c,q)\in \mathrm{M}(S,G)$ with associated (complexified) hyperbolic metric $h=\lambda |dz|^2$ and real solution vector $(u_\alpha)_{\alpha\in \Pi}$, any complex-valued function $(\dot{u}_\alpha)_{\alpha\in\Pi}$ satisfying 
\begin{equation}\label{eq: real linearization}
    \Delta_h \dot{u}_\alpha = 2\sum_{\beta\in \Pi}\nu(\alpha,\beta) r_\beta e^{u_\beta}\dot{u}_\beta-2\nu(\alpha,\delta)\frac{|q|^2}{\lambda^d}(\sum_{\gamma\in \Pi}n_\gamma \dot{u}_\gamma)\Big ( \prod_{\gamma \in \Pi} (-\nu(\gamma,\gamma))^{-n_\gamma}\Big )e^{-\sum_{\gamma\in \Pi}n_\gamma u_\gamma}
\end{equation}
for all $\alpha\in \Pi$ must vanish identically. Since $h$ is the complexification of a Riemannian metric, and all of the coefficients on the right hand side of (\ref{eq: real linearization}) are real, the equation (\ref{eq: real linearization}) commutes with taking real and imaginary parts. Thus, the real and imaginary parts of $(\dot{u}_\alpha)_{\alpha\in\Pi}$ also solve (\ref{eq: real linearization}), and we can assume that $(\dot{u}_\alpha)_{\alpha\in\Pi}$ is real. The non-existence of non-zero real solutions should morally be a consequence of the fact that the non-abelian Hodge correspondence can be framed as a diffeomorphism between appropriate moduli spaces. Via Lemma \ref{lem: maximum principle}, we give a direct proof.

In order to apply Lemma \ref{lem: maximum principle}, we add an equation to the system. Set $\dot{u}_{-\delta}=-\sum_{\gamma\in \Pi} n_\gamma \dot{u}_\gamma$. Substituting $u_{-\delta}=\prod_{\gamma\in \Pi} \frac{-u_\alpha^{-n_\alpha}}{\nu(\gamma,\gamma)^{n_\gamma}},$ the system becomes 
 $$\Delta_h \dot{u}_\alpha = 2\sum_{\beta\in \Pi}\nu(\alpha,\beta)r_\beta e^{u_\beta}\dot{u}_\beta-2\nu(\alpha,\delta)\frac{|q|^2}{\lambda^d}\dot{u}_{-\delta}e^{u_{-\delta}}, \hspace{1mm} \alpha\in \mathcal{Z}.$$
Setting $n_{-\delta}=1$, we apply Lemma \ref{lem: maximum principle} to the collection of functions $f_\alpha = n_\alpha  \dot{u}_\alpha,$ $\alpha\in \mathcal{Z}$, which satisfy
$$\Delta f_\alpha = \sum_{\beta \in \Pi} \nu(n_\alpha \alpha,n_\beta^{-1}\beta)r_\beta e^{u_\beta} f_\beta - \nu(n_\alpha \alpha, n_{-\delta}^{-1} \delta)e^{u_{-\delta}}\frac{|q|^2}{\lambda^d} f_{-\delta}, \hspace{1mm} \alpha\in\mathcal{Z}.$$ The goal is to show that every $f_\alpha$ is identically zero.
We check the conditions of Lemma \ref{lem: maximum principle}.
\begin{enumerate}
    \item The system is cooperative because of the basic properties of the Killing form and simple roots listed in Proposition \ref{lem: commutators}: $\nu(n_\alpha\alpha,n_\beta^{-1}\beta)e^{v_\beta}=n_\alpha n_\beta^{-1}\nu(\alpha,\beta)e^{v_\beta}$ is non-positive for $\alpha\neq \beta$ and positive for $\alpha=\beta.$
    \item For the column-diagonally dominant condition, note that $\sum_{\alpha\in \mathcal{Z}}n_\alpha\alpha=0,$ and hence, for all $\beta\in \Pi$, $\sum_{\alpha\in \mathcal{Z}}\nu(n_\alpha\alpha,n_\beta^{-1}\beta)r_\beta e^{v_\beta}=\nu(\sum_{\alpha\in \mathcal{Z}}n_\alpha\alpha,n_\beta^{-1}\beta)e^{v_\beta}=0.$ For $\beta=-\delta,$ the calculation is the same if we just replace $r_\beta e^{u_\beta}$ by $e^{u_{-\delta}}|q|^2$.
    \item The system is fully coupled because the Dynkin diagram associated with a simple complex Lie group is connected.
\end{enumerate}
Finally, note that $f_{-\delta}=-\sum_{\gamma \in \Pi} n_\gamma \dot{u}_\gamma= - \sum_{\gamma \in \Pi} f_\gamma,$ and hence $\sum_{\alpha \in \mathcal{Z}}f_\alpha=0$ identically. Thus, Lemma \ref{lem: maximum principle} returns that either each $f_\alpha$ is positive or all are identically zero. Since $\sum_{\alpha \in \mathcal{Z}}f_\alpha=0$, the former cannot occur. We deduce that every $f_\alpha=0,$ and the main result follows.
\end{proof}

\subsection{The marginal locus}\label{sec: marginal locus}
A restatement of Proposition \ref{prop: existence on marginal locus} is that there exists a unique real constant function $u_0=(u_\alpha)_{\alpha\in \Pi}$ that solves \eqref{eq: the F} for all points in $\mathrm M(S,G)\times \mathcal C(\overline S)$ and $\mathcal C(S)\times\mathrm M(\overline S,G)\subset \mathrm M(S)\times \mathrm M(\overline S)$. This function determines holomorphic sections from $\mathrm{M}(S,G)\times \mathcal{C}(\overline{S})$ and $ \mathcal{C}(S)\times \mathrm{M}(\overline{S},G)$ to $\mathrm{SOL}_G(S)$ (explicitly, for the first space, $(c_1,q_1,\overline{c_2},0)\mapsto (c_1,q_1,\overline{c_2},0,u_0)$). In this subsection, we show that these sections partially descend to the finite dimensional quotients, the main result being Theorem \ref{thm: section from the marginals} below.

Recall that the zero locus $Z$ of a non-constant holomorphic function has \textbf{pure codimension} 1, which means that every element in $Z$ has a neighbourhood $W$ such that $W\cap Z$ contains a submanifold of dimension $1$ and no manifold of dimension $2$.

\begin{thm}
\label{thm: section from the marginals}
There is a $\mathrm{MCG}(S)$-invariant complex analytic subset $\mathrm{Sing}_G\subset \TSTS$, which is disjoint from the diagonal and either empty or of pure codimension 1, such that the sections above determine holomorphic sections
\begin{align*}
\mathrm s&: (\mathcal M(S,G)\times \mathcal T(\overline S) )\setminus \mathcal S_G\to \mathcal{SOL}_G(S)\\
\overline{\mathrm s} &: (\mathcal T(S)\times \mathcal M(\overline S, G))\setminus \mathcal S_G\to \mathcal{SOL}_G(S),
\end{align*}
where $\mathcal S_G\subset\mathcal M(S,G)\times\mathcal M(\overline S, G)$ is the preimage of $\mathrm{Sing}_G$ through the projection to $\TSTS$.
\end{thm}
Observe that $\mathcal S_G$ is also either empty or of pure codimension 1. In particular, $(\mathcal M(S,G)\times \mathcal T(\overline S) )\setminus \mathcal S_G$ and $ (\mathcal T(S)\times \mathcal M(\overline S,G))\setminus \mathcal S_G$ are both connected.

The content of the above theorem is that the image of a big connected open subset of $\mathrm{M}(S,G)\times \mathcal{C}(\overline{S})\cup \mathcal{C}(S)\times \mathrm{M}(\overline{S},G)$ lands in $\mathrm{SOL}_G^*(S),$ and that taking the quotient to $\mathcal{SOL}_G(S)$ is not an issue. Note that, by Theorem D and Proposition \ref{prop: smooth locus}, the image points lie in the smooth holonomy locus. On the loci  in question, the linearization of \eqref{eq: the F} is 
\[\partial_u\Tau_G= \Delta_{\hpair}-L_G,\] where $L_G=(L_{ij})$ is an invertible matrix depending only on $G$. Since it makes no difference in the proof, we prove the main results toward Theorem \ref{thm: section from the marginals} for more general operators $\Delta_{\hpair}-L,$ where $L:\C^l\to \C^l$ is a linear endomorphism. As with $\partial_u \mathscr{T}_G,$ the estimate (\ref{eq: elliptic estimate multivariable}) holds and ensures the operator $\Delta_{\hpair}-L$ is Fredholm and has index zero on any $W^{k,2}$.

For a family of operators $\Delta_{\hpair}-L$ that are isomorphisms when $c_1=c_2$, such as our linearizations in question, it follows immediately from (a slight generalization of) the Analytic Fredholm Theorem that if we choose a finite dimensional complex poly-disk in $\mathcal{C}(S)\times \mathcal{C}(\overline{S)}$ containing a pair $(c_1,\overline{c_1})$, then, for $(c_1,\overline{c_2})$ in the complement of a codimension $1$ subset of that disk, $\Delta_{\hpair}-L$ is an isomorphism. To promote to an infinite dimensional setting, we show that if $\Delta_{\hpair}-L$ is invertible at a point $(c_1,\overline{c_2}),$ then the same holds on most of the $\mathrm{Diff}_0(S)\times \mathrm{Diff}_0(S)$-orbit. In the proposition below, we fix an endomorphism $L:\C^l\to \C^l.$

\begin{prop}
\label{prop: invertibility dense in fiber}
     The intersection of $U_L=\{\ccpair\in \CSCS\ |\ \Delta_{\hpair}-L  \text{ is invertible}\}$ with a fiber of the projection $\CSCS\to \TSTS$ is either empty or an open, connected, and dense subset of the fiber. 
\end{prop}
We point out immediately that $U_L$ is open, since the complement is easily seen to be closed. Indeed, let $(x_n)_{n=1}^\infty$ is a sequence in the complement of $U_L$ that converges to a point $x_{\infty}.$ By the elliptic estimate (\ref{eq: elliptic estimate multivariable}) and the Rellich-Kondrachov theorem, if $(v_n)_{n=1}^\infty \subset W^{k,2}(S,\C^l)$ is a unit-normalized sequence such that $v_n$ is in the kernel corresponding to $x_n$, then $v_n$ converges to a kernel vector for the map corresponding to $x_{\infty}$. See Lemma 5.16 in \cite{ElSa} for a more detailed example of this standard argument in a slightly different context.

    The proof of Proposition \ref{prop: invertibility dense in fiber} relies on several results and techniques from \cite{ElSa}. We say that a pair $\ccpair \in \CSCS$ is \textbf{co-real analytic} if the Beltrami form on $(S,c_1)$ representing of $c_2$ is real analytic, or equivalently, if $c_1$ and $\overline{c_2}$ determine the same real analytic structure on $S$. For any $c_1 \in \mathcal C(S)$, every element of $\mathcal T(\overline S)$ admits a representative $\overline{c_2}$ such that $\ccpair$ is co-real analytic: one can choose the unique representative $\overline{c_2}$ for which the identity map $(S, c_1) \to (S, h_2)$ is harmonic, where $h_2$ is the hyperbolic metric associated with $\overline{c_2}$ (see \cite{W2}).
  
    Determining whether $\Delta_{\hpair} - L$ is invertible is more approachable when working with co-real analytic pairs because we can use the complex transport from \cite{ElSa}. Specifically, certain tensors and functions associated with $(c_1,\overline{c_2})$ become real analytic, and one can short-time flow such objects along complex vector fields (see \cite[Section 4.3]{ElSa}), as we will do with functions below. Before starting proofs, we recall essential results that we will apply to co-real analytic pairs. A standard reference for this material is \cite{Bookanalytic} (see also \cite[Section 4.3]{ElSa}).


Consider $\Diff(S)$ as an open subset of the space $C^{\infty}(S, S)$ of smooth maps from $S$ to itself. The subset $\Diff^{\omega}(S) \subset \Diff(S)$, consisting of real analytic diffeomorphisms with respect to a fixed real analytic structure on $S$ is dense for the $C^{\infty}$ topology \cite[ Corollary 11.8]{Tsu}. Moreover, the inclusion $\Diff^{\omega}(S) \hookrightarrow \Diff(S)$ is a homotopy equivalence \cite[Proposition 11.10]{Tsu}. In particular, $\Diff_0^{\omega}(S) = \Diff_0(S) \cap \Diff^{\omega}(S)$, is connected and dense within $\Diff_0(S)$. Using local charts for $\Diff^{\omega}(S)$ modeled on real analytic vector fields \cite[$\S$ 30.12]{Bookanalytic}, it follows that $\Diff^{\omega}(S)$ is locally path connected via real analytic paths $t \mapsto \upphi_t$.

Turning to $\CSCS$, using local trivializations modeled on open subsets of $\TSTS \times \Diff_0(S) \times \Diff_0(S)$, the density and connectedness of real analytic diffeomorphisms implies that the subset of co-real analytic pairs is dense and connected in $\CSCS$. Moreover, the diagonal action of $\Diff_0(S)$ on $\CSCS$ preserves the locus of co-real analytic pairs. 

\begin{lem}
\label{lemma: solution in co-real analytic is analytic}
    Let $\ccpair$ be co-real analytic. If $(u_1, \dots, u_l)$ solves $(\Delta_{\hpair} -L)(u_1,\dots u_l)=0$, then it is a real-analytic function for the induced real analytic structure.
\end{lem}
    \begin{proof}
  By linearity of $\Delta,$ for all linear isomorphisms $Q:\C^l\to \C^l$, replacing $L$ with $QLQ^{-1}$ produces a solution $Qu$. So, up to triangularizing $L,$ it suffices to prove the statement under the assumption that $L$ is upper triangular.
    Since, in local coordinates, the coefficients of the Bers Laplacian are real analytic, by \cite[Theorem D]{ElSa}, the last equation $\Delta_{\hpair} u_l -L_{ll} u_l=0$ implies that $u_l$ is real analytic. By proceeding inductively backwards for $n=1,\dots, l-1$ and assuming that $u_{n+1},\dots, u_{l}$ are real-analytic, the same result implies that a solution $u_{n}$ to $\Delta_{\hpair}u_{n}=\sum_{j\ge n}L_{nj} u_j$ must be real analytic. 
\end{proof}

\begin{lem}
\label{prop: coreal analytic invariance of invertibility}
   Let $\ccpair$ and $(c_1', \overline{c_2}')$ be co-real analytic pairs corresponding to the same point in $\TSTS$. Then $\Delta_{\hpair}-L$ is invertible if and only if $\Delta_{\vb*h(c'_1, \overline{c_2}')}-L$ is.
\end{lem}
\begin{proof}We prove the contrapositive statement. The proof is essentially the same as that of Proposition 5.20 in \cite{ElSa}. Without loss of generality, we can assume $c_1=c_1'$ since $(\Delta_{\hpair}-L)(  v)=0$ if and only if $(\Delta_{\vb*h(\upphi^*c_1, \upphi^*{\overline{c_2}})}-L)(\upphi^*  v)=0$ for all $\upphi\in \mathrm{Diff}(S)$. In particular, $c_1, \overline{c_2}$, and $ \overline{c_2}'$  induce the same real analytic structure. Let $t\mapsto \upphi_t$, $t\in [0,1]$ be a piecewise real analytic isotopy such that $\overline{c_2}'=\upphi_1^*\overline{c_2}$. Assume $\Delta_{\hpair}-L$ has non-trivial kernel and let $T$ be maximal such that $\Delta_{\vb*h(c_1, \upphi^*_t\overline{c_2})}-L$ has non-trivial kernel for all $t\in[0,T]$ (such $T$ exists because $U_L$ is open). Assume for the sake of contradiction that $T<1$, and let $(v_1^T, \dots, v_l^T)$ be a non-trivial solution of $\Delta_{\vb*h(c_1, \upphi^*_T\overline{c_2})}-L$. By Lemma \ref{lemma: solution in co-real analytic is analytic}, $(v_1^T, \dots, v_l^T)$ is real analytic. By Proposition \ref{prop: tangent model CSCS}, there exists a unique path of vector fields $t\mapsto X_t$ such that $ \frac d {dt} (c_1, \upphi^*_t(\overline{c_2}))=\mathscr L_{X_t}(c_1, \upphi^*_t(\overline{c_2}))$. Observe that $t\mapsto X_t$ is a piecewise real analytic path of real analytic vector fields: indeed, locally, $X_t$ can be computed explicitly from the system of linear equations with real analytic coefficients
\[ \begin{cases}
    &d z(X_t)=0\\
    &d(\overline w\circ \upphi_t) (X_t)=d(\overline w\circ \upphi_t) (\dot \upphi_t),
\end{cases}
\]
where $z$ and $\overline{w}$ are local holomorphic coordinates for $c_1$ and $\overline{c_2}$ respectively.
As shown in Corollary 4.4 in \cite{ElSa}, $\mathscr L_{X_t} \vb*h(c_1, \upphi^*_t(\overline{c_2}))=\frac d {dt}  \vb*h(c_1, \upphi^*_t(\overline{c_2}))$. By the Cauchy-Kovalevskaya Theorem and Theorem 4.11 in \cite{ElSa}, there exists $\eps>0$ with $0<T-\eps<T+\eps<1$ and a unique path of functions $t\mapsto v^t=(v_k^t):S\to \C^l,$ mapping $T$ to $(v_k^T)$ from above, such that $\mathscr L_{X_t} v_k^t= \frac d{dt} v_k^t$ for all $t\in (T-\eps,T+\eps)$. As explained in Proposition 5.20 in \cite{ElSa}, the Laplacian of $v^t_k$ evolves according to the same evolution equation, namely, $\frac{d}{dt}\left(\Delta_{\vb*h(c_1, \overline{c_2}^t)} v^t_k\right)=\mathscr L_{X_t}\left(\Delta_{\vb*h(c_1, \overline{c_2}^t)} v^t_k\right)$, so we get that \[
\left(\frac{d}{dt}-\mathscr L_{X_t}\right)\left( (\Delta_{\vb*h(c_1, \overline{c_2}^t)}-L)(  v^t)\right)=0
\]
for all $t\in (T-\eps,T+\eps)$. Since $\left(\frac{d}{dt}-\mathscr L_{X_t}\right)(w^t)=0$ with $w^T=0$ has $w^t\equiv 0$ as a unique solution (see Theorem 4.11 in \cite{ElSa}) and $(\Delta_{\vb*h(c_1, \overline{c_2}^T)}-L)(  v^T)=0$, we conclude that $(\Delta_{\vb*h(c_1, \overline{c_2}^t)}-L)(  v^t)=0$ for $t\in (T-\eps,T+\eps)$. For $\eps$ small enough, we have $v^t\ne 0$ for all $t\in [T,T+\eps)$, hence $\Delta_{\vb*h(c_1, \overline{c_2}^t)}-L$ has non trivial kernel, contradicting the maximality of $T$.
\end{proof}

\begin{proof}[Proof of Proposition \ref{prop: invertibility dense in fiber}]
    Denote by $F$ any fiber of the projection $\CSCS\to \TSTS$. Assume $U_L\cap F$ is non-empty. As explained above, $U_L$ is open, and hence $U_L\cap F$ is 
    open in $F$. Necessarily, $U_L\cap F$ contains a co-real analytic representative. Indeed, if not, then $\Delta_{\hpair}-L$ has kernel for every co-real analytic pair $(c_1,\overline{c_2})$, and then by density of co-real analytic pairs, using that the the complement of $U_L$ is closed, it would have kernel for every pair in $F$.

    By Lemma \ref{prop: coreal analytic invariance of invertibility}, $F$ contains all co-real analytic pairs. Hence, $U_L\cap F$ is dense in $F$. Finally, since co-real analytic representatives define a connected subspace of $F$, and every neighbourhood of every point contains a co-real analytic representative, $U_L\cap F$ is connected.
\end{proof}
Let $\mathrm{Sing}_L\subset \mathcal T(S)\times \mathcal T(\overline S)$ be the subset of points at which $T_{\hpair}^L:\Delta_{\hpair}-L$ is not invertible for any (and hence all) co-real analytic representatives $\ccpair$. The final step before the formal proof of Theorem \ref{thm: section from the marginals} is to study this locus. 
As alluded to above, we use a generalized Analytic Fredholm Theorem.

\begin{thm}\label{thm: analytic fredholm}
   Let $X$ and $Y$ be complex Hilbert spaces and let $L(X,Y)$ be the space of continuous linear maps from $X$ to $Y$. Let $V\subset \C^n$ be an open and connected set and let $V\to L(X,Y)$, $z\mapsto A(z)$, be a holomorphic map such that for all $z\in V,$ $A(z)$ is Fredholm.  Then the set of $z\in V$ such that $A(z)$ is not invertible is either empty, the whole $V$, or a complex analytic subset of pure codimension 1. 
\end{thm}
\begin{proof}
    When $X=Y$, this is the Multidimensional Analytic Fredholm theorem from Taylor’s note \cite{Taylor}. In \cite{Taylor}, Taylor chooses an invertible Fredholm inverse to $A(z_0)$, say, $B:X\to X$. Setting $C(z)=BA(z),$ he explains that for $z$ in a small polydisk about $z_0$, the invertibility of $C(z)$ (which is equivalent to invertibility of $A(z)$) is equivalent to that of $I+K(z)$, where $K(z)$ is some well-chosen holomorphic family of compact operators from $X\to X$. He then proves the result for families of the form $I+K(z).$ 
    
    In our case, the minor modification is that, since $A(z)$ takes $X$ to a different Hilbert space $Y$, we should choose our Fredholm inverse $B$ for $A(z_0)$ to go from $Y\to X$. Then, $A(z)$ is invertible if and only if $BA(z)$ is, and the rest of the proof goes through word-for-word.
\end{proof}

\begin{prop}
\label{prop: singular locus complex analytic}
   The subset $\mathrm{Sing}_L$ of $\TSTS$ is either empty, the whole $\TSTS$, or a pure codimension-1 complex analytic subset.
\end{prop}
\begin{proof}
Assume that $\mathrm{Sing}_L$ is not empty. Let $([c^0_1],[\overline{c_2}^0])\in \mathrm{Sing}_L$ and let $(c_1^0, \overline{c_2}^0)$ be a co-real analytic representative. Consider a local holomorphic inverse $b$ of the projection map from a neighborhood in $\mathcal T(S)\times \mathcal T(\overline S)$ to $\mathcal C(S)\times\mathcal C(\overline S)$, whose image, which we call $V$, contains $(c^0_1,\overline{c_2}^0)$ and consists of co-real analytic representatives (this can be done, for instance, by using Theorem A from \cite{ElE}). By taking a holomorphic chart $\mathbb C^{6\mathrm g-6}\supset V\xrightarrow{\upvarphi} \mathcal T(S)\times\mathcal T(\overline S)$ centered at $([c^0_1],[\overline{c_2}^0])$ and composing with the local section $b$, we have a holomorphic map $ z\mapsto A_{ z}:=\Delta_{\vb*h(s\circ \upvarphi)( z)}-L$ to the space of continuous linear maps from $W^{k,2}(S, \C^l)$ to $W^{k-2,2}(S, \C^l)$. By Theorem \ref{thm: analytic fredholm}, up to shrinking $V$, $\mathrm{Sing}_L\cap \upvarphi(V)$ is either the whole $V$ or a codimension-1 analytic subset. Carrying out the above procedure over every point in $\mathrm{Sing}_L$, we conclude that, if $\textrm{int}(\mathrm{Sing}_L)= \emptyset$ and $\mathrm{Sing}_L \ne \emptyset$, then $\mathrm{Sing}_L$ is a pure codimension-1 complex analytic subset. Since the complement of a closed complex analytic subset is either empty or dense, we conclude that if $\mathrm{int}(\mathrm{Sing}_L)\ne \emptyset$, then $\mathrm{Sing}_L=\TSTS$.
\end{proof}

\begin{remark} 
Applying Proposition \ref{prop: singular locus complex analytic} for $l=1$, we get the eigenvalue equation for the Bers Laplacian $\Delta_h$.
\end{remark}

\begin{proof}[Proof of Theorem \ref{thm: section from the marginals}] Let $\mathrm{Sing}_G:= \mathrm{Sing}_{L_G}$ and $U_G:=U_{L_G}$ (as in Proposition \ref{prop: invertibility dense in fiber}). By Proposition \ref{prop: real locus}, the diagonal in $\TSTS$ is disjoint from $\mathrm{Sing}_G$. Hence, by Proposition \ref{prop: singular locus complex analytic}, $\mathrm{Sing}_G$ is either empty or complex analytic of pure codimension 1. The association $(c_1, q_1, \overline {c_2}, 0)\mapsto (c_1, q_1, \overline {c_2}, 0,u_0)$ defines a holomorphic injection from an open subset of $\mathrm M_G(S)\times \mathcal C(\overline S)$,
\begin{equation}
       \label{eq: before parametrization s}
    \bigcup_{\ccpair\in U_G} H^0(S,\mathcal K_{c_1}^{d_G})\times \{\overline{c_2}\}\to \mathrm{SOL}_G^*(S).
   \end{equation}
By definition, the image of the projection of the domain of \eqref{eq: before parametrization s} to $\mathcal M(S,G)\times \mathcal T(\overline 
    S)$ is the complement of $\mathcal S_G$, and the fibers are connected by Proposition \ref{prop: invertibility dense in fiber}. Therefore, \eqref{eq: before parametrization s} descends to a holomorphic embedding
\[
\mathrm s: (\mathcal M(S,G)\times \mathcal T(\overline S) )\setminus \mathcal S_G\to \mathcal{SOL}_G(S).
\]
The holomorphicity is easily seen by taking local holomorphic sections from $\mathcal M(S,G)\times \mathcal T(\overline S)$ to $\mathrm M(S,G)\times\mathcal C(\overline S)$. The holomorphic embedding $\overline{\mathrm s}$ is constructed in the same fashion.
\end{proof}

\subsection{The map $\mathcal L_G^{\C}$}\label{sec: proving Theorem B}
Finally, we prove Theorems B and B'. In an open subset $\Omega^0_G$ of $\mathcal M(S,G)\times \mathcal M(\overline S,G)$, we construct an inverse of the projection map $\pi: \mathcal{SOL}_G(S)\to \mathcal M(S,G)\times \mathcal M(\overline S,G)$. Then we compose with the holonomy map. We recall that the mapping class group $\mathrm{MCG}(S)$ acts by biholomorphisms on $\mathcal M(S,G)$, and hence on $\mathcal M(S,G)\times\mathcal M(\overline S,G)$. The group $\mathrm{MCG}(S)$ also acts on $\mathrm{SOL}_G^*(S)$ by direct summing the action on $\mathrm{M}(S,G)\times \mathrm{M}(\overline{S},G)$ with the natural action on functions. This last action descends to $\mathcal{SOL}_G(S)$. With respect to the actions on $\mathrm{SOL}_G^*(S)$ and $\mathcal{M}(S,G)\times \mathcal{M}(\overline{S},G)$, the projection map $\pi$ is clearly equivariant.

\begin{thm}
\label{thm: main inverse}
  There exists an open connected subset $\Omega^0_G\subset  \mathcal M(S,G)\times\mathcal  M_G(\overline S)$ containing the diagonal, $\mathcal M(S,G)\times\mathcal T(
    \overline S)\setminus \mathcal S_G$, and $\mathcal T(S)\times \mathcal  M_G(\overline S)\setminus \mathcal S_G$
    over which $\pi$ admits a global biholomorphic inverse $\varsigma: \Omega_G^0\to \mathcal{SOL}_G(S)$ that extends the maps $\mathrm s_0$, $\mathrm s$ and $\overline{\mathrm{s}}$ from above. Moreover, $\Omega^0_G$ can be chosen to be invariant for the action of $\mathrm{MCG}(S)$.
\end{thm}
\begin{proof} In order to ensure that $\Omega^0_G$ is invariant under $\mathrm{MCG}(S)$, we have to make a few careful choices. To simplify the constructions in the proof, it will be helpful to fix a $\mathrm{MCG}(S)$-invariant Riemannian metric on $\mathcal M(S,G)\times \mathcal M(\overline S,G)$. For instance, one could take the Riemannian metric on $\mathcal M(S,G)$ defined in \cite{Lab3}, and consider the product metric on $\mathcal M(S,G)\times \mathcal M(\overline S,G)$. 

In this proof, we denote $X=\Delta\cup \left( \mathcal M(S,G)\times\mathcal T(
    \overline S)\setminus \mathcal S_G\right)\cup \left( \mathcal T(S)\times \mathcal  M_G(\overline S)\setminus \mathcal S_G\right)$, over which there is a well-defined continuous map $\mathrm s_0\cup\mathrm s\cup \overline{\mathrm s}$ that inverts $\pi$.

We construct a $\mathrm{MCG}(S)$-invariant locally finite cover of $X$ with certain properties, together with matching local inverses of $\pi$. To this end, we first observe that, since $\mathrm{MCG}(S)$ acts properly discontinuously on $\mathcal T(S)$, it does the same on $\mathcal M(S,G)$ and $\mathcal M(S,G)\times \mathcal M(\overline S,G)$. It follows that $(\mathcal M(S,G)\times \mathcal M(\overline S,G))/\mathrm{MCG}(S)$ has the structure of an orbifold (see Proposition 4.12 in \cite{Lof} for easily-generalizable discussion in the case $G=\mathrm{PSL}(3,\R)$). We can thus choose a locally finite cover over $X/\mathrm{MCG}(S)$ that consists of connected orbifold charts, namely, open subsets such that the action of $\mathrm{MCG(S)}$ on the connected components of their preimages in $\mathcal{M}(S,G)\times \mathcal{M}(\overline{S},G)$ have finite stabilizers.  
Up to refining the cover, we can assume that each of its elements has a lift $V$ such that:
\begin{itemize}
  \item $V$ and $\overline{V}$ are geodesically convex balls for the chosen Riemannian metric, and
   \item there exists a continuous inverse $\varsigma: V\to \mathcal{SOL}_G(S)$ of $\pi$ that extends to $\overline V$ and that coincides with $\mathrm s_0\cup\mathrm s\cup \overline{\mathrm s}$ on $\overline{V}\cap X$.
\end{itemize}
For $V$ small enough, $\varsigma$ is simply the locally defined inverse to the local diffeomorphism $\pi$. Observe that if $V$ and $\varsigma$ are as above, then for all $\upphi\in \mathrm{MCG}(S)$, $\upphi(V)$ and $\upphi\circ \varsigma\circ \upphi^{-1}$ satisfy the same properties. This is because the action of $\mathrm{MCG}(S)$ is continuous and preserves the metric and $\pi$ is equivariant.

 As the action of $\mathrm{MCG}(S)$ on $\mathcal M(S,G)\times\mathcal M(\overline{S},G)$ is properly discontinuous, the connected components of the preimages of the elements in this cover and the corresponding sections (defined above) determine a $\mathrm{MCG}(S)$-invariant collection $\{(V_i, \varsigma_i)\}_{i\in I}$, where $\{V_i\}_{i\in I}$ is a locally finite cover of $X$. 
For use below, notice that, if $\varsigma_i$ and $\varsigma_j$ agree on some point of $\overline{V_i}\cap \overline{V_j}$, then they coincide on the whole $\overline{V_i}\cap \overline{V_j}$. Indeed, by taking local inverses, the condition of coinciding on a point is open and closed on $\overline{V_i}\cap \overline{V_j}$, and  $\overline{V_i}\cap \overline{V_j}$ is connected.

 The construction of the inverse map $\varsigma$ follows now by a standard technique (see, for instance, the suggested proof for the \emph{Generalized inverse function theorem} in Exercise 1.8.14 in \cite{GuillPoll}). Denote by $\Omega^0_G$ the connected component containing $X$ of the interior of
 \[
T=\left\{ v\in \bigcup_{i\in I} V_i \ \Bigg|\ \text{if $v\in V_i\cap V_j$, then $\varsigma_i(v)=\varsigma_j(v)$}  \right\}.
 \]
The local inverses $\varsigma_i$'s glue to define a global inverse $\varsigma: T\to \mathcal{SOL}_G(S)$. By construction, $T$ is $\mathrm{MCG}(S)$-invariant and $\varsigma$ is $\mathrm{MCG}(S)$-equivariant. We are only left to show that $X$ is contained in $\textrm{int}(T)$. The restriction of $\varsigma$ to $\Omega^0_G$ is then a biholomorphism because it inverts a local biholomorphism.

By local finiteness of the cover, for all $p\in X$ we can find a connected neighbourhood $W_p\subset \bigcup_{i\in I} V_i$ of $p$ that intersects only a finite amount of elements $V_i$ in the cover. Up to subtracting from $W_p$ the closure of some of them, we can assume that, for all $i\in I$, $V_i$ intersects $W_p$ if and only if $p\in \overline{V_i}$. We prove that $W_p\subset T$. By construction, if an element $v\in W_p$ is such that $v\in V_i\cap V_j$, then $p\in \overline{V_i}\cap \overline{V_j}$, and necessarily $\varsigma_i(p)=\varsigma_j(p)=(\mathrm s_0\cup \mathrm s\cup\overline{\mathrm s})(p)$. Therefore, as we remarked above, $\varsigma_i$ and $\varsigma_j$ agree on $\overline V_i\cap \overline{V_j}$, hence they agree on $v$. We have therefore shown that $p\in W_p\subset T$, hence $X\subset \textrm{int}(T)$.
\end{proof}

We now prove Theorems B and B'.

\begin{proof}[Proof of {Theorem B'}] 
    The map $\mathcal L_G^{\C}$ is obtained by composing the map $\varsigma$ from Theorem \ref{thm: main inverse} with the holonomy map $\mathrm{hol}: \mathcal{SOL}_G(S)\to \chi^{\mathrm{an}}(\pi_1(S), G)$. The uniqueness follows from the fact that $\Omega'_G$ is connected and from the uniqueness of the analytic continuation. By construction of $\mathcal{SOL}_G(S),$ $\mathcal{L}_G^{\C}$ takes points to holonomies of conformal complex harmonic maps. By Theorem D, it takes $\mathcal M(S,G)\times \mathcal T(\overline S))$ and $\mathcal T(S)\times\mathcal M(\overline S,G)$ to opers, and from Section \ref{sec: complex harmonic maps intro}, on $\mathcal T(S)\times  \mathcal{T}(\overline S)$ it agrees with Bers' map (this can also be seen using opers). We just have to construct $\Omega_G'$ to make the statement work. Let 
    \[
     \mathcal S_L^*=\{p=([c_1, q_1], [\overline {c_2}, \overline{q_2}])\in \Omega_G^0\ |\ d\mathcal L_G^{\C} \text{ is not invertible at $p$} \}.
    \]
   Since, in local coordinates, $d\mathcal L_G^{\C}$ is not invertible if and only if the determinant of the locally defined Jacobian vanishes, $\mathcal S_L^*$ is a complex analytic subset.
    Define 
    \begin{align*}
        (\Omega_G^0)^*&:=\Omega^0_G\setminus \mathcal S_L^*  \\
        \Omega_G'&:= (\Omega^0_G)^* \cup (\mathcal M(S,G)\times \mathcal T(\overline S)) \cup (\mathcal T(S)\times\mathcal M(\overline S,G)).
    \end{align*}
Since $\mathcal{L}_G$ is an immersion and $\mathcal{L}_G^{\C}$ is the holomorphic extension, $d\mathcal{L}_G^{\C}$ is injective on the real locus, i.e., $(\Omega^0_G)^*$ contains the diagonal. As a result, $\mathcal S_L^*$ is either empty or of pure codimension 1 in $\Omega^0_G$ and $(\Omega^0_G)^*$ is connected and dense in $\Omega^0_G$.

Since $\mathcal{L}_G$ is $\mathrm{MCG}(S)$-invariant, the uniqueness of analytic continuation implies that $\mathcal L_G^{\C}$ is $\mathrm{MCG}(S)$-invariant. It further follows that both $(\Omega_G^0)^*$ and $\Omega_G'$ are $\mathrm{MCG}(S)$-invariant. 

We finally prove that $ (\Omega^0_G)^*=\textrm{int}(\Omega'_G)$. Clearly $ (\Omega^0_G)^*\subset \textrm{int}(\Omega'_G)$. By construction, $(\Omega^0_G)^*$ is disjoint from $\mathcal S_L\cup \mathcal S_L^*$, and 
$$\Omega'_G\setminus (\Omega^0_G)^* = (\mathcal S_L\cup\mathcal S_L^*) \cap (\mathcal M(S,G)\times \mathcal T(\overline S) \cup \mathcal T(S)\times\mathcal M(\overline S,G))= \Omega'_G\cap (\mathcal S_L\cup \mathcal S_L^*).$$
Since $\mathcal S_L\cup\mathcal S_L^*$ has pure codimension 1, every neighborhood of every element of $\Omega'_G\setminus (\Omega^0_G)^*$ intersects $\mathcal S_L\cup\mathcal S_L^*\setminus  (\mathcal M(S,G)\times \mathcal T(\overline S) \cup \mathcal T(S)\times\mathcal M(\overline S,G))$, and hence no neighborhood of any element in $\Omega'_G\setminus (\Omega^0_G)$ is contained in $\Omega'_G$. This proves that $\textrm{int}(\Omega'_G)\subset(\Omega^0_G)^*$.
\end{proof}

\begin{proof}[Proof of Theorem B]
The set $\Omega_G$ and the map $\mathcal{B}_G$ are obtained by precomposing $\mathcal L_G^{\C}$ from Theorem B' with the biholomorphism $(\mathcal L_G^{-1}, \overline{\mathcal L_G^{-1}}): \mathrm{Hit}(S,G)\times \mathrm{Hit}(\overline S, G)\to \mathcal M(S,G)\times \mathcal M(\overline S,G)$.
\end{proof}
   As discussed in the introduction, for $G=\mathrm{PSL}(2,\R)^2$ and $\mathrm{PSL}(3,\R)$, we can extend $\Omega_G$ to include all of $\mathrm{Hit}(S,G)\times \mathcal T(\overline S)$ and $\mathcal T(S)\times \mathrm{Hit}(\overline S,G)$. This is because, by linearizing (\ref{eq: the F}), we see that, in these cases, $L_G=2I_2$, and hence by \cite[Theorem E]{ElSa}, $\mathrm{Sing}_{G}=\emptyset$.

Finally, let us address the question: are the representations in the image of $\mathcal{B}_G$ Anosov? For $G=\mathrm{PSL}(2,\R)^2$, we use our description of the image of the marginal locus as opers to show that the holonomies are not always Anosov. We only need to know that for $\mathrm{PSL}(2,\C)$, a representation is Anosov if and only if it is quasi-Fuschian. For other $G$, probing the Anosov question will require examining the image opers in more depth.

   Let $G=\mathrm{PSL}(2,\R)^2$. The data $\varsigma([c,q_1],[\overline{c},0])\in \mathcal{SOL}_G(S)$ corresponds to an equivariant minimal immersion to $\mathbb{G}\times \mathbb{G}$ with holonomy $\rho=(\rho_L,\rho_R)$ and such that the left factor has Hopf differentials $q_1$ and $0$, and the right factor has $-q_1$ and $0$. By Theorem D, recalling the comments proceeding Definition \ref{def: BD oper}, $\rho_L$ and $\rho_R$ are also the holonomies of complex projective structures on $(S,c)$ with Schwarzian derivatives $-\frac{1}{2}q_1$ and $\frac{1}{2}q_1$ respectively. Due to the well-known Nehari bound, once $q_1$ is large enough, neither $\rho_L$ nor $\rho_R$ are quasi-Fuchsian, i.e., Anosov. Since any parabolic subgroup for $\mathrm{PSL}(2,\C)^2$ is just a products of Borels or a product of a Borel with $\mathrm{PSL}(2,\C)$, $\rho$ is not Anosov for any parabolic subgroup. Thus, once we go far enough from the products of quasi-Fuchsian representations, $\mathcal{B}_G$ lands outside of the Anosov locus.

\section{Goldman's symplectic form}\label{sec: Goldman form thma A}
In this section, we prove Theorems A and A'. First we prove the compatibility condition, and then we carry out the signature computation.

\subsection{Goldman's symplectic form}\label{sec: G symplectic form}
Let $\nu$ be a real-valued non-degenerate symmetric bilinear form on the Lie algebra of a real semisimple Lie group $G$ that's invariant under the adjoint action of $G$. Starting from $(G,\nu)$, one can use group cohomology to define the real analytic Goldman symplectic form on $\chi^{\mathrm{an}}(\pi_1(S),G).$

When $G$ is complex, we can consider instead a complex valued non-degenerate bilinear form $\nu$, and Goldman's construction goes through more or less the same and produces a holomorphic symplectic form on $\chi^{\mathrm{an}}(\pi_1(S),G)$ (on the group cohomology level, one takes complex coefficients instead of real coefficients). For clarity, we refer to a holomorphic symplectic form obtained in this fashion as a \textbf{complex Goldman symplectic form}. If $V\subset \chi^{\mathrm{an}}(\pi_1(S),G)$ is admissible, and if $\nu$ is the complex bilinear extension of a real-valued non-degenerate symmetric bilinear form on the Lie algebra of $G$, then the complex Goldman form on $\chi^{\mathrm{an}}(\pi_1(S),G^{\C})$ associated with $\nu$ restricts to the Goldman form on $V$ associated with $G$ and the real-valued form.

When transported to the space of flat $G$-bundles, Goldman's (real or complex) form is the (real or complex) Atiyah-Bott symplectic form \cite{Golform}, and for this reason it is often referred to as the Goldman-Atiyah-Bott symplectic form. In this paper, we work only on the flat connection side, but we retain the name ``Goldman form." 

Finally, let us recall the definitions of Goldman's symplectic forms. We define the complex Goldman symplectic form $\omega_{G}^{\C}$ on $\mathcal{A}(S,G^{\C})$, and then the Goldman symplectic form $\omega_G$ for $G$ is defined similarly. In fact, on $\chi^{\mathrm{an}}(\pi_1(S),G)$, we will be concerned only with admissible subsets, and on such a subset $V$, identified with a subset of $\mathcal{A}(S,G^{\C}),$ we can define $\omega_G$ to be the restriction of $\omega_{G}^{\C}$. We choose $\nu$ to be a product of Killing forms on the simple factors (possibly with different scalings). Since it shouldn't cause much confusion, we ignore the choice of scalings when saying ``the Goldman symplectic form." Given $[(P,A)]\in \mathcal{A}(S,G^{\C})$ represented by $(P,A)\in \mathcal{F}^*(G^{\C})$, the tangent space $T_{[(P,A)]}\mathcal{A}(S,G^{\C})$ identifies with the first cohomology of the complex of $\mathrm{ad}P=P\times_{\mathrm{Ad}} \g^{\C}$-valued differential forms on $S,$ with differential arising from the flat connection $A$ (see \cite{Golform}, \cite{Atiyah1983TheYE}). On pairs $[\dot A_1], [\dot A_2]\in T_{[(P,A)]}\mathcal{A}(S,G^{\C})$ represented by $\mathrm{ad}P$-valued $1$-forms $\dot{A}_1$ and $\dot{A_2}$, the complex Goldman symplectic form $\omega_{G}^{\C}$ is defined by 
\[
\omega_G^{\C} ([\dot A_1], [\dot A_2])= \int_S \nu(\dot A_1,\dot A_2)\ .
\]

\subsection{Proof of compatibility}\label{sec: proof of compatibility}
Throughout this section, let $G$ be the adjoint form of a split real simple Lie group $G$ of rank at least $2,$ or $\mathrm{PSL}(2,\R)^2$. We resume the notations of Sections \ref{sec: bi-Hitchin},  \ref{sec: proof of Theorem C}, and 6. We keep the Chevalley base $\{x_\alpha,x_{-\alpha},h_\beta: \alpha \in \Delta^+,\beta\in \Pi\}$ of Section \ref{sec: excursion}.

For the rest of the paper, we identify tangent vectors in $\chi^{\mathrm{an}}(\pi_1(S),G^{\C})$ with variations in the space of flat $G$-connections $\mathcal A(S,G)$. Using the map $\mathcal{L}_{G}^{\C}:\Omega_G' \to \chi^{\mathrm{an}}(\pi_1(S),G^{\C})$ of Theorem B', we pull back $\omega_{G}^{\C}$ to $\textrm{int}(\Omega'_G)$. The $2$-form $(\mathcal{L}_{G}^{\C})^*\omega_{G}^{\C}$ of Theorem B' is non-degenerate in a neighbourhood of Teichm{\"u}ller space (and in rank $2$, in a neighbourhood of the diagonal), and even if a priori it degenerates somewhere, it still makes sense to say that a submanifold is Lagrangian. The following proposition is the key to the compatibility in Theorems A and A'.
\begin{prop}\label{prop:lagrangians}
      The intersections of $\textrm{int}(\Omega_G')$ with submanifolds of the form $\mathcal M(S,G)\times \{[\overline{c_2}, \overline{q_2}]\}$ and $\{[c_1, q_1]\}\times \mathcal M(\overline S,G)$ are $(\mathcal{L}_G^{\C})^*\omega_G^{\C}$-Lagrangian.
\end{prop}
\begin{proof}
We do the case where $G$ is simple in full detail, and then we explain the minor modifications for $G=\mathrm{PSL}(2,\R)^2.$ We first prove the statement for a submanifold of the form $\textrm{int}(\Omega_G')\cap (\{[c_1, q_1]\}\times \mathcal M(\overline S,G))$.  
Let $\sigma=(c_1,q_1,\overline{c_2},\overline{q_2},u)\in \mathrm{SOL}_G^*(S)$ be such that $[\sigma]\in \mathcal{SOL}_G(S)$ lies in the image of $\varsigma$, defined in Theorem \ref{thm: main inverse}, and consider a $C^1$ path $\sigma_t=(c_1, q_1, \overline{c_2}^t, \overline{q_2}^t, u_t)\in \mathrm{SOL}_G^*(S)$ passing by $\sigma$ at $t=0$, which defines a path of flat principal $G^{\C}$-connections $A^t=A_{\sigma_t}$ on $P_{G^{\C}}$.
In the usual notations, let $\mathcal{I}_t$ be the Lie algebra bundle isomorphism solving (\ref{eq: bi-Hitchin flatness}) for $\sigma_t$, and we write. 
\begin{equation}
\label{eq: variation in Lagrangian marginal}
    A^t= {A_0^t} + \phi_1 +\overline{\phi_2}^{t},
\end{equation}
where ${A_0^t}=A^{\mathcal{I}_t}$ and $\overline{\phi_2}^{t}=-\overline{\phi_2}^{\mathcal{I}_t}$.
Taking the derivative at $t=0$,
$$\dot{A}=\dot{A_0}+\dot {\overline{\phi_2}},$$ for some $\mathrm{ad}P_{G^{\C}}$-valued $1$-forms $\dot{A}_0$ and $\dot {\overline{\phi_2}}.$ 
To see the variations more explicitly, we work in a trivialization of $P_{G^{\C}}$ that is holomorphic for $P_{G^{\C}}\to (S,c_1).$ By part (3) of Proposition \ref{prop: h valued}, in such a trivialization, each $A_0^t$ lies in $\Omega_{c_1}^{1,0}(S,S\times\mathfrak{h}^{\C})$. Hence, the same holds for $\dot A_0$. For the variation of the Higgs field, first write $U_t=(U_\alpha^t)_{\alpha\in \Pi}$ with $U_\alpha^t = U_\alpha+t\dot{U}_\alpha + o(t),$ and similarly for the $U_{-\delta}$ term. As in Section \ref{sec: proof of Theorem C}, but in our current notation, we can write
$$\overline{\phi_2}^{t}=\sum_{\alpha\in \Pi} r_\alpha^{1/2}\lambda U_\alpha^t (x_\alpha\otimes dz) \otimes d\overline{w}+\lambda^{-d+1} U_{-\delta}^t\overline{q_2}(w)(x_{-\delta}\otimes dz^{-d+1} )\otimes d\overline{w}.$$
We then differentiate to obtain
    $$\dot {\overline{\phi_2}}=\sum_{\alpha\in \Pi} r_\alpha^{1/2}\lambda \dot{U}_\alpha (x_\alpha\otimes dz) \otimes d\overline{w}+\lambda^{-d+1} \dot{U}_{-\delta}\overline{q_2}(w)(x_{-\delta}\otimes dz^{-d+1} )\otimes d\overline{w},$$
Similarly, let $B^t=B_0^{t}+\phi + \overline{\phi'_2}^{t}$ be another path of flat connections corresponding to a path in our submanifold as above, with variation
 $$\dot{B}=\dot B_0+\dot{\overline{\phi'_2}}.$$ In our current notation, the complex Goldman symplectic form on $\dot{A}$ and $\dot{B}$ is defined by the integration 
 \begin{equation}\label{eq: goldman integral}
     \omega_{G}^{\C}(\dot{A},\dot{B})=\int_S \nu(\dot{A},\dot{B}).
 \end{equation}
To show that our submanifold is Lagrangian for $(\mathcal{L}_G^{\C})^*\omega_{G}^{\C}$, we need to prove that the integral (\ref{eq: goldman integral}) is always zero. We establish the slightly stronger result that $\nu(\dot{A},\dot{B})=0$. Computing the Killing form on these variations is simple: using that root spaces $\g_\alpha$ and $\g_\beta$ are Killing orthogonal unless $\alpha+\beta=0$, we obtain $$\nu(\dot A_0,\dot{\overline{\phi_2'}})=\nu(\dot B_0,\dot{\overline{\phi_2}})= \nu(\dot{\overline{\phi_2}},\dot{\overline{\phi_2'}})=0,$$ and hence that 
\begin{equation}\label{eq: pairing simplified}
\nu(\dot{A},\dot{B})=\nu(\dot{A}_0,\dot{B}_0).
\end{equation}
Since, in our trivializations of $P_{G^{\C}}$, $\dot A_0$ and $\dot B_0$ both live in $\Omega_{c_1}^{1,0}(S,S\times \mathfrak{h}^{\C})$, the pairing on the right hand side of (\ref{eq: pairing simplified}) is zero. We deduce that $\nu(\dot{A},\dot{B})=0$, and the result is established.

Having proved that every submanifold of the form $\textrm{int}(\Omega_G')\cap (\{[c_1, q_1]\}\times \mathcal M(\overline S,G))$ is Lagrangian, to prove the analogous statement for submanifolds of the form $\textrm{int}(\Omega_G')\cap (\mathcal M(S,G)\times \{[\overline{c_2}, \overline{q_2}]\})$, we consider the self-map $\mathrm{Conj}$ of $\mathcal M(S,G)\times \mathcal M(\overline S,G)$ given by $\mathrm{Conj}([c_1,q_1],[\overline{c_2},\overline{q_2}])=(([c_2,q_2],[\overline{c_1},\overline{q_1}])$. We didn't construct $\textrm{int}(\Omega_G')$ to be invariant under $\mathrm{Conj}$, but it wouldn't have been difficult to do so. Let $\Omega_G''\subset \Omega_G'$ be a $\mathrm{Conj}$-invariant open subset (which exists, since the diagonal is invariant). Set $\tau:G^{\C}\to G^{\C}$ to be the anti-holomorphic involution defining the real form $G$, which induces an involution $\tau^{\chi}$ of the $G^{\C}$-character variety. Then, on $\Omega_G''$, $ \mathcal{L}_{G}^{\C}=\tau^\chi\circ \mathcal{L}_{G}^{\C}\circ \mathrm{Conj}.$ Indeed, both maps are holomorphic and agree with $\mathcal{L}_G$ on the diagonal. It follows that $(\mathcal{L}_{G}^{\C})^*\omega_{G}^{\C}$ is Lagrangian on $\Omega_G''\cap (\mathcal M(S,G)\times \{[\overline{c_2}, \overline{q_2}]\})$. By analytic continuation, the Lagrangian property persists when we replace $\Omega_G''$ by $\textrm{int}(\Omega_G')$.

    Finally, we address the case $G=\mathrm{PSL}(2,\R)^2$. Now, the connections and Higgs fields split as sums of connections and Higgs fields respectively for $\mathrm{PSL}(2,\R)$. The above formulas for variations are identical, except every term splits as a sum. These sum-splittings are orthogonal for the Goldman pairing, and hence the arguments above, done on the two component separately, go through with zero change. 
\end{proof}

The fact that $\omega_G(\mathcal J_G\cdot, \mathcal J_G\cdot)=\omega_G(\cdot,\cdot)$ is now an immediate consequence of Proposition \ref{prop:lagrangians} and of general facts about complex geometry, which we explain now. Let $M$ be a complex manifold with almost complex structure $\mathcal J_0$ and let $\omega_0$ be a real analytic symplectic form on $M,$ which we don't assume to be compatible with $\mathcal J_0$. The complexification of $M$ is the complex manifold $M\times \overline{M}$ whose almost complex structure is $\mathcal J_0\times (-\mathcal J_0).$ We consider the totally real embedding $\delta: M\to M\times \overline M$ defined by $\delta(p)=(p,p).$ By analyticity, in a neighbourhood $U$ of $\delta(M)$ inside $M\times \overline{M}, $ $\delta_*\omega_0$ extends uniquely to a holomorphic symplectic form $\omega_0^{\C}$. 
\begin{prop}\label{prop:compatible}
    The form $\omega_0$ and the complex structure $\mathcal J_0$ are compatible if and only if, for all $p\in M$, the submanifolds $V\cap (M\times \{p\})$ and $V\cap (\{p\}\times \overline{M})$ of $V$ are $\omega_0^{\C}$-Lagrangian.
\end{prop}
The ``only if" direction is contained in Loustau-Sanders \cite[Theorem 3.8]{LouS} (assuming $(M,\mathcal J_0,\omega_0)$ is K{\"a}hler, which is unnecessary), and we only use the ``if" direction for Theorem \ref{thm: Hit is Kähler}. For completeness, we include the full statement and proof.
\begin{proof}
Given local holomorphic coordinates $(z_1,\dots, z_n)$ for $M$ {around $p$},  $(z_1,\dots, z_n,\overline{z}_1,\dots, \overline{z}_n)$ are local holomorphic coordinates for $M\times \overline{M}$ {around $(p,p)$}.
In such coordinates, we can write
$$\omega_0^{\C}=-\frac{1}{2i}(\varphi_{jk}dz_j\wedge dz_k +\varphi_{j\overline{k}}dz_j\wedge d\overline{z}_k+\varphi_{\overline{j}\overline{k}}d\overline{z}_j\wedge d\overline{z}_k)$$
for some holomorphic coefficient functions $\varphi_{jk},\varphi_{j\overline k}, \varphi_{\overline{jk}}$.

 It is easily checked that the submanifolds  $V\cap (M\times \{p\})$ and $V\cap (\{p\}\times \overline{M})$ are $\omega_0^{\C}$-Lagrangian if and only if, for every $j,k$, $\varphi_{jk}=\varphi_{\overline{jk}}=0,$ which, by analytic continuation, occurs if and only if $\varphi_{jk}$ and $\varphi_{\overline{jk}}$ vanish on $\delta(M)$, i.e., $\omega_0=\delta^*\omega^{\C}_0$ is a $(1,1)$-form for $\mathcal J_0$.

 To deduce the Proposition, we take note of the basic fact that $\omega_0$ is a $(1,1)$-form for $\mathcal J_0$ if and only if $\omega_0$ and $\mathcal J_0$ are compatible.
To see this, write $\omega_0=\upbeta_{jk} dz_j \wedge d z_k+ \upbeta_{j\overline  k}dz_j\wedge d\overline z_k+ \upbeta_{\overline{jk}}d\overline z_j\wedge d\overline z_k$, where $\omega_0$ being a real form implies $\overline{\upbeta_{jk}}=\upbeta_{\overline{jk}}$. Computing the 2-form $\omega_0(\mathcal J_0\cdot, \mathcal J_0\cdot)-\omega_0(\cdot, \cdot)$ in the real basis $\{\partial_{z_j}+ \partial_{\overline z_j}, i\partial_{z_k}-i\partial_{\overline z_k} \}$, we conclude that $\omega_0$ and $\mathcal J_0$ are compatible if and only if $\overline{\upbeta_{jk}}=-{\upbeta_{\overline{jk}}}$ for all $j,k$, hence if and only if $\upbeta_{jk}=\upbeta_{\overline{jk}}=0$. 
\end{proof}

Returning to our setting, Propositions \ref{prop:lagrangians} and \ref{prop:compatible} immediately imply the compatibility condition for Theorem A'. Indeed, consider the diagonal immersion $\delta: \mathcal M(S,G) \to \textrm{int}(\Omega_G')\subset \mathcal M(S,G)\times \mathcal M(\overline S,G)$. Since $\mathcal{L}_{G}^{\C}$ is holomorphic, the holomorphic extension of the real analytic symplectic form $\delta_*\omega_G$ to $\textrm{int}(\Omega_G')$ is $(\mathcal{L}_{G}^{\C})^*\omega_G^{\C}$.
Since $\mathcal L_G^{\C}\circ \delta=\mathcal L_G$, by Propositions \ref{prop:lagrangians} and \ref{prop:compatible}, the complex structure on $\mathcal M(S,G)$ is compatible with $(\mathcal L_G^{\C})^*\omega_G^{\C}= \delta^* (\mathcal L_{G}^{\C})^* \omega_G^{\C}$. 


\subsection{Variations on the Fuchsian locus}\label{sec: variations on fuchsian} 
We continue in the setting of the subsection above. Having proved that $\mathcal{L}_G^*\omega_G$ and $\mathcal{J}_G'$ are compatible (recall from Section \ref{sec: A' and B'} that $\mathcal{J}_G'$ is the almost complex structure on $\mathcal{M}(S,G)$), toward Theorem A', we compute the signature of the induced pseudo-Riemannian metric $\mathcal{L}_G^*\omega_G(\cdot, \mathcal{J}_G'\cdot )$ on the subset $\mathcal{M}_0(S,G)$ (which, in rank $2$, identifies with $\mathrm{Hit}(S,G)$). Since the signature is locally constant, it suffices to compute it at a single point; we will do so on the Fuchsian locus, i.e., for points of the form $[c,0].$ We also restrict here to the simple groups, since the variations for $\mathrm{PSL}(2,\R)^2$ are then obtained by taking sums.

To do the computation, in the proof below, we first compute $(\mathcal{L}_G^{\C})^*\omega_G^{\C}$ on the Fuchsian locus, by which we mean on variations of points of the form $([c,0],[\overline{c},0]).$ Toward this, we first compute variations of flat connections coming from applying $\mathcal{L}_G^{\C}$. In the decomposition $$T_{([c,0],[\overline c,0])} \mathcal M(S,G)\times \mathcal M(\overline S,G)= T_{[c]}\mathcal T(S)\oplus T_{[\overline c]} \mathcal T(\overline S) \oplus H^0(S,\mathcal K^r)\oplus H^0(S,\overline{\mathcal K}^r),$$ $T_{([c,0],[\overline c,0])} \mathcal M(S,G)\times \mathcal M(\overline S,G)$ is spanned by elements of the form
$$[\dot c_1]= ([\dot c_1, 0], [0, 0]), \hspace{1mm}
[\dot {\overline{ c_2}}]= ( [0,0], [ \dot{\overline{ c_2}}, 0]), \hspace{1mm}
 \dot q_1= ([0,\dot q_1], [0,0]), \hspace{1mm}
  \dot{\overline{ q_2}}=([0,0],[0,\dot{\overline{ q_2}}]),$$
and we study cases separately. In the following, set $h=\vb*h_{(c,\overline c)}=\lambda_0 dz \cdot d\overline z$. To use as a reference for the computations below, we recall the expression for the connection for a general point $(c_1,q_1, \overline {c_2},\overline{q_2})$, in a suitable trivialization, as $$A^{\mathcal{I}}+(\widetilde{e}+q_1x_\delta) + (\sum_{\alpha \in \Pi} \lambda U_\alpha r_\alpha^{1/2}(x_\alpha \otimes dz) \otimes d\overline{w}+ \lambda^{-d+1} U_{-\delta}\overline{q_2}(w)(x_{-\delta}\otimes dz^{-d+1} )\otimes d\overline{w}),$$
Computing variations for $\dot{q}_1$ and $\dot{\overline{q_2}}$ is simple. In the direction $\dot{q}_1,$ we compute immediately 
\begin{equation}\label{eq: q_1 variation}
    d\mathcal{L}_G^{\C}(\dot{q_1})=\dot{q}_1x_{\delta},
\end{equation}
and, in the direction $\dot{\overline{q_2}},$ using that the isomorphism $\mathcal{I}$ does not depend on $\dot{\overline{q_2}}$ when $q_1=0$, we have 
\begin{equation}\label{eq: q_2 variation}
    d\mathcal{L}_G^{\C}(\dot{\overline{q_2}})=\lambda_0^{-d+1}U_{-\delta}\dot{\overline{q_2}}x_{-\delta},
\end{equation}
where $U_{-\delta}$ is a constant and $\lambda^{-d+1}\dot{\overline{q_2}}x_{-\delta}$ is interpreted as a section of $\g_{-\delta}\otimes \overline{\mathcal{K}}^{-d+1}$ tensored with a $(0,1)$-form. Explicitly, in the local coordinate $z,$ $$\lambda_0^{-d+1}\dot{\overline{q_2}}x_{-\delta}=\lambda^{-d+1}dz^{d-1}d\overline{z}^{d-1}U_{-\delta}\dot{\overline{q_2}}(z)(x_{-\delta}\otimes dz^{-d+1})\otimes d\overline{z}.$$ 
For a variation $[\dot{\overline{c_2}}]$ tangent to the zero section, consider a tangential path of complex structures $(c,0, \overline{c_2}^t,0)$ and, as in \eqref{eq: variation in Lagrangian marginal}, denote the corresponding path of flat connections by $A^t=A_0^t+\phi_1+\overline{\phi_2}^t$. As we observed in the proof of Proposition \ref{prop:lagrangians}, $\dot A_0$ is a $1$-form valued in $S\times \mathfrak{h}^{\C}$ and $\dot{\overline{\phi_2}}$ is valued in $\g_1\otimes \mathcal{K}_{c}$. So $d\mathcal{L}_G^{\C}([\dot{\overline{c_2}}])$ is a $1$-form valued in $(S\times \mathfrak{h}^{\C})\oplus (\g_1\otimes \mathcal{K}_c),$ and while we could compute more explicitly, this is in fact all the information that we'll need. The computation for a variation of the form $[\dot{c_1}]$ is a bit more involved, since the usual trivializations change along a tangential path $t\mapsto (c_1^t,0, \overline{c},0),$ but still doable. We'll see below that we don't need to carry out this computation.
\subsection{Computation of the signature}\label{sec: signature}
Continuing with the setup from the previous two subsections, we complete the proofs of Theorems A and A'.
\begin{proof}[Proof of Theorem A']
In order to relax the notation, in this proof we use $\omega_G$ and $\omega_G^{\C}$ to denote $\mathcal{L}_G^*\omega_G$ and $(\mathcal{L}_G^{\C})^*\omega_G^C$ respectively. As stated above, with compatibility established, it is only required to compute the signature on the Fuchsian locus, i.e., for points of the form $[c,0].$ Before getting to that, we carry out computations involving the complex Goldman symplectic form. Specifically, in the notation of the subsection above, we compute the pairings $\omega_G^{\C}([\dot{c_1}],[\dot{\overline{c_2}}]),$ $\omega_G^{\C}(\dot{q}_1,[\dot{\overline{c_2}}])$, $\omega_G^{\C}(\dot{\overline{q_2}},[\dot{c_1}])$, and $\omega_G^{\C}(\dot{q_1},\dot{\overline{q_2}}).$ As we already mentioned, the form $\omega_G$ restricts to a multiple of the Weil-Petersson symplectic form on $\mathcal{T}(S)$. Thus, the pairing $\omega_G^{\C}([\dot{c_1}],[\dot{\overline{c_2}}])$ is (up to a constant) the pairing of $[\dot{c_1}]$ and $[\dot{\overline{c_2}}]$ with respect to the complexification of the Weil-Petersson symplectic form. This latter pairing is computed in \cite[Section 5]{ElE}, and we won't explicitly need it. By Proposition \ref{prop:lagrangians}, we already know the remaining pairings vanish.

For $\omega_G^{\C}(\dot{q}_1,[\dot{\overline{c_2}}])$, in the usual decomposition of $\mathrm{ad}P_{G^{\C}},$ $d\mathcal{L}_G^{\C}(\dot{q}_1)$ is valued in $\g_{\delta}\otimes \mathcal{K}_c^d,$ and $d\mathcal{L}_G^{\C}(\dot{[\overline{c_2}]})$ is a sum of two forms valued $S\times \mathfrak{h}^{\C}$ and in $\g_{1}\otimes \mathcal{K}_c$ respectively. Since root spaces $\g_\alpha$ and $\g_\beta$ are orthogonal unless $\alpha+\beta=0,$ $\omega_G^{\C}(\dot{q}_1,[\dot{\overline{c_2}}])=0$. Using the map $\mathrm{Conj}$ from the proof of Proposition \ref{prop:lagrangians}, we obtain $\omega_G^{\C}(\dot{\overline{q_2}},[\dot{c_1}])= \overline{\omega_G^{\C}(\dot{q_2},[\dot{\overline{c_1}}])}=0.$ 
For $\omega_G^{\C}(\dot{q}_1,\dot{\overline{q_2}}),$ using the formulas \eqref{eq: q_1 variation} and \eqref{eq: q_2 variation},
\begin{equation}\label{eq: q1 q2 pairing}
    \omega_G^{\C}(\dot{q_1},\dot{\overline{q_2}}) = U_{-\delta}\nu(x_\delta,x_{-\delta})\int_S \frac{\dot{q_1}(z)\dot{\overline{q_2}}(z)}{\lambda^{d-1}(z)} dz \wedge d\overline{z}.
\end{equation}
We highlight that, since $U_{-\delta}$ is coming from the solution from Proposition \ref{prop: existence on marginal locus}, it is positive.

Now we get to the main signature computation for $\omega_G(\cdot,\mathcal{J}_G'\cdot)$. We show that the zero section, i.e., Teichm{\"u}ller space, is orthogonal to the fiber, and then we are free to compute the signatures on these two pieces.  
This orthogonality is easy: considering two variations of the $([\dot{c}],0)$ and $(0,\dot{q})$ based at the point $([c],0),$
$$\omega_G([\dot{c}],\mathcal{J}_G'\dot{q}) = \omega_G^{\C}([\dot{c}]+[\dot{\overline{c}}],\dot{q_1}-i\dot{\overline{q}})= 0,$$ since the middle expression breaks into a sum of four pairings that all vanish.
Now, to conclude, we show that $\omega_G(\cdot,\mathcal{J}_G'\cdot)$ is positive definite on the zero section and negative definite on the transverse fiber. For the zero section, although it is not hard to compute explicitly, to justify that $\omega_G(\cdot, \mathcal{J}_G'\cdot)$ is positive definite, we just use that, on this piece, $\omega(\cdot, \mathcal{J}_G'\cdot)$ is a positive scalar multiple of the Weil-Petersson metric. Finally, to see that $\omega_G$ is negative definite on the fibers, we take a variation $(0,\dot{q})$ based at $([c],0)$ and simply compute
$$\omega_G(\dot{q},\mathcal{J}_G'(\dot{q})) = \omega_G^{\C}(\dot{q}+\dot{\overline{q}}, i\dot{q}-i\dot{\overline{q}}) = -2i\omega_G^{\C}(\dot{q},\dot{\overline{q}}).$$ Substituting the expression (\ref{eq: q1 q2 pairing}) yields
$$\omega_G(\dot{q},\mathcal{J}_G'(\dot{q}))=-2i U_{-\delta}\nu(x_\delta,x_{-\delta})\int_S \frac{\dot{q}(z)\dot{\overline{q}}(z)}{\lambda^{d-1}(z)} dz \wedge d\overline{z}=-4 U_{-\delta}\nu(x_\delta,x_{-\delta})\int_S \frac{\dot{q}(z)\dot{\overline{q}}(z)}{\lambda^{d-1}(z)} dx \wedge dy\leq 0,$$ which finishes the computation. By the discussion above, we are done.
\end{proof}

\begin{proof}[Proof of Theorem A]
   We apply Theorem A' to $\mathcal{M}(S,G)=\mathcal{M}_0(S,G)$ and push forward the complex structure and the symplectic form via Labourie's parametrization $\mathcal{L}_G.$
\end{proof}
To conclude the paper, we show that in rank $2$, $\mathcal T(S)$ is a totally geodesic submanifold of $\mathrm{Hit}(S,G)$. It follows by applying the theorem below for general $G$ and pushing forward through Labourie's parametrization.
\begin{thm}\label{rmk: Teich totally geodesic}
   The zero section is a totally geodesic submanifold of $(\mathcal{M}_0(S,G),\mathcal{L}_G^*\omega_G,\mathcal{J}_G').$
\end{thm}
\begin{proof}
    The involution on  $\mathcal M_0(S,G)$, $[c,q]\mapsto [c, -q]$, clearly preserves the complex structure. We claim that it preserves $\mathcal{L}_G^*\omega_G$. Consider two paths $t\mapsto [c,q_1^t],[c,q_2^t]$ in $\mathcal{M}_0(S,G)$, starting at $[c,q_1^0]=[c,q_1]$ and $[c,q_2^0]=[c,q_2].$ In the usual notation, they give paths of flat connections $$A_i^t =A_{i,0}^t+\phi_i^t-\hat{\sigma}_t(\phi_i^t)= A_i + t\dot{A}_i + o(t), \hspace{1mm} i=1,2,$$ where $\hat{\sigma}_t$ is the involution solving Hitchin's self-duality equations. It suffices to show the invariance of $\nu(\dot{A}_1,\dot{A}_2)$. For convenience, write $\hat{\sigma}(\dot{\phi}_i)=\frac{d}{dt}|_{t=0}\hat{\sigma}_t(\phi_i^t).$ As in the proof of Proposition \ref{prop:lagrangians}, we can use Killing orthogonality and type decompositions to see that $$\nu(\dot{A}_1,\dot{A}_2)=-\nu(\dot{\phi}_1,\hat{\sigma}(\dot{\phi}_2))-\nu(\hat{\sigma}(\dot{\phi}_1),\dot{\phi}_2).$$ We calculate directly, using our standard local expressions, that $$\nu(\dot{\phi}_1,\hat{\sigma}(\dot{\phi}_2))=\lambda^{d+1}\nu(x_{\delta},x_{-\delta})(\dot{U}_{-\delta}\dot{q}_1\overline{q_2}+U_{-\delta}\dot{q}_1\dot{\overline{q_2}}),$$ and it is clear that this expression is invariant under negating $q_1$, $\overline{q_2}$, and their tangent vectors (in fact invariant under the natural $S^1$-action). The expression for $\nu(\hat{\sigma}(\dot{\phi}_1),\dot{\phi}_2)$ is analogous, hence the claim follows. We conclude that the zero section is invariant under a pseudo-K{\"a}hler isometry, and hence it is totally geodesic.
\end{proof}


\printbibliography

\end{document}